\documentclass[reqno]{amsart}
\usepackage{amsaddr}
\usepackage[english]{babel}
\usepackage[utf8]{inputenc} 

\usepackage{amsfonts,amssymb,amsmath,amsthm}
\numberwithin{equation}{section}
\usepackage{dsfont}
\usepackage{multirow}
\usepackage{url}
\usepackage{xcolor} 
\usepackage{graphics}
\usepackage{booktabs}
\usepackage{epsfig}
\usepackage{epstopdf}
\usepackage{psfrag}                
\usepackage{mathrsfs}
\usepackage{mathtools}
\mathtoolsset{showonlyrefs}    
\usepackage{subfigure}
\usepackage{nicefrac}
\usepackage{bm}
\usepackage{diagbox}
\usepackage{caption}
\usepackage{algorithm}
\usepackage{algpseudocode}
\usepackage{paralist}
\usepackage[shortlabels,inline]{enumitem}

\usepackage[pdftex,
	pdftitle={Optimized multilevel Monte Carlo methods\\ in Banach spaces},
	pdfauthor={K.~Kirchner, F.~Nobile, Ch.~Schwab, and T.~Vanzan},
	bookmarksopen,
	colorlinks,
	linkcolor=black,
	urlcolor=black,
	citecolor=black
	]{hyperref}

\newcommand{\w}{\omega}
\newcommand{\dprime}{\prime\prime}

\newcommand{\pprime}{p^\prime}
\newcommand{\refe}{\text{ref}}
\newcommand{\tensorW}{\otimes_\varepsilon^2 W^{1,p}_{\left\{0\right\}}(I)}
\newcommand{\Wdual}{W^{-1,p'\!}_{\left\{0\right\}}(I)}
\newcommand{\lb}{\bm{l}}
\newcommand{\cb}{\bm{c}}
\newcommand{\ub}{\bm{u}}
\newcommand{\vb}{\bm{v}}
\newcommand{\hb}{\bm{h}}
\newcommand{\eb}{\bm{e}}
\newcommand{\cIt}{\widetilde{\mathcal{I}}}
\newcommand{\cTt}{\widetilde{\mathcal{T}}}
\newcommand{\triplenorm}[1]{\left\lvert\!\left\lvert\!\left\lvert #1 \right\rvert\!\right\rvert\!\right\rvert}
\newcommand{\Biggtriplenorm}[1]{\Biggl\lvert\!\Biggl\lvert\!\Biggl\lvert #1 \Biggr\rvert\!\Biggr\rvert\!\Biggr\rvert}

\newcommand{\bbE}{\mathbb{E}}
\newcommand{\bbN}{\mathbb{N}}
\newcommand{\bbP}{\mathbb{P}}
\newcommand{\bbR}{\mathbb{R}}
\newcommand{\bbZ}{\mathbb{Z}}

\newcommand{\bbM}{\mathbb{M}}
\newcommand{\E}{\mathbb{E}}
 
\newcommand{\cA}{\mathcal{A}}
\newcommand{\cB}{\mathcal{B}}
\newcommand{\cC}{\mathcal{C}}

\newcommand{\cI}{\mathcal{I}}

\newcommand{\cS}{\mathcal{S}}
\newcommand{\cT}{\mathcal{T}}

\newcommand{\Omegat}{\widetilde{\Omega}} 
\newcommand{\omegat}{\widetilde{\omega}} 
\newcommand{\cAt}{\widetilde{\cA}} 
\newcommand{\bbPt}{\widetilde{\bbP}} 
\newcommand{\bbEt}{\widetilde{\bbE}} 
\newcommand{\Nht}{\widetilde{N}_h}
\newcommand{\wOmega}{\widetilde{\Omega}}
\newcommand{\wA}{\widetilde{\mathcal{A}}}
\newcommand{\wP}{\widetilde{\mathbb{P}}}
\newcommand{\weta}{\widetilde{\eta}}

\newcommand{\norm}[2]{     \| #1       \|_{ #2 }}

\newcommand{\rd}{\mathop{}\!\mathrm{d}}
\newcommand{\from}{\colon} 
 
\newcommand{\indk}{\kappa} 

\newtheorem{lemma}{Lemma}[section]
\newtheorem{proposition}[lemma]{Proposition}
\newtheorem{theorem}[lemma]{Theorem}
\newtheorem{corollary}[lemma]{Corollary} 

\theoremstyle{remark}
\newtheorem{remark}[lemma]{Remark}

\theoremstyle{definition}
\newtheorem{definition}[lemma]{Definition}
\newtheorem{assumption}[lemma]{Assumption}
\newtheorem{example}[lemma]{Example}

\begin{document}

\author[K.~Kirchner, F.~Nobile, Ch.~Schwab, and T.~Vanzan]{Kristin Kirchner$^{{\lowercase\mathrm{a}},{\lowercase\mathrm{b}}}$,
		Fabio Nobile$^{{\lowercase\mathrm{c}}}$,
		Christoph Schwab$^{{\lowercase\mathrm{d}}}$
                \\
                \and 
		Tommaso Vanzan$^{{\lowercase\mathrm{e}}}$}

\address[Kristin Kirchner]{
	$\,^{\lowercase\mathrm{a}}$Department of Mathematics, 
	KTH Royal Institute of Technology\\
	Lindstedtsv\"agen 25, 
	114$\,$28 Stockholm, Sweden} 

\address[Kristin Kirchner]{
	$\,^{\lowercase\mathrm{b}}$Delft Institute of Applied Mathematics, 
	Delft University of Technology\\
	P.O.~Box 5031,  2600$\,$GA Delft, 
	The Netherlands}

\address[Fabio Nobile]{
	$\,^{\lowercase\mathrm{c}}$Institute of Mathematics, 
	EPF Lausanne \\
	Route Cantonale, 1015 Lausanne,  
	Switzerland}

\address[Christoph Schwab]{
	$\,^{\lowercase\mathrm{d}}$Seminar for Applied Mathematics, 
	ETH Z\"urich \\
	R\"amistrasse 101, 8092 Z\"urich,  
	Switzerland}
	
\address[Tommaso Vanzan]{
	$\,^{\lowercase\mathrm{e}}$Dipartimento di Scienze Matematiche, 
	Politecnico di Torino \\
  	Corso Duca degli Abruzzi 24, 10129 Torino,  
	Italy}


\title[Optimized multilevel Monte Carlo methods in Banach spaces]{
	Optimized multilevel Monte Carlo methods\\ in Banach spaces}




\date{\today}


\begin{abstract}
We present a comprehensive theoretical and numerical analysis of 
Monte Carlo methods for the estimation of statistical moments of 
random variables $X\colon\Omega\rightarrow E$ 
taking values in a Banach space $E$. For practical computation,
we consider finite-dimensional approximation subspaces ${(E_\ell)_{\ell\in\bbN}\subset E}$
of increasing dimension.
We  develop a refined error analysis that explicitly accounts for  
a dependence of the Rademacher type 
constants on the dimension of~$E_\ell$, leading to 
novel complexity results for 
single- and multilevel Monte Carlo methods 
to estimate the mean and injective moments 
of arbitrary order, which are, in certain 
cases, sharper than those derived in~\cite{KKChS2024}.
Moreover, we show that, in favorable cases, 
the resulting error-vs.-work bounds are independent of the Rademacher type of $E$. 

We then focus on $L^p(S)$-valued 
random variables for a $\sigma$-finite 
measure space $(S,\cS,\mu)$ 
satisfying certain approximation properties,
and 
prove that for a random variable
$X\in L^q(\Omega;L^p(S))\cap L^p(S;L^q(\Omega))$, 
with $q\in (1,\infty)$ and $p\in [1,\infty)$, 
the $L^q$-convergence rate of a multilevel Monte Carlo estimator 
is determined exclusively 
by the integrability parameter $\min\{q,2\}$, 
with no dependence on the 
Rademacher type $\min\{p,2\}$ of $L^p(S)$. 
We further investigate the impact of measuring the 
(multilevel) Monte Carlo error in the $L^q(\Omega;L^p(S))$-norm 
while $X$ possesses additional regularity, 
$X\in L^{\widetilde{q}}(\Omega;L^p(S))\cap L^p(S;L^{\widetilde{q}}(\Omega))$ 
with $\widetilde{q}\in [q,\infty)$.  
This analysis reveals an interplay between the sampling error
and the strong approximation error, 
and leads to optimized error–vs.-work bounds
for both single- and multilevel Monte Carlo methods.

Numerical experiments confirm the sharpness 
of the analyses presented in estimating 
both first and second moments.
\end{abstract}

\maketitle
 
\section{Introduction} 
Estimating statistical moments of random variables taking values 
in infinite-dimensional vector spaces is a fundamental challenge in the 
field of uncertainty quantification. 
This setting naturally arises in mathematical models involving 
random or stochastic differential equations, 
where the solution itself is not a finite-dimensional vector, 
but instead an element of an appropriate function space.
While the first moment, that is the mean, is often 
itself the quantity of interest, 
higher-order moments provide additional information 
about the probability distribution of the random variable, 
and can be relevant in statistical tests \cite{mardia1970measures}.

To estimate moments, a widely used approach is the 
Monte Carlo method. 
It is well-known that for a random variable $X\in L^2(\Omega;H)$, where $H$ is a Hilbert space and 
$(\Omega,\cA,\bbP)$ denotes a complete probability space, 
the Monte Carlo method to approximate the mean converges 
with the rate $\nicefrac{1}{2}$ in the number of samples. 
This result extends to the estimation 
of statistical moments of any order $k\in \bbN$, 
viewed as elements of the Hilbert tensor product space $H^{(k)}\!$, 
provided that ${X\in L^{2k}(\Omega; H)}$.
However, this picture drastically changes 
when the random variable has i) reduced integrability 
or ii) takes values in a Banach space $E$. 
The convergence rate indeed depends on both the integrability 
of the random variable and on 
geometric properties of the Banach space $E$. 
In rigorous terms, if $E$ is of \emph{Rademacher type} $p\in [1,2]$ (see \cite[Chapter 7]{AnalysisInBanachSpacesII2017})  
and $X\in L^q(\Omega;E)$, with $q\in [1,\infty)$, 
then the Monte Carlo method to approximate 
the mean converges in $L^q$ at the rate $1-\frac{1}{\min\left\{p,q\right\}}$, 
see e.g.\ \cite[Proposition 9.11]{LedouxTalagrand2011} 
and \cite[Corollary~3.15]{KKChS2024}. 
The extension of this result to arbitrary $k$th moments 
is not straightforward since, 
in contrast to the Hilbertian setting, 
there is no canonical choice for the norm 
on the algebraic tensor product space $\otimes^k E$. 
Provided that $X\in L^{kq}(\Omega;E)$, Kirchner and Schwab~\cite{KKChS2024} 
have recently proved the Monte Carlo convergence rate 
$1-\frac{1}{\min\left\{p,q\right\}}$ 
in the \emph{injective tensor norm}, 
which is the weakest crossnorm, 
and showed that the Monte Carlo method 
does in general not converge 
in the \emph{projective tensor norm}, 
which is the strongest crossnorm, 
see \cite[Chapter 6]{ryan2002introduction}. 
In practice, the convergence rate 
$1-\frac{1}{\min\left\{p,q\right\}}$ can be discouraging, 
since to achieve an error of order~$\epsilon$  
it suggests to use asymptotically 
$\epsilon^{-\max \{p'\!,\, q'\}}$ samples, 
where $p'$ and $q'$ are the H\"older conjugates of $p$ and $q$. 
This could be a much larger sample size 
than the asymptotic $\epsilon^{-2}$ many samples 
which are needed for a random variable $X\in L^2(\Omega;H)$, 
and it could become rapidly unaffordable 
as $p$ or $q$ decrease.
We note that while it is quite immediate to verify 
the sharpness of the rate $1-\frac{1}{\min\left\{p,q\right\}}$ 
with respect to~$q$ (take, e.g., the real-valued random variable 
$u(\omega):=\omega^{-\frac{1}{q+\delta}}$, $\delta>0$, 
which lies in $L^q((0,1);\bbR)$, and estimate 
$\int_0^1 u(\omega)\rd \w$ using the Monte Carlo method), 
to the best of the authors' knowledge, 
there are no numerical experiments yet 
confirming the sharpness of 
the rate with respect to the Rademacher type $p$; 
an aspect that will be addressed in this work.

When solving differential equations, 
numerical methods yield finite-dimensional approximations 
of the random variable $X$. 
As the corresponding additional discretization error 
has to be taken into account for the overall accuracy, 
this gives rise to more sophisticated sampling methods, 
which leverage the availability of 
several \emph{levels} of approximations, 
each of different accuracy but also of different computational cost. 
A vast literature has been devoted in the last two decades to 
MultiLevel Monte Carlo (MLMC) methods, 
starting from the seminal works of Heinrich 
in the context of numerical integration \cite{heinrich2001multilevel} 
and of Giles for 
stochastic differential equations \cite{giles2008multilevel}. 
More recently, a multi-index variant that permits 
to consider multiple discretization parameters 
simultaneously has been proposed in~\cite{haji2016multi}.
Single- and multilevel Monte Carlo methods 
have been extensively analyzed in the Hilbertian case 
to approximate the mean of the solution for  
partial differential equations with random coefficients, 
random initial data or random boundary 
conditions~\cite{barth2011multi, 
	cliffe2011multilevel, 
	graham2016mixed, 
	lord2014introduction, 
	mishra2012sparse, 
	teckentrup2013further} 
and for stochastic differential 
equations \cite{CoxEtAl2021, 
	giles2006improved, 
	giles2013analysis, 
	higham2015introduction,
	xia2012multilevel}.
Concerning higher-order moments, \cite{bierig2015convergence,bierig2016estimation} 
rely on multilevel strategies to 
approximate diagonals of central statistical moments, 
while \cite{barth2011multi,chernov2023simple, mishra2012sparse} 
explore sparse tensor methods to approximate full moments.
With regard to Banach space valued random variables, 
the MLMC method applied to a 
random degenerate convection diffusion equation
is analyzed in~\cite{koley2017multilevel}, 
while an abstract analysis of single- and multilevel Monte Carlo 
methods for both the mean and arbitrary injective $k$th moments 
is discussed in~\cite{KKChS2024}. 
In particular, it is shown in \cite{KKChS2024}  
that the Rademacher type of $E$ generally determines 
the Monte Carlo convergence rate of injective moments 
of arbitrary order; 
consequently, 
the number of samples to use on each level
for MLMC estimation is chosen accordingly. 

\subsection{Contributions}
 
The contributions of this manuscript are threefold.

Our first contribution relies on the observation that 
numerical methods yield approximations that belong 
to finite-dimensional subspaces 
$(E_\ell)_{\ell\in\bbN}\subseteq E$. 
It is known that a finite-dimensional normed space is of 
any Rademacher type ${r\in [1,2]}$, 
but with Rademacher \emph{type constants} 
that may depend on the dimension of the subspace. 
We therefore revisit the analysis of~\cite{KKChS2024}
by taking into account the dependence 
of these type constants on the dimension 
of the approximation space, and we derive 
refined single- and multilevel Monte Carlo error estimates 
for the first moment in the $L^q(\Omega;E)$-norm 
using the Rademacher type property 
for any $r\in [1,\min\{q,2\}]$. 
By further optimizing the error bound over $r$, 
we obtain new complexity results, which we further extend to 
arbitrary injective $k$th moments provided that $X\in L^{kq}(\Omega;E)$.
In the most favorable case, 
the new analysis reveals that 
the complexity of single- and multilevel Monte Carlo methods 
is \emph{independent} of the Rademacher type $p$ of $E$, 
and depends exclusively on the integrability~$q$ 
of the random variable, 
and on additional parameters describing 
the growth of the type constants with respect to the dimension.

Our second contribution is devoted to 
$L^p(S)$-valued random variables which are central 
in many applications. 
Here, $(S,\cS,\mu)$ is 
a $\sigma$-finite measure space 
which allows to approximate functions 
in $L^p(S)$ by simple functions based on 
integral averages on disjoints subsets in~$\cS$. 
This property is satisfied, e.g., by 
any (not necessarily bounded) Borel 
subset $S\subseteq \bbR^d$. 
We show that for a random variable 
${X\in L^q(\Omega;L^p(S))}$,  
$q\in(1,\infty)$ and $p\in[1,\infty)$, 
satisfying the additional integrability assumption 
$X\in L^p(S;L^q(\Omega))$, i.e., 
the pointwise $q$th moments 
are $L^{\nicefrac{p}{q}}(S)$-integrable, 
the convergence rate of the Monte Carlo first moment estimator, 
\emph{even at the continuous level}, 
is exclusively determined by the integrability parameter $\min\{q,2\}$, 
and thus independent of the Rademacher type $\min\{p,2\}$. 
In particular, this guarantees convergence of the 
Monte Carlo method for a class of 
$L^1(S)$-valued random variables 
which is not covered by the 
standard analysis based on the Rademacher type. 
Furthermore, motivated by numerical experiments, 
we examine the setting in which the Monte Carlo error
is measured in the $L^q(\Omega;L^p(S))$-norm, 
while $X$ possesses
additional regularity characterized 
by a parameter $\widetilde{q}$,
that is,
$~X\in L^{\widetilde{q}}(\Omega;L^p(S))\cap L^p(S;L^{\widetilde{q}}(\Omega))$
with $\widetilde{q}\in [q,\infty)$.
This analysis reveals a trade-off in the error estimates 
between the decay rate of the sampling error 
and of the strong approximation error and ultimately 
leads to the derivation of 
new optimized multilevel complexity estimates. 

Finally, our third contribution consists in extensive numerical 
experiments that i) verify the sharpness of 
the convergence rate $1-\frac{1}{\min\left\{p,q\right\}}$ 
with respect to the Rademacher type $p$ at the continuous level, 
ii) verify that for a random variable 
$X\in L^q(\Omega;L^p(S))\cap L^p(S;L^q(\Omega))$ 
the convergence rate at the continuous level is 
$1-\frac{1}{\min\{q,2\}}$ irrespectively of $p$, 
iii) confirm the sharpness of the 
single- and multilevel complexity analyses 
presented in this manuscript 
in approximating both the mean and the second moment.
Although the numerical experiments rely on rather simple 
mathematical models, 
their implementation is very delicate: 
the need to handle random variables 
with low integrability and with values 
in a Banach space leads to pathwise solutions 
with singularities, which in turn demand robust 
and accurate numerical methods.

\subsection{Layout}
The rest of the manuscript is organized as follows. 
Section~\ref{sec:prel} introduces 
the notation and, 
in Subsections~\ref{sec:prel_Rad_types} 
and~\ref{sec:prel_kmoments}, 
recalls the background material on Rademacher types 
and $k$-fold tensor product spaces 
that is relevant for this work. 
Section~\ref{sec:dimension_dependent_constants} is devoted 
to the study of Monte Carlo methods accounting for 
dimension-dependent type constants. 
Subsections~\ref{sec:SLMC} and~\ref{sec:MLMC} 
present optimized complexity results 
for single-level and multilevel Monte Carlo methods for the 
approximation of the first moment. 
Subsection~\ref{sec:SL_MLMC_kmoments} 
extends these analyses to the estimation 
of general $k$th moments. 
Section~\ref{sec:Minkowski} focuses 
on $L^p(S)$-valued random variables 
and presents new complexity results that, 
under additional conditions, 
lead to improved estimates 
for single- and multilevel Monte Carlo methods 
for both first and higher-order moments. 
Section~\ref{sec:numerics} discusses 
two numerical experiments aiming at verifying 
the theoretical results presented in 
Sections~\ref{sec:dimension_dependent_constants} 
and~\ref{sec:Minkowski}.
Section~\ref{sec:Concl} provides the concluding remarks, 
and finally the Appendix details an algorithm for the computation 
of the injective tensor norm, 
which is used in our numerical experiments.

\section{Preliminaries}\label{sec:prel}
\subsection{Notation}\label{sec:prel_notation}
Given parameter sets  
$\mathscr{P},\,\mathscr{Q}$, 
and mappings 
$\mathscr{F},\mathscr{G}\from \mathscr{P}\times\mathscr{Q}\to\bbR$, 
we use the notation  
$\mathscr{F} (p,q) \lesssim_{q} \mathscr{G} (p,q)$ 
to indicate that for each $q \in \mathscr{Q}$ 
there exists a constant ${C_q \in(0,\infty)}$ 
such that 
$\mathscr{F} (p,q) \leq C_q \, \mathscr{G} (p,q)$ 
holds 
for all ${p\in \mathscr{P}}$. 
Whenever both relations, 
$\mathscr{F} (p,q) \lesssim_q \mathscr{G}(p,q)$ 
and $\mathscr{G} (p,q) \lesssim_q \mathscr{F}(p,q)$, 
hold simultaneously, 
we write $\mathscr{F} (p,q) \eqsim_q \mathscr{G} (p,q)$.   

We denote by $(E,\norm{\,\cdot\,}{E})$ a Banach space 
over~$\bbR$,  
$B_E:= \{x\in E : \norm{x}{E} \leq 1 \}$ its closed unit ball, 
and 
$\cB(E)$ the Borel $\sigma$-algebra 
on 
$(E,\norm{\,\cdot\,}{E})$.
The dual space 
of all continuous linear functionals 
$g \from E \to \bbR$ 
is denoted by $E'$.  
We write $g(x)$ or $\langle g, x \rangle$ 
for the duality pairing between $g\in E'$ 
and $x\in E$, and 
$\|g\|_{E'} := \sup_{x\in B_E} | g(x) |$ 
for the norm on $E'$. 

For a measure space $(S,\cS,\mu)$ 
and $p\in [1,\infty)$, 
$L^p(S;E):=L^p(S,\cS,\mu;E)$ 
denotes  
the space of (equivalence classes of) 
$E$-valued, strongly measurable, 
$p$-integrable functions on $S$, 
equipped with the norm 
\[
	\|f\|_{L^p(S;E)} := 
	\left(\int_S \|f(s)\|_E^p \, \rd \mu(s)\right)^{\nicefrac{1}{p}}.
\]
Whenever the functions are real-valued,  
we abbreviate 
$L^p(S):=L^p(S;\bbR)$. The indicator function of a measurable subset $A$ of $S$ is denoted by $\mathbf 1_A$.
For a second measure space  
$(\widetilde{S},\widetilde{\cS},\widetilde{\mu})$, 
$(S\times\widetilde{S},\cS\otimes\widetilde{\cS},
\mu\otimes\widetilde{\mu})$ 
denotes 
the product measure space, i.e.,  
$S\times\widetilde{S}$ is the 
set of all tuples $(s,\widetilde{s})$ with  
$s\in S$, $\widetilde{s}\in\widetilde{S}$,
$\cS\otimes\widetilde{\cS}$ is the product 
$\sigma$-algebra generated by 
all sets of the form $A\times\widetilde{A}$ with $A\in\cS$, 
$\widetilde{A}\in\widetilde{\cS}$, 
and $\mu\otimes\widetilde{\mu}$ is the uniquely defined 
product measure satisfying 
$(\mu\otimes\widetilde{\mu})(A\times\widetilde{A})
= 
\mu(A)\widetilde{\mu}(\widetilde{A})$ 
for all $A\in\cS$,  
$\widetilde{A}\in\widetilde{\cS}$. 
For $p\in [1,\infty)$, $k\in \bbN$, 
and for an open subset 
$D$ of $\bbR^d$, $d\in\bbN$, 
$W^{k,p}(D)$ is the Sobolev space
 containing all elements 
in $L^p(D) = L^p(D,\cB(D),\lambda_d;\bbR)$, 
$\lambda_d$ being the Lebesgue measure, 
whose weak derivatives up to order~$k$ 
exist and are in~$L^p(D)$. 
For $p\in [1,\infty)$, 
$\ell^p$ is the Banach space of 
$p$-summable real-valued sequences 
equipped with the norm 
$\|x\|_{\ell^p}
:=
\bigl(\sum_{j=1}^\infty |x_j|^p \bigr)^{\nicefrac{1}{p}}$ 
and, for every $N\in \bbN:=\left\{1,2,\dots\right\}$, 
we denote by $\ell_N^p$ the space $\bbR^N$ 
equipped with the $\ell^p$-norm. 
For $p\in[1,\infty]$, we denote 
by~$p'$ the H\"older conjugate of~$p$, 
i.e., $\tfrac{1}{p} + \tfrac{1}{p'} = 1$, 
using the convention that $\tfrac{1}{\infty}=0$. 
In addition, $\ell^\infty$ is the space of all bounded sequences, $L^\infty(D)$ 
consists of essentially bounded  
strongly measurable functions, 
and $W^{k,\infty}(D)$ consists of functions 
whose weak derivatives up to order $k$ 
are in $L^\infty(D)$.

Throughout this article, we assume that
$(\Omega,\cA,\bbP)$ is 
a complete probability space
with expectation operator $\bbE$, 
and we mark statements which hold 
$\bbP$-almost surely with $\bbP$-a.s. 
We only consider strongly measurable mappings $X\from\Omega\rightarrow E$, 
that is, mappings which are
 i) measurable with respect to the 
 Borel $\sigma$-algebra $\mathcal{B}(E)$ on~$E$, 
 and ii) almost surely separably valued, i.e., 
 $X\in E_0$ holds $\bbP$-a.s.\ for some separable subspace 
 $E_0\subseteq E$. 
 For any strongly measurable map 
 $X\from\Omega\rightarrow E$ 
 further satisfying 
 $\bbE[ \|X\|_E ]<\infty$, 
 the Bochner integral
\[
	\int_\Omega X(\w) \rd\bbP(\w) 
	=:\bbE [ X ] \in E, 
\]
is well-defined 
\cite[Proposition~1.2.2]{AnalysisInBanachSpacesI2016}. 
For any $q\in [1,\infty)$ and real Banach space 
$(E,\norm{\,\cdot\,}{E})$ 
we denote by 
$L^q(\Omega;E)$ the 
Bochner space consisting of all (equivalence classes of) 
$E$-valued strongly measurable random variables 
${\eta\from\Omega\rightarrow E}$ 
such that 
$\bbE\left[ \|\eta\|_E^q\right]<\infty$, 
equipped with the norm 
$\|\eta\|_{L^q(\Omega;E)}
:=
\bigl(\bbE\bigl[ \| \eta \|_E^q \bigr]\bigr)^{\nicefrac{1}{q}}$. 

\subsection{Rademacher types}\label{sec:prel_Rad_types}
In this subsection, we recall the necessary definitions and tools 
to develop a Monte Carlo error analysis in Banach spaces. 
For a comprehensive overview, 
we refer to the monographs 
\cite{AnalysisInBanachSpacesII2017,LedouxTalagrand2011}.
\begin{definition}[Rademacher family]
\label{def:rade-family} 
	Let $(\Omegat,\cAt,\bbPt)$ 
	be a complete probability space and 
	$r_j \from \Omegat\to\{-1,1\}$, $j\in J\subseteq\bbN$, 
	be a family of independent random variables such that  
	$\bbPt(r_j = -1)=\bbPt(r_j = 1)=\tfrac{1}{2}$ for all $j\in J$. 
	Then, the collection of random variables 
	$(r_j)_{j\in J}$ is called a \emph{Rademacher family}.
\end{definition} 
Rademacher families frequently appear 
when estimating  
sums of independent zero-mean 
random variables taking values in 
a Banach space $E$ 
in the norm of $L^q(\Omega;E)$ 
due to the following symmetrization result,
see e.g.\  \cite[Lemma~6.3]{LedouxTalagrand2011}.
\begin{lemma}[Symmetrization]
\label{lem:symmetrization} 
	Let  $q\in[1,\infty)$, $M\in\bbN$, $(r_j)_{j=1}^M$ be 
	a Rademacher family on a complete probability space 
	$(\Omegat,\cAt,\bbPt)$, 
	and let $\eta_1,\ldots,\eta_M\in L^q(\Omegat;E)$ be 
	centered random variables, 
	i.e., $\bbEt[\eta_j]=0$ for $1\leq j \leq M$, 
	taking values 
	in a real Banach space $(E,\norm{\,\cdot\,}{E})$ 
	such that $\eta_1,\ldots,\eta_M, r_1, \ldots, r_M$ are independent. 
	Then, 
	\[
		\biggl\| 
		\sum_{j=1}^M \eta_j 
		\biggr\|_{L^q(\Omegat;E)} 
		\leq 
		2 \, 
		\biggl\| 
		\sum_{j=1}^M r_j \eta_j 
		\biggr\|_{L^q(\Omegat;E)} . 
	\]
\end{lemma} 

There exists an extensive literature 
on Rademacher sums 
(and on more general random sums, 
see e.g.~\cite[Chapter~6]{AnalysisInBanachSpacesII2017} 
for the corresponding analysis in Banach spaces). 
In this context, the following 
Khintchine and Kahane--Khintchine 
inequalities (see~\cite[Lemma~4.1]{LedouxTalagrand2011} 
and~\cite{kahane1985some}) 
play a central role. 
 
\begin{definition}[Khintchine constants]
\label{def:khintchine}
	Let $q\in [1,\infty)$. 
        The Khintchine constants $A_q, B_q\in(0,\infty)$ 
	are the 
	largest, respectively, 
	smallest positive constants such that, 
    for any Rademacher family $(r_j)_{j\in\bbN}$ 
    on a complete probability space~$(\Omegat,\cAt,\bbPt)$ 
    (with expectation $\bbEt$) 
	and any finite sequence $(a_j)_{j=1}^n$ 
	of real numbers, it holds that
	\[ 
		A_q 
		\Biggl( \sum_{j=1}^n a_j^2 \Biggr)^{\nicefrac{1}{2}} 
		\leq 
		\Biggl( \bbEt \Biggl[ \biggl| \sum_{j=1}^n r_j a_j \biggr|^q \Biggr] \Biggr)^{\nicefrac{1}{q}} 
		= 
		\biggl\| 
		\sum_{j=1}^n r_j a_j 
		\biggr\|_{L^q(\Omegat;\bbR)}  
		\leq 
		B_q 
		\Biggl( \sum_{j=1}^n a_j^2 \Biggr)^{\nicefrac{1}{2}}. 
	\] 
\end{definition} 

Since the norm on $L^q(\Omegat;\bbR)$ 
is increasing in $q$ and
\[
	\biggl\| \sum_{j=1}^n r_j a_j \biggr\|_{L^2(\Omegat;\bbR)}
	=
	\Biggl( \sum_{j=1}^n a_j^2 \Biggr)^{\nicefrac{1}{2}},
\]
it holds that $A_q=1$ for every $q\in [2,\infty)$ 
while $B_q=1$ for every $q\in [1,2]$. 
Sharp values of $A_q$ and $B_q$ 
in the remaining nontrivial cases were 
obtained in \cite{Haagerup1981}. 

\begin{definition}[Kahane--Khintchine constants]
\label{def:kahane-khintchine} 
	Let $p,q\in[1,\infty)$. 
	The  $(q,p)$ Kahane--Khintchine constant $K_{q,p}$ 
	is the smallest positive constant 
	such that 
	for any Rademacher family $(r_j)_{j\in\bbN}$ 
	on a complete probability space 
	$(\Omegat,\cAt,\bbPt)$, 
	for any real Banach space $(E,\norm{\,\cdot\,}{E})$,  
	for all $n\in\bbN$, 
	and every $x_1,\ldots,x_n\in E$, 
	\begin{equation*}
		\biggl\| 
		\sum_{j=1}^n r_j x_j
		\biggr\|_{L^q(\Omegat;E)} 
		\leq K_{q,p} \,
		\biggl\| 
		\sum_{j=1}^n r_j x_j
		\biggr\|_{L^p(\Omegat;E)}. 
	\end{equation*} 
\end{definition} 

\begin{remark} 
	For $q\leq p$, $K_{q,p}=1$ holds 
	due to H\"older's inequality. 
	For $q>p$, finiteness of the constants $K_{q,p}$ 
	was proved by Kahane \cite{kahane1985some}, 
	and it implies that all $L^q$-norms are equivalent 
	for Rademacher sums.
\end{remark}

\begin{definition}[Rademacher type]
\label{def:type-p-constant}
	A real Banach space $(E,\norm{\,\cdot\,}{E})$ 
	is said to be of \emph{(Rademacher) type} $p\in[1,2]$ 
	if there exists a constant $\tau\in(0,\infty)$ 
	such that 
	for any Rademacher family $(r_j)_{j\in\bbN}$ 
	on a complete probability space 
	$(\Omegat,\cAt,\bbPt)$ 
	(with expectation $\bbEt$), 
	for every $n\in\bbN$, 
	and all vectors $x_1,\ldots,x_n\in E$, 
	\begin{equation}\label{eq:type} 
		\biggl\| 
		\sum_{j=1}^n r_j x_j 
		\biggr\|_{L^p(\Omegat;E)}  
		=
		\Biggl( 
		\bbEt   
		\Biggl[ 
		\biggl\|
		\sum_{j=1}^n r_j x_j 
		\biggr\|_E^p 
		\Biggr] 
		\Biggr)^{\nicefrac{1}{p}} 
		\leq 
		\tau 
		\Biggl( 
		\sum_{j=1}^n \norm{x_j}{E}^p 
		\Biggr)^{\nicefrac{1}{p}} \!. 
	\end{equation} 
	In this case, the smallest constant $\tau\in(0,\infty)$ 
	in \eqref{eq:type} 
	is called the \emph{type $p$ constant of $E$} 
	and will be denoted by $\tau_p(E)$. 
\end{definition}

\begin{remark}\label{rm:type}
	Due to the triangle inequality, 
	every Banach space is at least of type~$1$. 
	Conversely, the type cannot be larger than $2$ 
	(due to finiteness of the Khintchine constant $A_q$ 
	for any $q\in [1,\infty)$) 
	and every Hilbert space is of Rademacher type~2, 
	see~\cite{Kwapien1972}.
	In addition, 
	since the Bochner norm $\|\cdot\|_{L^q(\Omegat;E)}$ 
	is increasing in~$q$ while the sequence norm 
	$\|\cdot\|_{\ell^q_n}$ is decreasing in~$q$, 
	if $(E,\|\cdot\|_E)$ has type $p\in [1,2]$ 
	with type constant $\tau_p(E)$, 
	then it is also of type $r\in [1,p]$, with $\tau_r(E)\leq \tau_p(E)$.
\end{remark}

\begin{example}\label{ex:Ls-type}
	Let $s\in [1,\infty)$, $k\in\bbN$, 
	$(S,\cS,\mu)$ be a $\sigma$-finite measure space 
	and $D\subset\bbR^d$ be a bounded Lipschitz domain.
	The spaces $\ell^s$, $L^s(S)$ and $W^{k,s}(D)$
	all have Rademacher type $p= \bar{s} := \min\left\{s,2\right\}$.
	For this $p$, 
	the type $p$ constants of $\ell^s$ and $L^s(S)$ satisfy 
	$\tau_p(\ell^s)=\tau_p(L^s(S))=\tau_{\bar{s}}(\ell^s)= 
	\tau_{\bar{s}}(L^s(S))=B_s$, 
	where $B_s\in(0,\infty)$ is 
	the Khintchine constant of Definition \ref{def:khintchine} 
	and we recall that $B_s = 1$ for $s\in[1,2]$.  
\end{example}

\begin{remark}\label{rm:type_finite-dim}
	Any real Banach space $(E,\|\cdot\|_E)$ 
	of finite dimension is of Rademacher type~$2$. 
	To see this, let $N$ be the dimension of $E$, 
	$\left\{\widetilde{x}_1,\dots,\widetilde{x}_N\right\}$ 
	a basis of $E$, 
	and consider the linear isomorphism 
	$T\from E\rightarrow \bbR^N$ 
	such that $T(\widetilde{x}_j)=\eb_j$, 
	where $\eb_j$ is 
	the $j$th canonical unit vector of $\bbR^N\!$, 
	for all $j\in\{1,\dots,N\}$. 
	The space $E$ is then isomorphic to~$\ell^2_N$, 
	that is, $\bbR^N$ equipped with the Euclidean norm, 
	but in general not isometrically. 
	However, as all norms are equivalent 
	in a finite-dimensional space, 
	there exist two constants 
	$\underline{C}_N$ and $\overline{C}_N$, 
	which may depend on the dimension $N$, 
	such that
	\[
		\underline{C}_N \|Tx\|_{\ell^2_N} 
		\leq \|x\|_E
		\leq \overline{C}_N \|Tx\|_{\ell^2_N} 
		\qquad 
		\forall x\in E. 
	\]
	Hence,
	for any Rademacher family $(r_j)_{j\in\bbN}$, 
	for every $n\in\bbN$, 
	and all $x_1,\ldots,x_n\in E$, 
	\begin{align*} 
		\biggl\| 
		\sum_{j=1}^n r_j x_j 
		\biggr\|_{L^2(\Omegat;E)} 
		&\leq 
		\overline{C}_N 
		\biggl\| 
		\sum_{j=1}^n r_j Tx_j 
		\biggr\|_{L^2(\Omegat;\ell^2_N)} 
		\\
		&\leq \overline{C}_N
		\Biggl( 
		\sum_{j=1}^n \|Tx_j\|_{\ell^2_N}^2 
		\Biggr)^{\nicefrac{1}{2}} 
		\leq \frac{\overline{C}_N}{\underline{C}_N} 
		\Biggl( \sum_{j=1}^n \|x_j\|_E^2 \Biggr)^{\nicefrac{1}{2}}\!,
	\end{align*} 
	where we used that $\ell^2_N$ is of 
	Rademacher type~$2$ 
	and has type~$2$ constant equal to~$1$. 
	We conclude that $E$ is of type $2$ 
	with $\tau_2(E)\leq \frac{\overline{C}_N}{\underline{C}_N}$.
\end{remark}

\subsection{\texorpdfstring{$k$-fold}{k-fold} 
	tensor product spaces and 
	\texorpdfstring{$k$th}{kth} moments}
\label{sec:prel_kmoments}
In this subsection, 
we introduce $k$-fold tensor products of a 
real Banach space $(E, \norm{\,\cdot\,}{E})$ 
and interpret the $k$th moment
of a random variable $X\from\Omega\rightarrow E$ 
as an element of a specific 
complete normed tensor product space.

For an integer $k\in\bbN$, 
the \emph{algebraic tensor product space} $\otimes^k E$ 
is defined as the vector space consisting of all finite sums 
of the form
\[
	\sum_{j=1}^M x_{j,1}\otimes\cdots\otimes x_{j,k},
	\qquad 
	x_{j,i}\in E,\;1\leq j\leq M,\;1\leq i\leq k,\; M\in\mathbb{N},
\] 
equipped with the algebraic operations rendering 
the tensor product linear in each of its $k$ components, 
see e.g.~\cite[Section 1.1]{floret1997natural} 
and \cite[Chapter~1]{ryan2002introduction}.
The space $\otimes^k E$ can be equipped 
with different norms. 
Common choices are the 
\emph{injective} and \emph{projective} 
tensor norms, see 
\cite[Chapters~2 and~3]{ryan2002introduction}, 
which are respectively the weakest 
and strongest crossnorm, 
see \cite[Proposition~6.1]{ryan2002introduction}.
As shown in \cite[Example 3.21]{KKChS2024}  
Monte Carlo methods will, in general, 
not convergence in the projective tensor norm, 
and we therefore 
introduce only the former. 
For $U\in \otimes^k E$ with representation 
$U=\sum_{j=1}^M x_{j,1}\otimes\cdots\otimes x_{j,k}$, 
the \emph{injective tensor norm} is defined by
\[ 
	\|U\|_{\varepsilon} := 
	\sup\left\{ 
	\biggl|\sum_{j=1}^M \prod_{i=1}^k f_i(x_{j,i})\biggl|
	\;\Bigg|\; f_1,\dots,f_k\in B_{E'} 
	\right\},
\] 
and it can be shown that the value of this supremum 
is independent of the choice for the representation of 
$U\in \otimes^k E$. 
The crossnorm property implies that, 
for every elementary tensor 
$x_1\otimes \cdots\otimes x_k$,  
\begin{equation}\label{eq:crossnorm-property_injective_norm}
	\|x_1\otimes \cdots\otimes x_k\|_{\varepsilon} 
	=
	\sup\left\{ \biggl| \prod_{i=1}^k f_i(x_i)\biggl|
	\;\Bigg|\; 
	f_1,\dots,f_k\in B_{E'}\right\}
	=
	\prod_{i=1}^k \|x_i\|_E. 
\end{equation}
The \emph{$k$-fold injective tensor product space}, 
denoted by $\otimes_{\varepsilon}^k E$, 
is then defined as the closure of $\otimes^k E$ 
with respect to the injective norm, that is, 
$\otimes_{\varepsilon}^k E
:=\overline{\otimes^k E}^{\|\,\cdot\, \|_{\varepsilon}}\!$.

The analysis of Subsection~\ref{sec:SL_MLMC_kmoments} 
builds on the convergence results for 
multilevel Monte Carlo methods presented in \cite{KKChS2024}, 
which are formulated in a \emph{symmetric} injective tensor norm. 
We therefore introduce 
subspaces of $\otimes^k E$ and of $\otimes^k_\varepsilon E$, 
containing only symmetric elements. 
To this end, for any $x_1\otimes\cdots\otimes x_k \in \otimes^k E$, 
we introduce its symmetrization,
\[
	s(x_1\otimes\cdots\otimes x_k)
	:=
	\frac{1}{k!}\sum_{\sigma\in S_k} x_{\sigma(1)}\otimes\cdots\otimes x_{\sigma(k)},
\]
where $S_k$ is the group of permutations of 
the set $\left\{1,\dots,k\right\}$. 
The $k$-fold symmetric algebraic tensor product of $E$, 
denoted by $\otimes^{k,s} E$, 
is then defined by 
\begin{align*} 
	\otimes^{k,s} E
	&:=
	\left\{
	\sum_{j=1}^M s(x_{j,1}\otimes\cdots \otimes x_{j,k}) 
	\;\Bigg| \; M\in \bbN,\; x_{j,i}\in E,\;1\leq j\leq M,\; 1\leq i\leq k
	\right\}
	\\
	&\textcolor{white}{:}=
	\left\{
	\sum_{j=1}^M \lambda_j \, {\otimes^k} x_{j} 
	\;\Bigg| \; 
	M\in \bbN,\; \lambda_j\in\bbR,\; 
	x_{j}\in E,\;1\leq j\leq M 
	\right\}, 
\end{align*} 
see \cite[Section~1.5]{floret1997natural},
where we used the notation
\[
	{\otimes^k} x 
	:= 
	\underbrace{x\otimes \cdots \otimes x}_{\normalfont 
	\text{$k$ times}} 
	\qquad \forall x \in E.
\]
For   
$U=\sum_{j=1}^M \lambda_j \, {\otimes^k} x_j\in \otimes^{k,s} E$, 
the \emph{symmetric injective tensor norm} is given by
\[ 
	\|U\|_{\varepsilon_s} 
	:= 
	\sup\left\{ 
	\biggl|\sum_{j=1}^M \lambda_j f(x_j)^k \biggr| 
	\;\Bigg|\; f\in B_{E'}
	\right\},
\] 
which does not depend on the choice 
of the representation of~$U$ either.
The \emph{$k$-fold symmetric injective tensor product of $E$} 
is defined at the closure of $\otimes^{k,s} E$ 
with respect to the norm $\|\cdot\|_{\varepsilon_s}$. 
We finally recall that the symmetrization~$s$ 
can be extended to a linear and continuous projection 
from $\otimes^k_{\varepsilon} E$ 
to $\otimes^{k,s}_{\varepsilon_s} E$, 
and further that~$\|\cdot\|_{\varepsilon_s}$ 
and~$\|\cdot\|_{\varepsilon}$ 
are equivalent as norms on $\otimes^{k,s}_{\varepsilon_s} E$, 
with equivalence constants depending on $k$, 
see \cite[Section~2.2]{KKChS2024}.

For a strongly measurable $E$-valued random variable $X$, 
its \emph{injective $k$th moment $\bbM^k_{\varepsilon}[X]$} 
is defined as 
the expectation 
(see, e.g., \cite[Section~3.1]{janson2015higher}) 
\[
	\bbM^k_{\varepsilon}[X]
	:= 
	\bbE\bigl[ {\otimes^k} X \bigr] 
	= 
	\int_\Omega  {\otimes^k} X(\omega) \, \rd \bbP(\omega) 
	= 
	\int_\Omega 
	\underbrace{X(\omega) \otimes \cdots \otimes X(\omega)}_{ 
		\text{\normalfont $k$ times}}
	\, \rd\bbP(\omega), 
\] 
whenever this integral exists (in Bochner sense)
in the $k$-fold injective 
tensor product space $\otimes^k_{\varepsilon} E$. 
A sufficient condition for the existence 
of the injective $k$th moment 
$\bbM^k_{\varepsilon}[X]$ is that 
$X\in L^k(\Omega;E)$. 
Indeed, the strong measurability of~$X$ 
implies that 
${\otimes^k} X \from \Omega\to \otimes^k_{\varepsilon} E$ 
is strongly measurable since the nonlinear mapping 
$E\ni x \mapsto {\otimes^k} x \in \otimes^k_{\varepsilon} E$ 
is continuous and, secondly, due to 
\eqref{eq:crossnorm-property_injective_norm},
\[
	\bbE\bigl[ 
	\| {\otimes^k} X\|_{\varepsilon} 
	\bigr] 
	= 
	\bbE\bigl[ \| X\|^k_E \bigr]<\infty.
\]
Thus, ${\otimes^k} X\in L^1(\Omega;\otimes^k_{\varepsilon} E)$ 
and the existence of $\bbM^k_{\varepsilon}[X]$ 
follows from the definition of the Bochner integral  
\cite[Proposition~1.2.2]{AnalysisInBanachSpacesI2016}. 
Finally, notice that the injective $k$th moment is 
an element of the $k$-fold \emph{symmetric} 
injective tensor product space since, 
due to the continuity and linearity of the symmetrization,
\[
	s\bigl(\bbM^k_{\varepsilon} [X ]\bigr) 
	= 
	s\bigl(\bbE\bigl[ {\otimes^k X}\bigr]\bigr)
	=
	\bbE\bigl[ s( {\otimes^k} X) \bigr] 
	= 
	\bbE\bigl[ {\otimes^k}X\bigr] 
	= 
	\bbM^k_{\varepsilon}[X].
\]

We have so far introduced the $k$-fold 
(symmetric) injective tensor product of 
a generic Banach space $E$. 
In Section~\ref{sec:Minkowski}, 
we will perform an analysis 
tailored for $L^p(S)$-valued 
random variables, where $p\in[1,\infty)$. 
The results presented there 
for first moments readily extend to 
higher-order moments 
when considered as elements 
of a tensor product space equipped 
with a crossnorm which is stronger 
than the injective tensor norm and 
which arises naturally 
on the algebraic tensor product space ${\otimes^k} L^p(S)$, 
see e.g.~\cite[Chapter I, Section 7]{defant1993tensor}.
To introduce this tensor norm, 
let $(S^k\!,\cS^k\!,\mu^k)$ 
be the product measure space 
of $k$ copies of $(S,\cS,\mu)$, $p\in [1,\infty)$,
and consider the injective linear map 
$\phi\from {\otimes^k} L^p(S)
\rightarrow  L^p(S^k)$ such that
\[ 
	\phi:\; 
	\sum_{j=1}^M\bigotimes_{i=1}^k f_{j,i}
	\mapsto 
	 \sum_{j=1}^M \prod_{i=1}^k f_{j,i}.
\] 
Here, $\prod_{i=1}^k f_{j,i}$ is interpreted as 
(the equivalence class of)
the product function 
\[ 
	\prod_{i=1}^k f_{j,i}:\;
	(s_1,\dots,s_k)
	\mapsto 
	f_{j,1}(s_1)\cdots f_{j,k}(s_k),
\]
and the injectivity of $\phi$ follows from 
\cite[Proposition~1.2]{ryan2002introduction},
the identification of $(L^p(S))'$ with $L^{p'\!}(S)$, 
and the Fubini theorem.
Using the mapping $\phi$, any element of ${\otimes^k} L^p(S)$
can be identified with a function in $L^p(S^k)$. 
We thus can define a crossnorm 
$\triplenorm{\,\cdot\,}$ on ${\otimes^k} L^p(S)$ as
\[
	\Biggtriplenorm{\sum_{j=1}^M \bigotimes_{i=1}^k f_{j,i}}
	:=
	\Biggl\|
	\phi\Biggl(\sum_{j=1}^M \bigotimes_{i=1}^k f_{j,i}\Biggr)
	\Biggr\|_{L^p(S^k)}
	=
	\Biggl\|\sum_{j=1}^M \prod_{i=1}^k f_{j,i}\Biggr\|_{L^p(S^k)}\!,
\]
and consider the tensor product space 
obtained by taking the closure of ${\otimes^k} L^p(S)$ 
with respect to  this norm, 
$\otimes^k_p L^p(S)
:=\overline{\otimes^k L^p(S)}^{\triplenorm{\cdot}}$.

Assume now that 
$X\from\Omega\rightarrow L^p(S)$ 
is a strongly measurable random variable . 
Its \emph{$L^p$ $k$th moment} is defined as the expectation 
\begin{equation}\label{eq:Lpkmoment}
	\bbM^k_p [ X ]
	:=
	\bbE\bigl[ {\otimes^k} X \bigr] 
	= 
	\int_\Omega {\otimes^k} X(\w) \, \rd\bbP(\w),
\end{equation}
whenever this integral exists (in Bochner sense) 
as an element of
$\otimes^k_p L^p(S) \cong L^p(S^k)$.
A sufficient condition for the $k$th moment to exist is 
$X\in L^k(\Omega;L^p(S))$, since then
\[
	\bbE\bigl[ \triplenorm{ {\otimes^k} X}
	\bigr] 
	= 
	\bbE\bigl[ \| X \|_{L^p(S)}^k \bigr]<\infty,
\]
and the existence of $\bbM^k_p[ X ]$ follows again 
from the theory of the Bochner integral.

\section{Optimized Monte Carlo error analysis 
	based on dimension-dependent type constants}
\label{sec:dimension_dependent_constants}
In this section, we present a Monte Carlo error analysis 
that accounts for the dependence of the Rademacher 
type constants on the dimension of the approximation space. 
Subsections~\ref{sec:SLMC} and~\ref{sec:MLMC} study, 
respectively, single-level and multilevel Monte Carlo methods 
to approximate the first moment, 
that is the mean, of a random variable with values 
in a generic real Banach space $(E,\|\cdot\|_E)$. 
Subsection~\ref{sec:SL_MLMC_kmoments} 
extends this analysis to the approximation of arbitrary $k$th moments.

\subsection{Single-level Monte Carlo methods}\label{sec:SLMC}

Suppose that $(E,\norm{\,\cdot\,}{E})$ is a real Banach space 
of Rademacher type $p\in[1,2]$, 
with type $p$ constant $\tau_p(E)\in(0,\infty)$. 
From Definition \ref{def:type-p-constant}, 
we thus have for any Rademacher family $(r_j)_{j\in\bbN}$ 
on a complete probability space 
$(\widetilde{\Omega},\widetilde{\cA},\widetilde{\bbP})$ 
(with expectation $\widetilde{\bbE}$), 
for every $n\in\bbN$, 
and all vectors $x_1,\ldots,x_n\in E$, 
\[
	\biggl\| 
	\sum_{j=1}^n r_j x_j 
	\biggr\|_{L^p(\widetilde{\Omega};E)}  
	=
	\Biggl( 
	\widetilde{\bbE}
	\Biggl[ 
	\biggl\|
	\sum_{j=1}^n r_j x_j 
	\biggr\|_E^p 
	\Biggr] 
	\Biggr)^{\nicefrac{1}{p}} 
	\leq 
	\tau_p(E) 
	\Biggl( 
	\sum_{j=1}^n \norm{x_j}{E}^p 
	\Biggr)^{\nicefrac{1}{p}} \!. 
\]
Assume that $(E_\ell)_{\ell\in\bbN}$ is a family of 
subspaces of~$E$,  
which have finite dimension 
$N_\ell := \dim(E_\ell) < \infty$. 
Then, for every $\ell\in\bbN$, 
$E_\ell$ has any Rademacher type $r\in[1,2]$ 
(due to $\dim(E_\ell)<\infty$, 
see Remarks~\ref{rm:type} and~\ref{rm:type_finite-dim}) 
and we denote the corresponding type constants 
by $\tau_r(E_\ell)$, $r\in[1,2]$, $\ell\in\bbN$. 
Consequently, for every $\ell\in\bbN$, for all $r\in[1,2]$, 
for any Rademacher family $(r_j)_{j\in\bbN}$ 
on a complete probability space 
$(\widetilde{\Omega},\widetilde{\cA},\widetilde{\bbP})$ 
(with expectation $\widetilde{\bbE}$), 
for every $n\in\bbN$, 
and all vectors $x_1,\ldots,x_n\in E_\ell$, 
\[
	\biggl\| 
	\sum_{j=1}^n r_j x_j 
	\biggr\|_{L^r(\widetilde{\Omega};E)}  
	=
	\Biggl( 
	\widetilde{\bbE}
	\Biggl[ 
	\biggl\|
	\sum_{j=1}^n r_j x_j 
	\biggr\|_E^r 
	\Biggr] 
	\Biggr)^{\nicefrac{1}{r}} 
	\leq 
	\tau_r(E_\ell) 
	\Biggl( 
	\sum_{j=1}^n \norm{x_j}{E}^r 
	\Biggr)^{\nicefrac{1}{r}} \!. 
\]
Note that the type constants satisfy the following: 
\begin{equation}\label{eq:tau-uni-bdd} 
	\forall r \in [1,p]: 
	\quad 
	\sup_{\ell\in\bbN} \tau_r(E_\ell) 
	\leq 
	\tau_p(E) < \infty, 
\end{equation}
and that, for $r>p$, the type $r$ constant 
$\tau_r(E_\ell) $ will, in general, 
not be bounded uniformly in~$\ell$. 

\begin{example} 
	Let $s\in[1,\infty)$ and consider the Banach space $\ell^s$
	which is of type $p=\bar{s}:=\min\{s,2\}$. 
	Furthermore, for every $N\in\bbN$, 
	let $E_N := \ell^s_N \subset \ell^s$. 
	Then,
	\[
		\tau_r(E_N) 
		= 
		\tau_r(\ell^s_N) 
		= 
		\begin{cases} 
			B_s & \text{ if } r\in[1,p] , 
			\\
			N^{\frac{1}{p} - \frac{1}{r}} 
			& \text{ if } r\in(p,2], 
		\end{cases} 
	\]
where $B_s\in(0,\infty)$ is the 
	Khintchine constant from Definition~\ref{def:khintchine}. 
For $r\in [1,p]$, the claim follows from 
\eqref{eq:tau-uni-bdd} and Remark~\ref{rm:type}, 
see also Example~\ref{ex:Ls-type}, 
while for $p=s<2$ and 
$r\in (p,2]$ we may use 
the equivalence relations 
\[ 
	\|u\|_{\ell_N^r}
	\leq 
	\|u\|_{\ell_N^p}
	\leq N^{\frac{1}{p}-\frac{1}{r}}\|u\|_{\ell^r_N}, 
	\quad \forall u\in \ell^p_N=\ell^r_N,
	\quad  1\leq p\leq r<\infty.
\]
\end{example} 

Since the finite-dimensional subspaces $(E_\ell)_{\ell\in \bbN}$ 
of~$E$ have any Rademacher type ${r\in [1,2]}$,  
we readily obtain from \cite[Corollary~3.15]{KKChS2024} 
the following single-level Monte Carlo 
error estimate for the first moment. 

\begin{proposition}\label{prop:dim-dep-MC-1} 	
	Let $\eta\in L^1(\Omega;E)$, 
	$q\in[1,\infty)$, 
	and set $\bar{q}:=\min\{q,2\}$. 
	For every $\ell\in\bbN$, 
	let $M_\ell\in\bbN$ and suppose that 
	$\xi_{\ell,1},\ldots,\xi_{\ell,M_\ell}\in L^q(\Omega;E_\ell)$ are  
	independent and identically distributed 
	random variables taking values in a finite-dimensional 
	subspace $E_\ell\subseteq E$. 
	Then, we have, for 
	every $\ell\in\bbN$ and 
	all $r\in[1,\bar{q}]$, 
	\begin{align*} 
		\biggl\| 
		\bbE[\eta] 
		- \frac{1}{M_\ell}
		\sum_{j=1}^{M_\ell} 
		\xi_{\ell,j}  
		\biggr\|_{L^q(\Omega;E)} 
		&\leq 
		\bigl\| \bbE[ \eta - \xi_{\ell,1}] \bigr\|_{E}  
		\\
		&\quad 
		+ 
		2 K_{q,r} \tau_r(E_\ell)  
		M_\ell^{-\left( 1-\frac{1}{r} \right)} 
		\bigl\| \xi_{\ell,1} - \bbE[\xi_{\ell,1}] \bigr\|_{L^q(\Omega; E)}. 
	\end{align*}
\end{proposition} 
 
In the case that, in addition to the assumptions made in 
Proposition~\ref{prop:dim-dep-MC-1}, 
there exist constants $\alpha,C_\alpha, C_{\sf stab}\in(0,\infty)$ 
such that 
\begin{equation}\label{eq:alpha-stab} 
	\sup_{\ell\in\bbN} 
	N_\ell^\alpha 
	\bigl\| \bbE[ \eta - \xi_{\ell,1}] \bigr\|_{E}  
	\leq 
	C_\alpha , 
	\qquad 
	\sup_{\ell\in\bbN} \, 
	\bigl\| \xi_{\ell,1} - \bbE[\xi_{\ell,1}] \bigr\|_{L^q(\Omega; E)} 
	\leq 
	C_{\sf stab}, 
\end{equation}
then 
we obtain, for every $\ell\in\bbN$ and all $r\in[1,\bar{q}]$, 
the estimate  
\begin{equation}\label{eq:dim-dep-estimate_SLMC}
	\biggl\| 
	\bbE[\eta] 
	- \frac{1}{M_\ell}
	\sum_{j=1}^{M_\ell} 
	\xi_{\ell,j}  
	\biggr\|_{L^q(\Omega;E)} 
	\leq 
	C_\alpha N_\ell^{-\alpha}
	+ 
	2 K_{q,1} C_{\sf stab} 
	\tau_r(E_\ell) 
	M_\ell^{- \frac{1}{r'} \!}, 
\end{equation} 
where we interpret $\tfrac{1}{r'} = 0$ for $r=1$.
Here, we also bounded the Kahane--Khintchine 
constant $K_{q,r}$ from above by~$K_{q,1}$, 
since it does not depend on $\ell$. 
Note that by \cite[Theorem~6.2.4 and p.~49]{AnalysisInBanachSpacesII2017} 
we have
\begin{equation}\label{eq:bound_Kahane}
	\sup_{r\in[1,2]}K_{q,r} 
	\leq 
	K_{q,1} 
	\leq 
	\begin{cases} 
	K_{2,1} \quad\;\;\; = \sqrt{2}
	&\text{ if } q \in [1,2], 
	\\
	K_{q,2}K_{2,1} 
	\leq \sqrt{2(q-1)}  
	&\text{ if } 
	q\in(2,\infty). 
	\end{cases} 
\end{equation}
Provided that $q>1$ and $r\in(1,\bar{q}]$, 
we thus choose the number of Monte Carlo samples as 
\[ 
	M_\ell 
	= 
	\Bigl\lceil [ \tau_r(E_\ell) N_\ell^\alpha]^{r'} \Bigr\rceil
	= 
	\Bigl\lceil 
	[\tau_r(E_\ell) ]^{r'\!} N_\ell^{\alpha r'} 
	\Bigr\rceil ,  
	\qquad 
	\ell\in\bbN, 
\] 
which balances the two contributions 
in \eqref{eq:dim-dep-estimate_SLMC} 
and yields the estimate 
\begin{equation}\label{eq:dim-dep_SLMC-conv} 
	\forall \ell\in\bbN: 
	\quad 
	\biggl\| 
	\bbE[\eta] 
	- \frac{1}{M_\ell}
	\sum_{j=1}^{M_\ell} 
	\xi_{\ell,j}  
	\biggr\|_{L^q(\Omega;E)} 
	\leq 
	\bigl( C_\alpha 
	+ 
	2 K_{q,1} C_{\sf stab} \bigr) 
	N_\ell^{-\alpha}\!. 
\end{equation}
Since this estimate holds for all $r\in(1,\bar{q}]$, 
we can optimize the number of Monte Carlo samples,  
for all $\ell\in\bbN$, by setting
\[ 
	M_\ell^{\rm opt} 
	:= 
	\begin{cases} 
		\bigl\lceil 
		\inf_{r\in(1,\bar{q}]} 
		[ \tau_r(E_\ell) N_\ell^\alpha]^{r'} 
		\bigr\rceil 
		&\quad\text{ if } p = 1, \quad\;\;\; q\in(1,\infty) , 
		\\
		\bigl\lceil 
		\inf_{r\in [p,\bar{q}]} 
		[ \tau_r(E_\ell) N_\ell^\alpha]^{r'}
		\bigr\rceil 
		&\quad\text{ if } p \in (1,2], \ q\in[p,\infty), 
	\end{cases} 
\] 
whenever these infima exist. 
To restrict the interval of $r$ 
from $(1,\bar{q}]$ to $[p,\bar{q}]$ 
in the second case, 
we used \eqref{eq:tau-uni-bdd}  
as well as the fact that the map 
$r\mapsto N_\ell^{\alpha r'\!}$ 
is monotonically decreasing in $r$. 
To summarize the previous observations in a corollary we define, 
for every $\ell\in\bbN$, the function 
\begin{equation}\label{eq:def:fl}
	f_{\alpha;\ell} \from (1,2]  
	\to [0,\infty), 
	\qquad 
	f_{\alpha;\ell}(r) 
	:= 
	[ \tau_r(E_\ell) N_\ell^\alpha]^{r' \!} . 
\end{equation}

\begin{corollary}\label{cor:dim-dep-MC-1} 
	Assume that  
	$(E,\norm{\,\cdot\,}{E})$ 
	is of Rademacher type $p\in[1,2]$, 
	let ${\eta\in L^1(\Omega;E)}$, 
	$q\in(1,\infty)\cap[p,\infty)$, 
	and set  
	$\bar{q}:=\min\{q,2\}$,  
	$R(p,q)  := [p,q]\cap(1,2]$. 
	In addition, for every $\ell\in\bbN$, 
	let  
	${(\xi_{\ell,j})_{j\in\bbN}\subset L^q(\Omega;E_\ell)}$ be 
	independent and identically distributed 
	random variables taking values in a subspace 
	$E_\ell\subseteq E$ of dimension 
	$N_\ell := \dim(E_\ell) < \infty$ 
	such that~\eqref{eq:alpha-stab} holds 
	for some constants~$\alpha, C_\alpha, C_{\sf stab}\in(0,\infty)$. 
	\begin{enumerate}[label={\normalfont(\roman*)},leftmargin=8mm, topsep=3pt]
		\item 
			In the case that 
			the function 
			$r\mapsto f_{\alpha;\ell}(r)$ defined 
			in \eqref{eq:def:fl} 
			is 
			monotonically  
			decreasing on $R(p,q)\subseteq(1,2]$  
			for every $\ell\in\bbN$, 
			then for the choice 
			\[
				M_\ell 
				:= 
				M_\ell^{\rm opt} 
				= 
				\Bigl\lceil [\tau_{\bar{q}}(E_\ell) N_\ell^\alpha]^{\bar{q}'} \Bigr\rceil, 
				\qquad 
				\ell\in\bbN, 
			\] 
			the (single-level) 
			Monte Carlo error estimate \eqref{eq:dim-dep_SLMC-conv} holds. 
		\item 
			Whenever $p \in(1,2]$
			and the function 
			$r\mapsto f_{\alpha;\ell}(r)$ 
			defined in \eqref{eq:def:fl} is 
			monotonically increasing on $R(p,q)$ 
			for every $\ell\in\bbN$, 
			then for the choice 
			\[
				M_\ell 
				:=
				M_\ell^{\rm opt} 
				= 
				\Bigl\lceil [\tau_p(E_\ell) N_\ell^\alpha]^{p'} \Bigr\rceil, 
				\qquad 
				\ell\in\bbN, 
			\]
			the (single-level) 
			Monte Carlo error estimate \eqref{eq:dim-dep_SLMC-conv} holds. 
	\end{enumerate} 
\end{corollary} 

In order to formulate a complexity result 
for the single-level Monte Carlo method,	
we make the following assumption about the 
growth of the type constants with respect 
to the dimension of the subspaces $(E_\ell)_{\ell\in \bbN}$.

\begin{assumption}[Growth of type constants]
\label{ass:constants_subspaces} 
	Let $(E,\norm{\,\cdot\,}{E})$ 
	be of Rademacher type $p\in[1,2]$,  
	$(E_\ell)_{\ell\in\bbN}$ be a family of subspaces 
	of $(E,\norm{\,\cdot\,}{E})$ with dimensions 
	$N_\ell := \dim(E_\ell)<\infty$, $\ell\in\bbN$. 
    There exist constants $a_0, a_1 \in[0,\infty)$ 
    and $C_\tau \in[1,\infty)$ such that 
	\begin{equation}\label{eq:tau-poly}
				\sup_{\ell\in\bbN} 
				\sup_{r\in[p,2]\cap(1,2]}  
				\Bigl( N_\ell^{-(a_0 - a_1 r')}
				[\tau_r(E_\ell)]^{r'\!} \Bigr) 
				\leq C_\tau. 
	\end{equation}
\end{assumption}

Under Assumption \ref{ass:constants_subspaces}, 
for every $\ell\in\bbN$,  
the function $f_{\alpha;\ell}\from(1,2]\to[0,\infty)$ 
in \eqref{eq:def:fl}  
admits the upper bound 
\[
	\forall r\in[p,2]\cap(1,2]: 
	\quad 
	f_{\alpha;\ell}(r) 
	\leq 
	C_\tau 
	N_\ell^{a_0 + (\alpha - a_1) r' \!}, 
\]
which is monotonically decreasing in $r$ 
if $\alpha\geq a_1$ and 
strictly monotonically increasing
in $r$ if $\alpha < a_1$. 
This motivates to choose the number 
of Monte Carlo samples for a 
single-level Monte Carlo method 
in the setting of Corollary~\ref{cor:dim-dep-MC-1}
as 
\begin{equation}\label{eq:explicit_choice_Mell_SLMC}
	M_\ell^{\rm opt} 
	:= 
	\begin{cases} 
		\Bigl\lceil N_\ell^{ a_0 +  \bar{q}' \! ( \alpha  - a_1 ) } \Bigr\rceil 
		&\text{ if } 
		\alpha \geq a_1, 
		\\
		\bigl\lceil N_\ell^{ a_0 +  p' \! ( \alpha  - a_1 ) } \bigr\rceil 
		&\text{ if } 
		\alpha < a_1 \text{ and  $p\in(1,2]$ is the type of $E$}.
	\end{cases} 
\end{equation}

\begin{example}\label{ex:choice_Mell_SLMC}
	Consider again the Banach space $E:=\ell^s$ which is of Rademacher type
	$p = \min\{s,2\}$ if $s\in[1,\infty)$ 
	and of type $p=1$ for $s=\infty$. 
	Let $\eta\in L^1(\Omega;\ell^s)$, $q\in(1,\infty)\cap[p,\infty)$, 
	$\bar{q}:=\min\{q,2\}$ 
	and, for every $\ell\in\bbN$, 
	let $N_\ell\in\bbN$ and  
	$(\xi_{\ell,j} )_{j\in\bbN}$ 
	be a sequence of i.i.d.\ random variables 
	in $L^q(\Omega; \ell^s_{N_\ell})$ 
	such that \eqref{eq:alpha-stab}~holds. 
	\begin{enumerate}[label={\normalfont(\alph*)}, leftmargin=7mm]
		\item 
			For $s\in[1,2)$, we have the relation  
			$\tau_r(\ell^s_{N_\ell}) 
			= 
			N_\ell^{\frac{1}{s} - \frac{1}{r} \!} 
			= 
			N_\ell^{\frac{1}{r'} - \frac{1}{s'} }$ 
			for every $\ell\in\bbN$ and all $r\in[s,2]$. 
			Thus, 
			the finite-dimensional spaces 
			$(\ell^s_{N_\ell})_{\ell\in\bbN}$ 
			satisfy~\eqref{eq:tau-poly} 
			with $a_0 = 1$, $a_1 = \frac{1}{s'}$ and 
			$C_\tau=1$, and 
			$f_{\alpha;\ell}(r) 
			= [\tau_r(\ell^s_{N_\ell}) N_\ell^{\alpha }]^{r'\!} 
			= 
			N_\ell^{1 + \left( \alpha  - \frac{1}{s'} \right) r' 
			\!}$ 
			is 
			monotonically decreasing in $r$ if $\alpha \geq \frac{1}{s'}$ 
			and strictly monotonically increasing in $r$ 
			if $\alpha < \frac{1}{s'}$.  
			We thus pick the number of Monte Carlo samples 
			as 
			\[		
				M_\ell^{\rm opt} := 
				\begin{cases} 
					\bigl\lceil [\tau_{\bar{q}}(\ell_{N_\ell}^s) N_\ell^\alpha]^{\bar{q}'} \bigr\rceil  
					= 
					\Bigl\lceil N_\ell^{ 1 +  \bar{q}' \! \left( \alpha  - \frac{1}{s'} \right) } \Bigr\rceil 
					&\text{ if } 
					\alpha \geq \frac{1}{s'}, 
					\\
					\bigl\lceil [\tau_s(\ell_{N_\ell}^s) N_\ell^\alpha]^{s'} \bigr\rceil 
					= 
					\bigl\lceil N_\ell^{ \alpha s' } \bigr\rceil 
					&\text{ if } 
					\alpha < \frac{1}{s'}.  
				\end{cases} 
			\]
			Note that, for $s=1$, 
			we obtain 
			$\tfrac{1}{s'} = 0$ and we 
			are thus in the first case 
			which simplifies to 
			$\bigl\lceil N_\ell^{ 1 +  \alpha \bar{q}'} \bigr\rceil$. 
		\item 
			For $s\in[2,\infty)$, 
			we have 
			$p = r = \bar{q} = 2$, 
			\eqref{eq:tau-poly} trivially holds 
			for $a_0 = a_1 = 0$ and 
			$C_\tau=\tau_2(\ell^s)^2 = B_s^2$ 
			(see Example~\ref{ex:Ls-type}),   
			and we pick the number of Monte Carlo samples 
			as $M_\ell^{\rm opt} 
			:= 
			\bigl\lceil [\tau_2(\ell_{N_\ell}^s) N_\ell^\alpha]^2 \bigr\rceil 
			= 
			\bigl\lceil B_s^2 N_\ell^{2\alpha} \bigr\rceil$. 
		\item 
			For $s=\infty$, the type $r$ constant of 
			$\ell_{N_\ell}^\infty$ admits, 
			for every $\ell\in\bbN$ and all ${r\in(1,2]}$, 
			the bounds 
			(see \cite[Proposition~7.1.7]{AnalysisInBanachSpacesII2017}) 
			\[
				\frac{1}{2} ( \lfloor \log_3 N_\ell \rfloor + 1)^{\frac{1}{r'}} 
				\leq 
				\tau_r(\ell_{N_\ell}^\infty) 
				\leq 
				( \max\{ 2 e \log N_\ell, 2\} )^{\frac{1}{r'} \!} . 
			\]
			Thus, 
			the condition \eqref{eq:tau-poly} is fulfilled 
			for $a_1=0$ 
			and any $a_0\in(0,\infty)$, 
			where the magnitude of 
			$C_\tau\in[1,\infty)$ depends on 
			the choice of $a_0\in(0,\infty)$ 
			(e.g., $C_\tau = 2$ for $a_0=1$).   
			As  
			$f_{\alpha;\ell}(r) = [\tau_r(\ell^\infty_{N_\ell}) N_\ell^{\alpha }]^{r'}$ 
			is monotonically decreasing in~$r$, 
			we pick, for $N_\ell\geq 2$, the number of Monte Carlo samples 
			as 
			$M_\ell^{\rm opt} := \bigl\lceil 2 e (\log N_\ell) 
			N_\ell^{\alpha\bar{q}'}\bigr\rceil$. 
	\end{enumerate} 
\end{example} 

The above discussion leads to the 
following complexity result 
for a single-level Monte Carlo estimator.

\begin{theorem}\label{thm:complexity_SLMC_dim_dep} 
	Assume that  
	$(E,\norm{\,\cdot\,}{E})$ 
	is of Rademacher type $p\in[1,2]$, 
	let ${X\in L^1(\Omega;E)}$, 
	$q\in(1,\infty)\cap[p,\infty)$, 
	and set  
	$\bar{q}:=\min\{q,2\}$.
	Let $(E_\ell)_{\ell\in \bbN}$ be a 
	family of subspaces of $E$ 
	of 
	dimension 
	$N_\ell := \dim(E_\ell) <\infty$ 
	satisfying Assumption~\ref{ass:constants_subspaces}. 
	For every $\ell\in\bbN$, 
	let $( X_{\ell,j} )_{j\in\bbN}$ 
	be a sequence of independent 
	copies of 
	an $E_\ell$-valued random variable 
	$X_\ell\in L^q(\Omega;E_\ell)$ such that 
	\[	
		\sup_{\ell\in\bbN} 
		N_\ell^\alpha 
		\bigl\| \bbE[ X ] - \bbE[ X_{\ell} ] \bigr\|_{E}  
		\leq 
		C_\alpha , 
		\qquad 
		\sup_{\ell\in\bbN} \, 
		\bigl\| X_{\ell} - \bbE[X_{\ell}] \bigr\|_{L^q(\Omega; E)} 
		\leq 
		C_{\sf stab}, 
	\]	
	for some constants~$\alpha, C_{\alpha}, C_{\sf stab}\in (0,\infty)$. 
	In addition, for all $\ell \in \bbN$, 
	let $\cC_\ell$ denote the cost 
	(number of floating point operations) 
	to generate one sample of $X_\ell$ 
	and assume that there exist 
	constants $\gamma, C_{\gamma}\in (0,\infty)$ 
	and $A\in (1,\infty)$ such that 
	$\cC_\ell\leq C_\gamma N_\ell^\gamma$ 
	and $N_\ell \eqsim A^\ell$.
	Then, for every $\epsilon\in(0,\nicefrac{1}{2}]$, 
	there exist integers $L\in\bbN$ and  
	$M_L \in\bbN$ such 
	that the $L^q$-accuracy $\epsilon$ of 
	the single-level Monte Carlo 
	estimator for $\bbE[X]$, 
	\[ 
		\mathrm{err}^{\sf SL}_{q}(X)
		:=  
		\biggl\| 
		\bbE[X] 
		-
		\frac{1}{M_L} 
		\sum_{j=1}^{M_L}  
		X_{L,j}
		\biggr\|_{L^q(\Omega; E)} 
		<\epsilon, 
	\]
	can be achieved at a computational cost of order
	\begin{equation}\label{eq:complexity_SLMC_dim_dep}
		\cC^{\sf SL}_{q}(X) 
		\lesssim
		\begin{cases} 
				\epsilon^{-\frac{\gamma}{\alpha}} +\epsilon^{-\bar{q}'-\frac{\gamma+a_0-a_1\bar{q}'}{\alpha}}  & 
				\text{ if }  
				\alpha \geq a_1, 
				\\
				\epsilon^{-\frac{\gamma}{\alpha}} 
				+
				\epsilon^{-p'-\frac{\gamma+a_0-a_1 p'}{\alpha}} 
				 & 
				\text{ if }  
				\alpha <  a_1 
				\text{ and  $p\in(1,2]$}. 	
				\end{cases}
	\end{equation}
\end{theorem} 

\begin{proof}
	By assumption we may choose 
	$L\in\bbN$ 
	as the smallest integer 
	such that 
	$N_L^{-\alpha} 
	\bigl( C_\alpha + 2 K_{q,1} C_{\sf stab} C_\tau^{1/2} \bigr) < \epsilon$. 
	Defining $M_L$ as 
	in \eqref{eq:explicit_choice_Mell_SLMC} 
	with $\ell:=L$ 
	then yields 
	\[
		\mathrm{err}^{\sf SL}_{q}(X)< \epsilon, 
	\]
	cf.~\eqref{eq:dim-dep-estimate_SLMC}. 
	The cost of this estimator can be bounded as
	\begin{equation}\label{eq:cost-SL-1-bound} 
	\begin{split}  
		\cC^{\sf SL}_{q}(X) 
		\eqsim 
		\cC_L M_L  
		&\leq 
		\begin{cases} 
				C_\gamma N_L^\gamma 
				\Bigl\lceil N_L^{a_0+\bar{q}^\prime(\alpha-a_1)} \Bigr\rceil  
				& 
				\text{ if }  
				\alpha \geq a_1, 
				\\[6pt]
				C_\gamma N_L^\gamma 
				\Bigl\lceil N_L^{a_0+ p^\prime(\alpha-a_1)} \Bigr\rceil  
				& 
				\text{ if }  
				\alpha <  a_1 
				\text{ and  $p\in(1,2]$}, 
				\end{cases}
		\\
		&\lesssim_{(\alpha,\gamma,q,A)}
		\begin{cases} 
				\epsilon^{-\frac{\gamma}{\alpha}} +\epsilon^{-\bar{q}'-\frac{\gamma+a_0-a_1\bar{q}'}{\alpha}}  
				& 
				\text{ if }   
				\alpha \geq a_1, 
				\\
				\epsilon^{-\frac{\gamma}{\alpha}} +
				\epsilon^{- p'-\frac{\gamma+a_0-a_1 p'}{\alpha}} 
				& 
				\text{ if }   
				\alpha <  a_1 
				\text{ and  $p\in(1,2]$}, \hspace*{-5mm} 			
		\end{cases}
	\end{split} 
	\end{equation}
	which concludes the proof.
\end{proof}

\begin{remark}
	The complexity bound 
	of Theorem~\ref{thm:complexity_SLMC_dim_dep} 
	identifies two different cases: 
	If the rate~$\alpha$ is sufficiently large, i.e., 
	$\alpha \geq a_1$, 
	then the single-level Monte Carlo error achieves 
	an error of size~$\epsilon$ with an asymptotic cost 
	that is \emph{independent} of the type~$p$ of~$E$, 
	and only determined by the integrability parameter $\bar{q}$, 
	and constants~$a_0$, $a_1$. 
	Conversely, if $\alpha < a_1$, 
	achieving the desired accuracy 
	requires a higher-dimensional subspace $E_L\subseteq E$, 
	and then the asymptotic complexity 
	does depend on the type $p$.
	
	Furthermore, note that if, in addition to 
	Assumption~\ref{ass:constants_subspaces}, 
	one assumes that the constants $a_0,a_1$ 
	satisfy $a_0 - a_1 p' = 0$ 
	(which, e.g., is satisfied by the spaces 
	$\ell^s$ for all $s\in (1,\infty)$), 
	then the second case of 
	\eqref{eq:complexity_SLMC_dim_dep} 
	reduces to 
	$\cC^{\sf SL}_{q}(X) 
	\lesssim \epsilon^{-p'-\frac{\gamma}{\alpha}}$ 
	which is the same asymptotic complexity obtained 
	in a Monte Carlo error analysis that does not account for 
	dimension-dependent type constants \cite{KKChS2024}.
\end{remark}

\subsection{Multilevel Monte Carlo methods}\label{sec:MLMC}
In this subsection we focus on the analysis of 
multilevel Monte Carlo methods.
To this end, we make an additional
assumption on the  
finite-dimensional subspaces $(E_\ell)_{\ell\in\bbN}$. 

\begin{assumption}[Nestedness]
\label{ass:subspaces-nested} 
	The family $(E_\ell)_{\ell\in\bbN}$ of 
	subspaces  
	of $(E,\norm{\,\cdot\,}{E})$ 
	is nested, 
	\[
		E_1\subseteq E_2 \subseteq \ldots 
		\subseteq E_\ell \subseteq \ldots \subseteq E . 
	\]
\end{assumption}

\begin{proposition}\label{prop:dim-dep-MLMC-1}
	Suppose that $(E_\ell)_{\ell\in\bbN}$ 
	are finite-dimensional subspaces 
	of~$E$ satisfying  
	Assumption~\ref{ass:subspaces-nested}. 
	Let ${X\in L^1(\Omega;E)}$, 
	$q\in[1,\infty)$,  
	and define ${\bar{q}:=\min\{q,2\}}$. 
	Assume further that $L\in\bbN$ and, 
	for every $\ell\in\{1,\ldots,L\}$, 
	${X_\ell \in L^q(\Omega;E_\ell)}$,  
	$M_\ell\in\bbN$, and 
	$\xi_{\ell,1},\ldots,\xi_{\ell,M_\ell}$ 
	are independent 
	copies of the $E_\ell$-valued 
	random variable 
	\begin{equation}\label{eq:def:xi-ell} 
		\xi_\ell 
		:= 
		X_\ell - X_{\ell-1} \in L^q ( \Omega; E_\ell ), 
		\qquad 
		X_0 := 0\in E_1. 
	\end{equation}
	Then, we have, for all $r\in[1,\bar{q}]$, 
	\begin{align*} 
		\biggl\| 
		&\bbE[X]
		-
		\sum_{\ell = 1}^{L} 
		\frac{1}{M_\ell} 
		\sum_{j=1}^{M_\ell}  
		\xi_{\ell,j}
		\biggr\|_{L^q(\Omega; E)} 
		\leq 
		\bigl\| 
		\bbE[X] - \bbE[X_L ] 
		\bigr\|_{E}
		\\
		&\qquad\quad +   
		2 K_{q,r} 
		\sum_{\ell=1}^L    
		\Bigl[ 
		\tau_r(E_\ell) 
		M_\ell^{-\left( 1 - \frac{1}{r} \right)} 
		\bigl\| X_\ell - \bbE[X_\ell] 
		- ( X_{\ell-1} - \bbE[X_{\ell-1}] ) \bigr\|_{L^q(\Omega;E)} 
		\Bigr] .  
	\end{align*} 
\end{proposition} 

\begin{proof} 
	First note that, for every $\ell\in\{1,\ldots,L\}$ 
	the random variables 
	$\xi_{\ell,1},\ldots,\xi_{\ell,M_\ell}$ 
	are identically distributed 
	and we have that 
	\[ 
		\bbE\Biggl[ 
		\sum_{\ell = 1}^{L} 
		\frac{1}{M_\ell} 
		\sum_{j=1}^{M_\ell}  
		\xi_{\ell,j} \Biggr] 
		= 
		\sum_{\ell = 1}^{L} 
		\bbE[ \xi_\ell ] 
		=
		\sum_{\ell = 1}^{L} 
		\bigl( \bbE[ X_\ell ] - \bbE[ X_{\ell-1} ] \bigr) 
		= 
		\bbE[X_L] . 
	\] 
	Thus, we find by the triangle inequality 
	on $L^q(\Omega; E)$ that 
	\[ 
		\biggl\| 
		\bbE[X] 
		-
		\sum_{\ell = 1}^{L} 
		\frac{1}{M_\ell} 
		\sum_{j=1}^{M_\ell}  
		\xi_{\ell,j}
		\biggr\|_{L^q(\Omega; E)} 
		\leq 
		\bigl\| 
		\bbE[X - X_L ] 
		\bigr\|_{E} 
		+ 
		\sum_{\ell=1}^L 
		\biggl\| 
		\bbE[\xi_\ell ] - 
		\frac{1}{M_\ell} 
		\sum_{j=1}^{M_\ell}  
		\xi_{\ell,j}
		\biggr\|_{L^q(\Omega; E)} \!. 
	\] 
	By Proposition~\ref{prop:dim-dep-MC-1} above we obtain, 
	for every $\ell\in\{1,\ldots,L\}$ and all $r\in[1,\bar{q}]$, that
	\[
		\biggl\| 
		\bbE[\xi_\ell ] - 
		\frac{1}{M_\ell} 
		\sum_{j=1}^{M_\ell}  
		\xi_{\ell,j}
		\biggr\|_{L^q(\Omega; E)} 
		\leq 
		2 K_{q,r} \tau_r(E_\ell) 
		M_\ell^{-\left(1-\frac{1}{r}\right)} 
		\bigl\| \xi_{\ell} - \bbE[\xi_{\ell}] \bigr\|_{L^q(\Omega;E)},
	\]
	which leads to the claimed estimate.
\end{proof}

\begin{theorem}\label{thm:dim-dep-alpha-beta-gamma-1} 
	Suppose that 
	$(E,\norm{\,\cdot\,}{E})$ 
	is of Rademacher type $p\in[1,2]$ and 
	let $(E_\ell)_{\ell\in \bbN}$ be a 
	family of subspaces of $E$ 
	of 
	dimension 
	$N_\ell := \dim(E_\ell) <\infty$ 
	satisfying 
	Assumptions~\ref{ass:constants_subspaces} 
	and~\ref{ass:subspaces-nested}. 
	Let 
	${X \in L^1(\Omega; E)}$, 
	$q\in(1,\infty)\cap [p,\infty)$, and 
	define ${\bar{q}:=\min\{q,2\}}$. 
	For all ${\ell\in\bbN}$,
	let 
	$X_\ell \in L^q(\Omega;E_\ell)$ 
	be an 
	$E_\ell$-valued random variable 
	and define   
	$\xi_\ell  \in L^q ( \Omega; E_\ell )$ 
	as in \eqref{eq:def:xi-ell}. 
	For every $\ell\in\bbN$, 
	let $\cC_\ell$ denote the cost 
	(number of floating point operations)
	to generate one sample of 
	the random variable $\xi_\ell$ 
	in~\eqref{eq:def:xi-ell}, 
	and suppose that 
	there exist constants 
	${\alpha,\beta,\gamma,C_\alpha,C_\beta,C_\gamma\in(0,\infty)}$ 
	and 
	$A\in(1,\infty)$ such that  
	$N_\ell \eqsim A^\ell$ 
	for every $\ell\in\bbN$ and, moreover, 
	\begin{align} 
		\quad 
		\forall \ell\in\bbN : 
		&&
		\bigl\| \bbE[X] - \bbE[X_\ell] \bigr\|_E   
		&\leq 
		C_\alpha N_\ell^{-\alpha} \!,
		\qquad
		\tag{$\alpha$} 
		\label{eq:ass:alpha-dim-dep-1} 
		\\
		\quad
		\forall \ell\in\bbN : 
		&&
		\| X_\ell - \bbE[X_\ell] - 
			(X_{\ell-1} - \bbE[X_{\ell-1}]) \|_{L^q(\Omega;E)} 
		&\leq 
		C_\beta N_\ell^{-\beta} \!, 
		\qquad 
		\tag{$\beta$} 
		\label{eq:ass:beta-dim-dep-1} 
		\\
		\quad
		\forall \ell\in\bbN : 
		&&
		\cC_\ell 
		&\leq 
		C_\gamma N_\ell^\gamma . 
		\qquad 
		\tag{$\gamma$} 
		\label{eq:ass:gamma-dim-dep-1} 
	\end{align} 
	Furthermore, for each $\ell\in\bbN$, 
	let $( \xi_{\ell,j} )_{j\in\bbN} \subset L^q ( \Omega; E_\ell )$ 
	be a sequence of independent 
	copies of 
	the $E_\ell$-valued random variable 
	$\xi_\ell$ in \eqref{eq:def:xi-ell}. 
	
	Then, for every $\epsilon\in(0,\nicefrac{1}{2}]$, 
	there exist integers 
	$L\in\bbN$ 
	and 
	$M_1,\ldots,M_L \in \bbN$ such 
	that the $L^q$-accuracy $\epsilon$ of 
	the multilevel Monte Carlo 
	estimator for $\bbE[X]$, 
	\begin{equation}\label{eq:MLMC-err-dim-dep-1} 
		\mathrm{err}^{\sf ML}_{q}(X)
		:=  
		\biggl\| 
		\bbE[X] 
		-
		\sum_{\ell = 1}^{L} 
		\frac{1}{M_\ell} 
		\sum_{j=1}^{M_\ell}  
		\xi_{\ell,j}
		\biggr\|_{L^q(\Omega; E)} 
		<\epsilon ,  
	\end{equation} 
	can be achieved at a computational 
	cost of order 
	\begin{equation}\label{eq:MLMC-cost-dim-dep-1}  
		\cC^{\sf ML}_{q}(X) 
		\lesssim
		\begin{cases} 
		\epsilon^{-\frac{\gamma}{\alpha}} 
		+ 
		\epsilon^{-\bar{q}'} 
		&\text{ if } 
		t<\bar{q}'\!, 
		\\[2pt]
		\epsilon^{-\frac{\gamma}{\alpha}} 
		+ 
		\epsilon^{-\bar{q}'} 
		|\log_A \epsilon|^{\bar{q}'+1} 
		&\text{ if }  
		t = \bar{q}'\!,
		\\[2pt]
		\epsilon^{-\frac{\gamma}{\alpha}} 
		+ 
		\epsilon^{-\bar{q}' - \frac{\gamma + a_0 - (\beta+a_1) \bar{q}' \!}{\alpha}} 
		&\text{ if } 
		t \in(\bar{q}'\!,p'] 
		\;\; \text{and} \;\; 
		\alpha \geq \beta+a_1, 
		\\[2pt]
		\epsilon^{-\frac{\gamma}{\alpha}} 
		+ 
		\epsilon^{-t} 
		|\log_A \epsilon|^{t+1} 
		&\text{ if } 
		t \in(\bar{q}'\!,p'] 
		\;\; \text{and} \;\; 
		\alpha < \beta+a_1, 
		\\[2pt]
		\epsilon^{-\frac{\gamma}{\alpha}} 
		+ 
		\epsilon^{-\bar{q}' - \frac{\gamma + a_0 - (\beta+a_1) \bar{q}' \!}{\alpha}} 
		&\text{ if }  
		t > p'\! \;\;\text{and}\;\; 
		\alpha \geq \beta+a_1, 
		\\[2pt]
		\epsilon^{-\frac{\gamma}{\alpha}} 
		+ 
		\epsilon^{-p' - \frac{\gamma + a_0 - (\beta+a_1) p' \!}{\alpha}} 
		&\text{ if }  
		t > p'\! \;\;\text{and}\;\; 
		\alpha < \beta+a_1, 
		\end{cases}
	\end{equation} 
	where $t := \frac{\gamma+a_0}{\beta+a_1}$. 
\end{theorem}

\begin{proof} 
	We will show by explicit construction that, for every $\epsilon\in(0,\nicefrac{1}{2}]$, 
	assumptions \eqref{eq:ass:alpha-dim-dep-1}, \eqref{eq:ass:beta-dim-dep-1} 
	and \eqref{eq:ass:gamma-dim-dep-1}  
	allow to choose 
	the algorithmic steering parameters $L\in\bbN$ 
	and $M_1,\ldots,M_L\in\bbN$ so that 
	\eqref{eq:MLMC-err-dim-dep-1} holds 
	with cost \eqref{eq:MLMC-cost-dim-dep-1}.  
	
	To this end, let 
	$r\in R(p,q) := [p,q]\cap (1,2]$. 
	By Proposition~\ref{prop:dim-dep-MLMC-1} 
	as well as 
	Assumption~\ref{ass:constants_subspaces}, and 
	conditions \eqref{eq:ass:alpha-dim-dep-1}, 
	\eqref{eq:ass:beta-dim-dep-1},   
	we obtain 
	the estimate 
	\begin{align}
		\mathrm{err}^{\sf ML}_{q}(X)
		&\leq 
		\mathrm{err}^{\sf ML}_{q;r}(X)
		:= 
		\bigl\| \bbE[X] - \bbE[X_L] \bigr\|_E   
		\notag 
		\\
		&\quad +   
		2 K_{q,r} 
		\sum_{\ell=1}^L    
		\Bigl[ 
		\tau_r(E_\ell) 
		M_\ell^{-\left( 1 - \frac{1}{r} \right)} 
		\bigl\| X_\ell - \bbE[X_\ell] 
		- ( X_{\ell-1} - \bbE[X_{\ell-1}] ) \bigr\|_{L^q(\Omega;E)} 
		\Bigr] 
		\notag 
		\\
		&\leq C_\alpha N_L^{-\alpha}  
		+ 
		2 K_{q,r} C_\beta
		\sum_{\ell=1}^L   
		\Bigl[ 
		\tau_r(E_\ell) 
		M_\ell^{ - \frac{1}{r'}\! }  
		N_\ell^{-\beta} 
		\Bigr] 
		\notag 
		\\
		&\leq C_\alpha N_L^{-\alpha}  
		+ 
		2 K_{q,1} C_\beta C_\tau^{\frac{1}{r'}}
		\sum_{\ell=1}^L   
		\Bigl[ 
		M_\ell^{ - \frac{1}{r'}\! }  
		N_\ell^{\frac{a_0}{r'}\!} 
		N_\ell^{ -(\beta+a_1) } 
		\Bigr] . 
		\label{eq:bound_proof_complexity_MLMC}
	\end{align} 
	Choose $L\in\bbN$ as the smallest integer such that 
	$N_L^{-\alpha} 
	< 
	\bigl( C_\alpha + 
	2 K_{q,1} C_\beta C_\tau^{ 1/2} \bigr)^{-1} 
	\epsilon$ holds 
	and, for every $\ell\in\{1,\ldots,L\}$, let $M_\ell\in\bbN$ be 
	defined as 
	\[
		M_\ell 
		= 
		\biggl\lceil 
		N_L^{ \alpha r' }
		S_{L;r}^{r'}  
		N_\ell^{-\frac{ (\beta+a_1+\gamma) r'\! - a_0 }{r'+1}} 
		\biggr\rceil, 
		\qquad 
		\text{where} 
		\qquad 
		S_{L;r}
		:= 
		\sum_{\ell=1}^L 
		N_{\ell}^{\frac{ \gamma + a_0 - (\beta + a_1) r'\! }{r' +1}}. 
	\] 
	The magnitude of $S_{L;r}$ 
	behaves asymptotically (for $L$ large) 
	as 
	\begin{equation}\label{eq:theta-L} 
		S_{L;r} 
		\eqsim_{(\beta,\gamma,a_0,a_1,A,r)} 
		\begin{cases} 
		1 & 
		\text{ if } 
		(\beta+a_1) r' > \gamma + a_0, 
		\\
		L & 
		\text{ if } 
		(\beta+a_1) r' = \gamma + a_0, 
		\\
		N_L^{\frac{\gamma+a_0-(\beta+a_1) r' \! }{r'+1}} 
		&\text{ if } 
		(\beta+a_1) r' < \gamma + a_0. 
		\end{cases}
	\end{equation}
	For this choice of $L$ and 
	$M_1,\ldots,M_L$, we can bound the error 
	as follows, 
	\begin{align*} 
		\mathrm{err}^{\sf ML}_{q;r}(X)
		&\leq 
		C_\alpha 
		N_L^{-\alpha}  
		+ 
		2 K_{q,1} C_\beta  C_\tau^{\frac{1}{2}} 
		\sum_{\ell=1}^{L}   
		\Bigl[ 
		N_L^{-\alpha} 
		S_{L;r}^{-1} 
		N_\ell^{\frac{\beta+a_1+\gamma-a_0/r' \!}{r'+1}} 
		N_{\ell}^{\frac{a_0}{r'}} N_\ell^{-(\beta+a_1)} 
		\Bigr] 
		\\
		&= 
		N_L^{-\alpha}  
		\Biggl( 
		C_\alpha 
		+ 
		2 K_{q,1} C_\beta  C_\tau^{\frac{1}{2}} 
		S_{L;r}^{-1} 
		\sum_{\ell=1}^L 
		N_\ell^{\frac{\gamma + a_0 -(\beta+a_1) r' \! }{r'+1}}
		\Biggr) 
		<
		\epsilon . 
	\end{align*} 
	For the total cost, 
	we first compute 
	\begin{align*} 
		\cC^{\sf ML}_{q;r}(X) 
		\eqsim 
		\sum_{\ell=1}^L \cC_\ell  M_\ell 
		&\leq   
		C_\gamma 
		\sum_{\ell=1}^L N_\ell^\gamma 
		\biggl( 
		1 + 
		N_L^{ \alpha r' }
		S_{L;r}^{r'}  
		N_\ell^{-\frac{ (\beta+a_1+\gamma) r'\! - a_0 }{r'+1}} 
		\biggr) 
		\\
		&=  
		C_\gamma 
		\sum_{\ell=1}^L N_\ell^\gamma 
		+ 
		C_\gamma 
		N_L^{ \alpha r' }
		S_{L;r}^{r'}  
		\sum_{\ell=1}^L 
		N_\ell^{\frac{\gamma + a_0 - (\beta+a_1) r' \!}{r'+1}} 
		\\
		&= 
		C_\gamma 
		\sum_{\ell=1}^L N_\ell^\gamma 
		+ 
		C_\gamma 
		N_L^{ \alpha r' }
		S_{L;r}^{r'\!+1} \!.  
	\end{align*}
	By the choice of $L$ 
	we have  
	$A^L \eqsim  N_L \eqsim \epsilon^{-\nicefrac{1}{\alpha}}$ 
	and, since $\epsilon\in(0,\nicefrac{1}{2}]$, 
	we find that $L \eqsim_{\alpha} |\log_A \epsilon|$. 
	Thus, using \eqref{eq:theta-L} 
	we conclude that the computational cost satisfies
	\begin{align*} 
		\cC^{\sf ML}_{q;r}(X) 
		\lesssim_{(\beta,\gamma,a_0,a_1,A,r,q)}  
		\begin{cases} 
		N_L^\gamma 
		+ 
		N_L^{\alpha r'\!}  
		&\text{ if } 
		(\beta+a_1) r' > \gamma + a_0,  
		\\
		N_L^\gamma 
		+ 
		N_L^{\alpha r' \!} 
		L^{r'+1}  
		&\text{ if }  
		(\beta+a_1) r' = \gamma + a_0, 
		\\
		N_L^\gamma 
		+ 
		N_L^{\alpha r'\!+\gamma+a_0-(\beta+a_1) r' \!}  
		&\text{ if }  
		(\beta+a_1) r' < \gamma + a_0, 
		\end{cases}
	\end{align*} 
	and behaves in terms of the accuracy 
	$\epsilon$ as follows, 
	\begin{align*} 
		&\cC^{\sf ML}_{q;r}(X) 
		\lesssim_{(\alpha,\beta,\gamma,a_0,a_1,A,r,q)}  
		\begin{cases} 
		\epsilon^{-\frac{\gamma}{\alpha}} 
		+ 
		\epsilon^{-r'} 
		&\text{ if } 
		r' > t, 
		\\ 
		\epsilon^{-\frac{\gamma}{\alpha}} 
		+ 
		\epsilon^{-r'} 
		|\log_A \epsilon|^{r'+1} 
		&\text{ if }  
		r' = t, 
		\\
		\epsilon^{-\frac{\gamma}{\alpha}} 
		+ 
		\epsilon^{-r' - \frac{\gamma + a_0 - (\beta+a_1) r' \!}{\alpha}} 
		&\text{ if }  
		r' < t, 
		\end{cases}
	\end{align*} 
	where $t := \frac{\gamma+a_0}{\beta+a_1}$. 
	Finally, since the preceding error and cost bounds 
	hold for all $r\in R(p,q) = [p,q]\cap (1,2]$, we
	can minimize the costs 
	by choosing 
	$r'\! \in [q'\!,p']\cap [2,\infty)$ 
	and, thus, 
	$r\in  R(p,q)$ 
	as follows: 
	\begin{align*}
		r' 
		= 
		\begin{cases} 
		\bar{q}' 
		&\text{ if } 
		t \leq \bar{q}'\!, 
		\\ 
		\bar{q}' 
		&\text{ if } 
		t \in(\bar{q}'\!,p'] 
		\;\;\text{and}\;\; 
		\alpha \geq \beta+a_1,  
		\\ 
		t
		&\text{ if } 
		t \in(\bar{q}'\!,p'] 
		\;\;\text{and}\;\; 
		\alpha < \beta+a_1, 
		\\
		\bar{q}' 
		&\text{ if } 
		t >p'\! \;\; \text{and}\;\; 
		\alpha \geq \beta + a_1, 
		\\ 
		p'  
		&\text{ if } 
		t >p'\! \;\; \text{and}\;\; 
		\alpha < \beta + a_1. 
		\end{cases} 
	\end{align*}
	This choice yields the cost bound 
	\eqref{eq:MLMC-cost-dim-dep-1} 
	which completes the proof. 
\end{proof} 

\subsection{Extension to \texorpdfstring{$k$th}{kth} moments}
\label{sec:SL_MLMC_kmoments}
The analysis presented in 
Subsections \ref{sec:SLMC} and~\ref{sec:MLMC} 
can straightforwardly be generalized 
to the estimation of $k$th injective moments. 
The technical tools are slight generalizations 
of the error estimates 
derived in \cite{KKChS2024}, 
which are explicit in the dependence on the type constant. 

The following error estimate 
plays for $k$th moments the same role as  
Proposition~\ref{prop:dim-dep-MC-1} for the first moment.

\begin{proposition}\label{prop:error_estimate_kthmoment}  
	Let $q\in [1,\infty)$, 
	$k\in\bbN$, 
	and set $\bar{q}:=\min\{q,2\}$. 
	For every $\ell\in\bbN$, 
	let $M_\ell\in\bbN$ and suppose that 
	$\xi_{\ell,1},\ldots,\xi_{\ell,M_\ell}\in L^{kq}(\Omega;E_\ell)$ are  
	independent and identically distributed 
	random variables taking values in a finite-dimensional 
	subspace $E_\ell\subseteq E$.  
	Then, 
	for every $\ell\in\bbN$, 
	for  all $r\in[1,\bar{q}]$,  
	and for any $U\in \otimes_{\varepsilon_s}^{k,s} E$,
	\begin{align*} 
		\biggl\| 
		U -\frac{1}{M_\ell} \sum_{j=1}^{M_\ell} \otimes^k \xi_{\ell,j} 
		\biggr\|_{L^q(\Omega;\otimes_{\varepsilon_s}^{k,s} E)} 
		&\leq 
		\bigl\|
		U-\bbM_{\varepsilon}^k [\xi_{\ell,1} ] 
		\bigr\|_{\varepsilon_s}\\ 
		&\quad + 
		C_{q,k}^{\sf SL} 
		\tau_r(E_{\ell}) 
		M_\ell^{-\left(1-\frac{1}{r}\right)} 
		\|\xi_{\ell,1}\|^k_{L^{kq}(\Omega;E)}, 
	\end{align*}
	where $C_{q,k}^{\sf SL}:=2(2k K_{q,1}+B_q)$, 
	$B_q, K_{q,1} \in (0,\infty)$ being the 
	Khintchine 
	and Kahane--Khintchine 
	constants 
	from Definitions~\ref{def:khintchine} 
	and~\ref{def:kahane-khintchine}, respectively. 
\end{proposition} 

\begin{proof}
	By the triangle inequality we obtain that 
	\begin{align*} 
		\biggl\| 
		U -\frac{1}{M_\ell} \sum_{j=1}^{M_\ell} \otimes^k \xi_{\ell,j} 
		\biggr\|_{L^q(\Omega;\otimes_{\varepsilon_s}^{k,s} E)} 
		&\leq 
		\bigl\|
		U-\bbM_{\varepsilon}^k [\xi_{\ell,1} ] 
		\bigr\|_{\varepsilon_s}\\ 
		&\quad+ 
		\biggl\|
		\bbM^k_{\varepsilon}\left[\xi_{\ell,1}\right] 
		-\frac{1}{M_\ell} \sum_{j=1}^{M_\ell} \otimes^k \xi_{\ell,j}
		\biggr\|_{L^q(\Omega;\otimes_{\varepsilon_s}^{k,s} E)} . 
	\end{align*}
	Since $E_\ell$ is of type $r$ for any $r\in [1,\bar{q}]$, 
	the second term on the right-hand side 
	can be bounded 
	by an application of \cite[Theorem 3.16]{KKChS2024}, 
	which yields the estimate 
	\[ 
		\biggl\|
		\bbM^k_{\varepsilon}\left[\xi_{\ell,1}\right] 
		-
		\frac{1}{M_\ell} \sum_{j=1}^{M_\ell} \otimes^k \xi_{\ell,j}
		\biggr\|_{L^q(\Omega;\otimes_{\varepsilon_s}^{k,s} E)} 
		\leq 2(2kK_{q,r}\tau_r(E_\ell)+B_q) 
		M_\ell^{-\frac{1}{r'}} 
		\|\xi_{\ell,1}\|^k_{L^{kq}(\Omega;E)}.
	\] 
	The claim follows then using that 
	$K_{q,r} \leq K_{q,1}$ 
	and $\tau_r(E_\ell)\geq 1$ for any $r\in[1,\bar{q}]$, 
	see \eqref{eq:bound_Kahane} and, e.g., 
	\cite[Section 7.1.a]{AnalysisInBanachSpacesII2017}.
\end{proof}

In the case that, in addition to the assumptions made in 
Proposition~\ref{prop:error_estimate_kthmoment}, 
there exist constants 
$\alpha,C_\alpha, C_{\sf stab}\in(0,\infty)$ 
such that 
\[ 
	\sup_{\ell\in\bbN} 
	N_\ell^\alpha 
	\bigl\| 
	U - \bbM_{\varepsilon}^k [\xi_{\ell,1} ] 
	\bigr\|_{\varepsilon_s}
	\leq 
	C_\alpha , 
	\qquad 
	\sup_{\ell\in\bbN} \, 
	\|\xi_{\ell,1}\|^k_{ L^{k q}(\Omega;E)}
	\leq 
	C_{\sf stab}, 
\] 
we obtain, for every $\ell\in \bbN$ 
and for all $r\in [1,\bar{q}]$, the estimate
\begin{equation}\label{eq:bound_kmoments}
	\biggl\| 
	U - \frac{1}{M_\ell}\sum_{j=1}^{M_\ell} \otimes^k \xi_{\ell,j}
	\biggr\|_{L^q(\Omega;\otimes_{\varepsilon_s}^{k,s} E)} 
	\leq 
	C_{\alpha}N_{\ell}^{-\alpha} 
	+ 
	C_{q,k}^{\sf SL} C_{\sf stab} 
	\tau_r(E_{\ell}) M_\ell^{-\left(1-\frac{1}{r} \right)},
\end{equation}
which has the same structure as  
\eqref{eq:dim-dep-estimate_SLMC}, 
modulo the constants. 
We thus derive the following complexity result 
for single-level Monte Carlo methods 
	for $k$th moments.

\begin{theorem}\label{thm:complexity_SLMC_kth} 
	Assume that  
	$(E,\norm{\,\cdot\,}{E})$ 
	is of Rademacher type $p\in[1,2]$, 
	let $k\in \bbN$, 
	$q\in(1,\infty)\cap[p,\infty)$,
	${X\in L^k(\Omega;E)}$, 
	and set  
	$\bar{q}:=\min\{q,2\}$.
	Let $(E_\ell)_{\ell \in \bbN}$ be a family 
	of subspaces of $E$ of dimensions 
	$N_\ell:=\dim(E_\ell)<\infty$ 
	satisfying Assumption~\ref{ass:constants_subspaces}. 
	For every $\ell\in\bbN$, 
	let $( X_{\ell,j} )_{j\in\bbN}$ 
	be a sequence of independent 
	copies of 
	an $E_\ell$-valued random variable 
	$X_\ell\in L^{kq}(\Omega;E_\ell)$ such that 
	\[	
		\sup_{\ell\in\bbN} 
		N_\ell^\alpha 
		\bigl\|\bbM_{\varepsilon}^k[X]
		-\bbM_{\varepsilon}^k [X_{\ell} ] \bigr\|_{\varepsilon_s}
		\leq 
		C_\alpha , 
		\qquad 
		\sup_{\ell\in\bbN} \, 
		\|X_{\ell} \|^k_{L^{kq}(\Omega;E)}
		\leq 
		C_{\sf stab},
	\]
	for some constants 
	$\alpha, C_\alpha,C_{\sf stab}\in (0,\infty)$.
	In addition, for every $\ell \in \bbN$, 
	assume that the cost $\cC_\ell$ 
	(number of floating points operations)
	to generate one sample  of $\otimes^k X_\ell$ 
	satisfies 
	$\cC_\ell\leq C_\gamma N_\ell^\gamma$, 
    for some $\gamma, C_{\gamma}\in (0,\infty)$, 
    and that there exists $A\in (1,\infty)$ 
    such that $N_\ell\eqsim A^\ell$.
	Then, for every $\epsilon\in(0,\nicefrac{1}{2}]$, 
	there exist integers $L\in\bbN$ and
	$M_L \in\bbN$ such 
	that the $L^q$-accuracy $\epsilon$ of 
	the single-level Monte Carlo 
	estimator for 
	$\bbM_{\varepsilon}^k [ X ]$, 
	\[ 
		\mathrm{err}^{k,{\sf SL}}_{q,\varepsilon_s}(X)
		:=  
		\biggl\| 
		\bbM_{\varepsilon}^k [ X ]
		-
		\frac{1}{M_L} 
		\sum_{j=1}^{M_L}  
		\otimes^k X_{L,j}
		\biggr\|_{L^q(\Omega;\otimes_{\varepsilon_s}^{k,s} E)} 
		<\epsilon , 
		\tag{$\epsilon$}
	\] 
	can be achieved at a computational cost of order
	\[
		\cC^{k,{\sf SL}}_{q,\varepsilon_s}(X) \lesssim
		\begin{cases} 
		\epsilon^{-\frac{\gamma}{\alpha}} +\epsilon^{-\bar{q}'-\frac{\gamma+a_0-a_1\bar{q}'}{\alpha}}  
		& 
		\text{ if }   
		\alpha \geq a_1, 
		\\
		\epsilon^{-\frac{\gamma}{\alpha}} + 
		 \epsilon^{-p'-\frac{\gamma+a_0-a_1 p'}{\alpha}}   
		 & 
		\text{ if }   
		\alpha <  a_1 \text{ and  $p\in(1,2]$.}	
		\end{cases}
	\] 
\end{theorem} 

\begin{proof}
	The proof is analogous to that of 
	Theorem~\ref{thm:complexity_SLMC_dim_dep}.
	Choosing~$L\in\bbN$ 
	as the smallest integer 
	such that
	$N_L^{-\alpha} 
	\bigl( C_\alpha +  C_{q,k}^{\sf SL} C_{\sf stab} C_\tau^{1/2} \bigr) 
	< \epsilon$ holds, 
	and $M_L$ as in \eqref{eq:explicit_choice_Mell_SLMC} 
	yields $\mathrm{err}^{k,{\sf SL}}_{q,\varepsilon_s}(X)< \epsilon$ 
	by  \eqref{eq:bound_kmoments}. 
	The cost of this estimator can then 
	be bounded as for that of the first moment 
	in \eqref{eq:cost-SL-1-bound}.   
\end{proof}
We conclude this section by deriving a complexity result 
for the multilevel Monte Carlo method for the estimation of 
$\bbM_{\varepsilon}^k [ X ]$. 
The next proposition adapts to our needs \cite[Proposition~3.23 and Theorem~3.24]{KKChS2024}.

\begin{proposition}\label{prop:error_estimate_MLMC_kth}
	Suppose that $(E_\ell)_{\ell\in\bbN}$ 
	are finite-dimensional subspaces 
	of~$E$ satisfying  
	Assumption~\ref{ass:subspaces-nested}. 
	Let $q\in [1,\infty)$, $k,L\in \bbN$  
	and set $\bar{q}:=\min\{q,2\}$. 
	For every $\ell\in\{1,\ldots,L\}$, let  
	$X_{\ell}\in L^{kq}(\Omega;E_{\ell})$, 
	$M_{\ell}\in \bbN$, 
	and $\xi_{\ell,1},\dots,\xi_{\ell,M_{\ell}}$ 
	be independent copies of the 
	$\otimes^{k,s}_{\varepsilon_s} E_{\ell}$-valued 
	random variable
	\begin{equation}\label{eq:def:xi-ell-k}
		\xi_\ell  
		:=
		\otimes^k X_{\ell} 
		- 
		\otimes^k X_{\ell-1}
		\in L^q(\Omega; \otimes^{k,s}_{\varepsilon_s} E_{\ell}),
		\qquad 
		X_0:=0\in E_{1}.
	\end{equation}
	In addition, assume that there exists a constant 
	$C_{\sf stab}\in (0,\infty)$ such that
	\begin{equation}\label{eq:stability_MLMC_kth}
		\forall \ell\in \bbN:
		\quad
		 \| X_\ell \|_{L^{kq}(\Omega;E)}\leq C_{{\sf stab}}.
	\end{equation} 
	
	Then, there exists a constant 
	$C_\star=C_\star(k,q,C_{\sf stab})\in (0,\infty)$
    such that, 
    for every $U\in \otimes^{k,s}_{\varepsilon_s} E$ 
    and all $r\in [1,\bar{q}]$,
	\begin{align*}
		\biggl\| 
		U-\sum_{\ell=1}^L \frac{1}{M_{\ell}}\sum_{j=1}^{M_\ell} \xi_{\ell,j} 
		\biggr\|_{L^q(\Omega;\otimes^{k,s}_{\varepsilon_s}E)}		
		&\leq 
		\bigl\| 
		U-\bbM^k_{\varepsilon} [X_L ] 
		\bigr\|_{\varepsilon_s}
		\\
		&\quad+
		C_\star 
		\sum_{\ell=1}^L 
		\tau_r(E_\ell)M_\ell^{-\left(1-\frac{1}{r}\right)}
		\|X_\ell-X_{\ell-1}\|_{L^{kq}(\Omega;E)}. 
	\end{align*}
\end{proposition}

\begin{proof}
	We first observe that 
	$\bbE\Bigl[
	\sum_{\ell=1}^L \frac{1}{M_{\ell}}
	\sum_{j=1}^{M_\ell} \xi_{\ell,j} 
	\Bigr] 
	= 
	\bbM^k_{\varepsilon} [X_L]$ 
	and use the triangle inequality on $L^q(\Omega;\otimes^{k,s}_{\varepsilon_s}E)$ 
	to obtain 
    \begin{equation}\label{eq:first_step_MLMC_kth_estimate}
    		\biggl\| U
    		-
    		\sum_{\ell=1}^L \frac{1}{M_{\ell}}\sum_{j=1}^{M_\ell} \xi_{\ell,j}
    		\biggr\|_{L^q(\Omega;\otimes^{k,s}	_{\varepsilon_s}E)}
    		\leq 
    		\bigl\| U-\bbM^k_{\varepsilon}[X_L]
    		\bigr\|_{\varepsilon_s}\\
    		+
    		\sum_{\ell=1}^L 
    		\mathrm{err}^{{\sf SL}}_{q,\varepsilon_s}(\xi_\ell), 
    \end{equation}
    where we define 
    \[
    	\mathrm{err}^{{\sf SL}}_{q,\varepsilon_s}(\xi_\ell) 
    	:= 
    	\biggl\|
    	\bbE [\xi_{\ell} ] - 
    	\frac{1}{M_{\ell}}\sum_{j=1}^{M_\ell} \xi_{\ell,j}
    	\biggr\|_{L^q(\Omega;  \otimes^{k,s}_{\varepsilon_s} E)}.
    \]
	Now, for every $\ell\in\left\{1,\dots,L\right\}$, 
	we use that $E_{\ell}$ 
	has any Rademacher type $r\in[1,2]$, 
	so that \cite[Proposition 3.23]{KKChS2024}
    leads for any $\ell\in\left\{1,\dots,L\right\}$ 
    and $r\in [1,\bar{q}]$,
	\begin{align*}  
			\mathrm{err}^{{\sf SL}}_{q,\varepsilon_s}(\xi_\ell) 
			&\leq 
			2 C^{{\sf diff}}_{q,k} 
			\tau_r(E_\ell) M_{\ell}^{-\left(1-\frac{1}{r}\right)} 
			\sum_{i=1}^k 
			\left[ \binom{k}{i}  
			\|X_{\ell}-X_{\ell-1}\|^i_{L^{kq}(\Omega;E)} 
			\|X_{\ell-1}\|^{k-i}_{L^{kq}(\Omega;E)}
			\right]\\
			&\;\; + 
			2 B_q 
			M_{\ell}^{-\frac{1}{2}} 
			\|X_{\ell}-X_{\ell-1}\|_{L^{k}(\Omega;E)}
			\sum_{i=0}^{k-1}
			\left[
			\|X_{\ell}\|^i_{L^k(\Omega;E)}
			\|X_{\ell-1}\|^{k-i-1}_{L^k(\Omega;E)}
			\right]
			\\
			&\leq 
			2 C^{{\sf diff}}_{q,k}
			\tau_r(E_\ell) M_\ell^{-\left(1-\frac{1}{r} \right)} 
 			\|X_\ell-X_{\ell-1}\|_{L^{kq}(\Omega;E)} 
 			\\
			&\;\;\cdot
			\left[\sum_{i=0}^{k-1} 
			\left( \binom{k}{i+1} 
			\|X_{\ell}-X_{\ell-1}\|^i_{L^{kq}(\Omega;E)} + 
			\|X_{\ell}\|^i_{L^k(\Omega;E)} 
			\right) 
			\|X_{\ell-1}\|_{L^{kq}(\Omega;E)}^{k-i-1}  
			\right] \!,  
	\end{align*} 
	where $C^{{\sf diff}}_{q,k}:=16 k \sqrt{\pi}K_{q,1}K_{q,2}$ 
	satisfies the relation $C^{{\sf diff}}_{q,k} > K_{q,2} 
	\geq B_q$, 
	and in the last step we used that 
	$r\in[1,\bar{q}]\subseteq [1,2]$ 
	to combine the two sums.
	Finally using the stability 
	assumption \eqref{eq:stability_MLMC_kth}, 
	we deduce the existence of a positive constant 
	$C_\star = C_\star(k,q,C_{\sf stab})\in (0,\infty)$ 
	such that for every $\ell\in\left\{1,\dots,L\right\}$ 
	and all $r\in [1,\bar{q}]$,
	\[
		\mathrm{err}^{{\sf SL}}_{q,\varepsilon_s}(\xi_\ell) 
		\leq 
		C_\star \tau_r(E_\ell)M_\ell^{-\left(1-\frac{1}{r} \right)}
		\|X_\ell-X_{\ell-1}\|_{L^{kq}(\Omega;E)}.
	\]
    Inserting this estimate into \eqref{eq:first_step_MLMC_kth_estimate} 
    concludes the proof.
\end{proof} 

\begin{theorem}
	Suppose that 
	$(E,\norm{\,\cdot\,}{E})$ 
	is of Rademacher type $p\in[1,2]$ and 
	let $(E_\ell)_{\ell\in \bbN}$ be a 
	family of subspaces of $E$ 
	of 
	dimension 
	$N_\ell := \dim(E_\ell) <\infty$ 
	satisfying 
	Assumptions~\ref{ass:constants_subspaces} 
	and~\ref{ass:subspaces-nested}. 
	Let $k\in \bbN$, 
	$q\in(1,\infty)\cap[p,\infty)$,
	${X\in L^k(\Omega;E)}$, 
	and set  
	$\bar{q}:=\min\{q,2\}$. 
	For all ${\ell\in\bbN}$,
	let 
	$X_\ell \in L^{kq}(\Omega;E_\ell)$ 
	be an 
	$E_\ell$-valued random variable 
	such that \eqref{eq:stability_MLMC_kth} holds, 
	and define   
	$\xi_\ell  
	\in L^q(\Omega;\otimes^{k,s}_{\varepsilon_s}E_{\ell})$ 
	as in 
	\eqref{eq:def:xi-ell-k}. 
	For every $\ell\in\bbN$, 
	let $\cC_\ell$ denote the cost 
	(number of floating point operations)
	to generate one sample of 
	the random variable~$\xi_\ell$ 
	in~\eqref{eq:def:xi-ell-k}, 
	and suppose that 
	there exist constants 
	${\alpha,\beta,\gamma,C_\alpha,C_\beta,C_\gamma\in(0,\infty)}$ 
	and 
	$A\in(1,\infty)$ such that  
	$N_\ell \eqsim A^\ell$ 
	for every $\ell\in\bbN$, and 
		\begin{align} 
		\quad 
		\forall \ell\in\bbN : 
		&&
		\bigl\| \bbM^k_{\varepsilon}[X] - \bbM^k_{\varepsilon}[X_\ell] \bigr\|_{\varepsilon_s}   
		&\leq 
		C_\alpha N_\ell^{-\alpha} \!,
		\qquad
		\tag{$\alpha$} 
		\label{eq:ass:alpha_MLMC_kth} 
		\\
		\quad
		\forall \ell\in\bbN : 
		&&
		\| X_\ell - X_{\ell-1} \|_{L^{kq}(\Omega;E)} 
		&\leq 
		C_\beta N_\ell^{-\beta} \!, 
		\qquad 
		\tag{$\beta$} 
		\label{eq:ass:beta_MLMC_kth} 
		\\
		\quad
		\forall \ell\in\bbN : 
		&&
		\cC_\ell 
		&\leq 
		C_\gamma N_\ell^\gamma . 
		\qquad 
		\tag{$\gamma$} 
		\label{eq:ass:gamma_MLMC_kth} 
	\end{align} 
	Furthermore, for each $\ell\in\bbN$, 
	let $( \xi_{\ell,j} )_{j\in\bbN} \subset 
	L^q(\Omega;\otimes^{k,s}_{\varepsilon_s}E_{\ell})$ 
	be a sequence of independent 
	copies of 
	the $\otimes^{k,s}_{\varepsilon_s}E_{\ell}$-valued 
	random variable $\xi_\ell$ in \eqref{eq:def:xi-ell-k}. 

	Then, for every $\epsilon\in (0,\nicefrac{1}{2}]$, 
	there exist integers $L\in\bbN$ and $M_1,\dots,M_L\in \bbN$ such that 
	the $L^q$-accuracy of the multilevel 
	Monte Carlo estimator for $\bbM_{\varepsilon}^k [X]$ satisfies
	\[ 
		\mathrm{err}^{k,{\sf ML}}_{q,\varepsilon_s}(X)
		:=  
		\biggl\| 
		\bbM_{\varepsilon}^k [X]
		-\sum_{\ell=1}^L
		\frac{1}{M_\ell} 
		\sum_{j=1}^{M_\ell}  
		\xi_{\ell,j}
		\biggr\|_{L^q(\Omega;\otimes_{\varepsilon_s}^{k,s} E)} 
		<\epsilon ,  
	\] 
	and, 
	letting $t := \frac{\gamma+a_0}{\beta+a_1}$, 
	can be achieved at a computational cost of order
	\[ 
		\cC^{k,{\sf ML}}_{q,\varepsilon_s}(X) 
		\lesssim
		\begin{cases} 
			\epsilon^{-\frac{\gamma}{\alpha}} 
			+ 
			\epsilon^{-\bar{q}'} 
			&\text{ if } 
			t<\bar{q}'\!, 
			\\[2pt]
			\epsilon^{-\frac{\gamma}{\alpha}} 
			+ 
			\epsilon^{-\bar{q}'} 
			|\log_A \epsilon|^{\bar{q}'+1} 
			&\text{ if }  
			t = \bar{q}'\!,
			\\[2pt]
			\epsilon^{-\frac{\gamma}{\alpha}} 
			+ 
			\epsilon^{-\bar{q}' - \frac{\gamma + a_0 - (\beta+a_1) \bar{q}' \!}{\alpha}} 
			&\text{ if } 
			t \in(\bar{q}'\!,p'] 
			\;\; \text{and} \;\; 
			\alpha \geq \beta+a_1, 
			\\[2pt]
			\epsilon^{-\frac{\gamma}{\alpha}} 
			+ 
			\epsilon^{-t} 
			|\log_A \epsilon|^{t+1} 
			&\text{ if } 
			t \in(\bar{q}'\!,p'] 
			\;\; \text{and} \;\; 
			\alpha < \beta+a_1, 
			\\[2pt]
			\epsilon^{-\frac{\gamma}{\alpha}} 
			+ 
			\epsilon^{-\bar{q}' - \frac{\gamma + a_0 - (\beta+a_1) \bar{q}' \!}{\alpha}} 
			&\text{ if }  
			t > p'\! \;\;\text{and}\;\; 
			\alpha \geq \beta+a_1, 
			\\[2pt]
			\epsilon^{-\frac{\gamma}{\alpha}} 
			+ 
			\epsilon^{-p' - \frac{\gamma + a_0 - (\beta+a_1) p' \!}{\alpha}} 
			&\text{ if }  
			t > p'\! \;\;\text{and}\;\; 
			\alpha < \beta+a_1. 
		\end{cases}
	\] 
\end{theorem}

\begin{proof}
	By Proposition~\ref{prop:error_estimate_MLMC_kth},  
	together with 
	\eqref{eq:ass:alpha_MLMC_kth}, \eqref{eq:ass:beta_MLMC_kth} and 
	Assumption~\ref{ass:constants_subspaces}, we obtain 
	\[
		\mathrm{err}^{k,{\sf ML}}_{q,\varepsilon_s}(X)
		\leq 
		C_{\alpha} N_L^{-\alpha} 
		+ 
		C_\star C_{\beta} C_{\tau}^{\frac{1}{r'}}
		\sum_{\ell=1}^L 
		M_{\ell}^{-\left( 1-\frac{1}{r} \right)} 
		N_{\ell}^{-(\beta+a_1)+\frac{a_0}{r'}},
	\]
	which has the same structure 
	as the bound for the error 
	$\mathrm{err}^{{\sf ML}}_{q}(X)$ 
	of the multilevel estimator  
	for the first moment 
	in \eqref{eq:bound_proof_complexity_MLMC}, 
	but with different constants.
	To achieve the $L^q$-accuracy~$\epsilon$,  
	the finest level $L$ and 
	the number of samples 
	$M_1,\ldots, M_L$ 
	can then be chosen as detailed in the proof 
	of 
	Theorem~\ref{thm:dim-dep-alpha-beta-gamma-1} 
	(by suitably adapting the constants). 
	By invoking the 
	assumption \eqref{eq:ass:gamma_MLMC_kth} 
	the asymptotic computational cost 
	of this estimator expressed in terms of 
	the parameters $(\alpha,\beta,\gamma)$ 
	is thus the same.
\end{proof}

\section{Optimized Monte Carlo error analysis in \texorpdfstring{$L^p$}{Lp} spaces based on Minkowski's integral inequality}\label{sec:Minkowski}

In this section, we restrict our study to $L^p$ spaces, 
and we present an analysis that, 
for a class of $L^p(S)$-valued random variables, 
under an additional integrability assumption 
on the random variables 
and an approximation property 
of the measure space $(S,\cS,\mu)$, 
permits to obtain sharper Monte Carlo error estimates 
than those provided by 
the analysis based on Rademacher types. 

We start with recalling the following classical result, 
see \cite[Appendix~A.1]{Stein1970} 
or \cite[Propositions~1.2.22 and~1.2.25]{AnalysisInBanachSpacesI2016}. 

\begin{theorem}[Minkowski's integral inequality]
\label{thm:minkowski-pq} 
	Assume that $(S,\cS,\mu)$ and $
	(T,\cT,\nu)$ are two 
	$\sigma$-finite measure spaces 
	and let $f\from (S\times T, 
	\cS\otimes\cT)\to(\bbR,\cB(\bbR))$ 
	be measurable. 
	Then we have, for every $1\leq p \leq q <\infty$, 
	\begin{equation}\label{eq:minkowski-pq} 
		\biggl[ \int_T \biggl( \int_S 
		|f(s,t)|^p \, \rd\mu(s) \biggr)^{\frac{q}{p}} 
		\rd \nu(t) \biggr]^{\frac{1}{q}}
		\leq
		\biggl[ \int_S \biggl( \int_T |f(s,t)|^q \, \rd\nu(t) \biggr)^{\frac{p}{q}} \rd\mu(s)
		\biggr]^{\frac{1}{p}}. 
	\end{equation} 

	In particular, the Bochner spaces 
	$L^q(T;L^p(S)) 
	= 
	L^q(T,\cT,\nu;L^p(S,\cS,\mu))$ 
	and 
	$L^p(S;L^q(T)) 
	= 
	L^p(S,\cS,\mu;L^q(T,\cT, \nu))$ 
	satisfy the following relation:  
	\[
		\forall \; 1 \leq  p\leq q < \infty : 
		\qquad 
		\norm{f}{L^q(T;L^p(S))}
		\leq 
		\norm{f}{L^p(S;L^q(T))}, 
	\]
	holding for all  
	$f\from S\times T \to \bbR$ 
	measurable, 
	and $L^p(S;L^q(T))$ 
	may be identified with 
	a subspace of $L^q(T;L^p(S))$ 
	which is continuously embedded 
	in $L^q(T;L^p(S))$. 
\end{theorem}

\begin{remark}
\label{rem:fct-on-product-space} 
	For Minkowki's integral inequality 
	(Theorem~\ref{thm:minkowski-pq}), 
	it is assumed that 
	$f$ is measurable on the 
	product measurable space $(S\times T, 
	\cS\otimes\cT)$. 
	In what follows, 
	we will consider  
	a probability space 
	$(\Omega,\cA,\bbP)$, 
	a $\sigma$-finite measure space 
	$(S,\cS,\mu)$, and 
	$L^p(S)$-valued random variables 
	in $L^q(\Omega;L^p(S))$ for some 
	$p,q\in[1,\infty)$. 
	We note that 
	Bochner integrability 
	of $\eta \from \Omega \to L^p(S)$ 
	implies that there exists  
	a measurable 
	function on the product space 
	$\widetilde{\eta} \from 
	(S\times\Omega, \cS\otimes\cA)
	\to(\bbR,\cB(\bbR))$, 
	which is 
	$(\mu\otimes\bbP)$-almost everywhere 
	uniquely defined, 
	such that 
	$\eta(\omega) = \widetilde{\eta}(\,\cdot\, , \omega)$ 
	in $L^p(S)$ 
	for $\bbP$-almost all $\omega\in\Omega$, 
	see e.g.\ 
	\cite[Proposition~1.2.25]{AnalysisInBanachSpacesI2016}. 
	The additional integrability condition 
	$\eta\in L^p(S;L^q(\Omega))$ means that the 
	mapping $s\mapsto \widetilde{\eta}(s,\,\cdot\,)$ is 
	in $L^p(S;L^q(\Omega))$. 
	Finally, we sometimes write 
	$\eta(s)$ 
	or 
	$\eta(s,\omega)$ 
	for $s\in S$, $\omega\in\Omega$, 
	and it will be clear from the context 
	that we are referring to 
	$\widetilde{\eta}(s,\,\cdot\,)$ 
	and $\widetilde{\eta}(s,\omega)$, respectively. 
\end{remark} 

\begin{definition}[Approximation by averaging]
\label{def:approx-averaging} 
	We say that a $\sigma$-finite 
	measure space 
	$(S,\cS,\mu)$ satisfies the 
	\emph{approximation-by-averaging property} 
	if, for every $n \in \bbN$, 
	there exists a family  
	$\mathfrak{B}_n 
	= 
	\{ B_{\indk}^n : \indk\in \mathscr{K}_n \}$, 
	which contains finitely or countably infinitely many  
	pairwise disjoint measurable subsets of $S$ 
	of positive, finite $\mu$-measure,  
	\begin{equation}\label{eq:Bn-definition} 
		\forall \indk \in \mathscr{K}_n : 
		\quad 
		B_\indk^n \in\cS, 
		\quad 
		\mu(B_\indk^n)\in(0,\infty), 
		\quad 
		B_\indk^n \cap B_j^n = \emptyset 
		\;\;\text{if}\;\;  
		j\neq \indk,  
	\end{equation}
	such that,  
	for any  
	Banach space $(E,\norm{\,\cdot\,}{E})$, 
	all $p\in[1,\infty)$, 
	and every function $f\in L^p(S;E)$, 
	the sequence of functions 
	$(f_n)_{n\in\bbN}$ 
	defined by averaging $f$ 
	with respect to $\mathfrak{B}_n$, 
	\begin{equation}\label{eq:fn-definition} 
		f_n 
		:= 
		\sum_{\indk\in\mathscr{K}_n} 
		\mathbf 1_{B_\indk^n} 
		\left( 
		\frac{1}{\mu(B_\indk^n)} \int_{B_\indk^n} f(s) \,\rd\mu(s) 
		\right) , 
	\end{equation}
	converges to $f$ in the norm of $L^p(S;E)$, i.e., 
	$\lim_{n\to\infty} 
	\| f - f_n \|_{L^p(S;E)} 
	= 0$. 
\end{definition} 

\begin{remark} 
	Note that by Definition~\ref{def:approx-averaging}, 
	for every $n\in\bbN$, 
	the family $\mathfrak{B}_n$ 
	and thus
	the index set $\mathscr{K}_n$
	and the subsets $\{B_\indk^n : \indk\in\mathscr{K}_n\}$
	are independent 
	of the choice 
	of the Banach space $E$, 
	the integrability $p\in[1,\infty)$, 
	and the function $f\in L^p(S;E)$. 
\end{remark} 

\begin{remark} 
	The fact that, 
	for every $n\in\bbN$, 
	the sets $\{ B_\indk^n : \indk\in \mathscr{K}_n\}$ 
	of $\mathfrak{B}_n$ 
	are pairwise disjoint implies, 
	firstly, 
	that the mapping $f\mapsto f_n$ 
	in \eqref{eq:fn-definition} 
	is stable in $L^p(S;E)$, 
	$\norm{f_n}{L^p(S;E)} \leq \norm{f}{L^p(S;E)}$, 
	and, secondly, 
	that in the case that $\mathfrak{B}_n$ 
	is countable and $|\mathscr{K}_n| = \infty$, 
	the series in \eqref{eq:fn-definition} 
	converges unconditionally in $\indk\in\mathscr{K}_n$. 
\end{remark} 

\begin{example} 
	Let $D\subseteq \bbR^d$, $d\in\bbN$,  
	be a Borel set in the Euclidean space, 
	$\cB(D)$ be the Borel $\sigma$-algebra generated 
	by all the subsets of~$D$ 
	which are open in $D$, 
	$\lambda_d$ denote the 
	Lebesgue measure on $\bbR^d$, 
	and suppose that 
	$\lambda_d(D) \in (0,\infty]$.  
	Then, $(D,\cB(D),\lambda_d)$  
	has the approximation-by-averaging 
	property. 
	
	Indeed, one may construct the 
	sets in~\eqref{eq:Bn-definition} explicitly: 
	For every $n\in\bbZ$, 
	consider the  
	standard dyadic cubes of side-length $2^{-n}\!$, 
	\[ 
		Q_\indk^n 
		:= 
		2^{-n} \bigl( \indk + [0,1)^d \bigr), 
		\quad 
		\indk \in \bbZ^d, 
	\]
	and define, 
	for every $n\in\bbZ$, 
	the family $\mathfrak{B}_n$ by 
	\begin{equation}\label{eq:Bn-D} 
		\mathfrak{B}_n 
		:= \{B_\indk^n : \indk\in\mathscr{K}_n\}, 
		\quad\;\; 
		B_\indk^n := 
		D\cap Q_\indk^n, 
		\quad\;\; 
		\mathscr{K}_n 
		:= 
		\bigl\{ 
		\indk\in\bbZ^d : 
		\lambda_d (B_\indk^n) > 0 
		\bigr\}.  
	\end{equation}
	Note that its cardinality  
	$|\mathscr{K}_n|$ is at most countably infinite   
	and finite if $D$ is bounded. 
	This construction readily implies that, 
	for every $n\in\bbN$, 
	$\mathfrak{B}_n$ 
	satisfies \eqref{eq:Bn-definition} 
	for the measure space 
	$(S,\cS,\mu) = (D,\cB(D),\lambda_d)$. 
	
	Let $(E,\norm{\,\cdot\,}{E})$ be a Banach space 
	and $p\in[1,\infty)$. 
	It remains to prove that, 
	for every $f\in L^p(D;E)$, 
	the sequence of functions $(f_n)_{n\in\bbN}$ 
	defined as in \eqref{eq:fn-definition} 
	using the families 
	$\mathfrak{B}_n$ 
	from \eqref{eq:Bn-D}, 
	for $n\in\bbN$, 
	converges to~$f$ in $L^p(D;E)$. 
	To this end, we first 
	observe that the 
	Borel $\sigma$-algebra on $\bbR^d$ 
	is the smallest $\sigma$-algebra 
	containing all the dyadic cubes
	$\{Q_\indk^n : \indk \in \bbZ^d, \, n\in\bbZ\}$.  
	Consequently,   
	the Borel $\sigma$-algebra on $D$ 
	coincides with the $\sigma$-algebra 
	generated by the sets 
	$B_\indk^n$ in~\eqref{eq:Bn-D}, i.e., 
	\[ 
		\cB(D) =  
		\sigma\bigl( B_\indk^n : \indk\in\bbZ^d, \, n\in\bbZ \bigr)   
		= 
		\sigma\Bigl( \bigcup\nolimits_{n\in\bbZ} \mathscr{F}_n \Bigr), 
		\qquad 
		\mathscr{F}_n 
		:= 
		\sigma\bigl( B_\indk^n : \indk\in\bbZ^d \bigr). 
	\]  
	Using the notation and definitions of 
	\cite[Chapter~3]{AnalysisInBanachSpacesI2016}, 
	we find that 
	$(\mathscr{F}_n)_{n\in\bbZ}$ is a 
	$\sigma$-finite filtration,   
	and the conditional expectation 
	of $f$ with respect to $\mathscr{F}_n$ 
	is given by $f_n$ in \eqref{eq:fn-definition}, 
	see \cite[Examples~2.6.13 and~3.1.3]{AnalysisInBanachSpacesI2016}. 
	Finally, the forward martingale convergence 
	theorem 
	\cite[Theorem~3.3.2(2)]{AnalysisInBanachSpacesI2016}
	shows that 
	$\lim_{n\to\infty} f_n = f$ 
	in the norm of $L^p(D;E)$.  
\end{example} 

\begin{lemma} 
\label{lem:approx-by-averag} 
	Assume that  
	$(S,\cS,\mu)$ is a  
	$\sigma$-finite measure space 
	which satisfies the approximation-by-averaging 
	property of Definition~\ref{def:approx-averaging}. 
	Let 
	$p,q\in[1,\infty)$, 
	${M\in\bbN}$, 
	and 
	$\eta_1,\ldots,\eta_M$ 
	be 
	independent and identically 
	distributed~$L^p(S)$-valued random variables 
	in $L^q(\Omega;L^p(S)) \cap L^p(S;L^q(\Omega))$ 
	with vanishing mean, 
	$\bbE[\eta_1]=0$. 
	
	There exist  
	sequences 
	$(\eta_1^n)_{n\in\bbN}, \ldots, (\eta_M^n)_{n\in\bbN}$ 
	of simple~functions in $L^p(S;L^q(\Omega))$ 
	with the property that, 
	for every $n\in\bbN$, there are finitely many pairwise disjoint 
	sets $A_1^n, \ldots, A_{K_n\!}^n \in \cS$  
	with positive, finite $\mu$-measure 
	and zero-mean 
	random variables 
	$Y_{11}^n, \ldots, Y_{1K_n}^n, 
	\ldots, 
	Y_{M1}^n, \ldots, Y_{MK_n\!}^n \in L^q(\Omega)$ 
	such that, for all $j\in\{1,\ldots,M\}$, 
	\begin{equation}\label{eq:lem:approx-by-averag-1}
		\eta_j^n 
		= 
		\sum\limits_{\indk=1}^{K_n} 
		\mathbf 1_{A_\indk^n} 
		Y_{j\indk}^n  
		\quad 
		\forall n \in\bbN, 
		\qquad 
		\lim_{n\to\infty} 
		\| \eta_j - \eta_j^n \|_{L^p(S;L^q(\Omega))} = 0, 
	\end{equation} 
	and, for all $n\in\bbN$ and every $\indk\in\{1,\ldots,K_n\}$, 
	we have that 
	\begin{equation}\label{eq:lem:approx-by-averag-2} 
		Y_{1\indk}^n, \ldots, Y_{M\indk}^n 
		\text{ are independent and  
		identically distributed}. 
	\end{equation}
\end{lemma} 	

\begin{proof} 
	Since $(S,\cS,\mu)$ has the 
	approximation-by averaging property, 
	there exists a sequence of families 
	$\mathfrak{B}_n = \{B_\indk^n : \indk \in \mathscr{K}_n\}$, $n\in\bbN$, 
	each of finite or countably infinite 
	cardinality~$|\mathscr{K}_n|$ and 
	containing subsets $B_\indk^n\subseteq S$ satisfying \eqref{eq:Bn-definition}, 
	such that for any $f \in L^p(S;L^q(\Omega))$  
	the sequence $f_n$ defined via \eqref{eq:fn-definition} 
	converges to $f$ in $L^p(S;L^q(\Omega))$. 
	
	\emph{Case 1:} $|\mathscr{K}_n| < \infty$ for all $n\in\bbN$. 
	If the cardinality $K_n:=|\mathscr{K}_n|$ 
	of the family $\mathfrak{B}_n$ is finite for every $n\in\bbN$, 
	we may choose the sets 
	$\{A_1^n, \ldots, A_{K_n}^n\} 
	= 
	\{B_\indk^n : \indk\in\mathscr{K}_n\}$ 
	(in an arbitrary but fixed ordering) 
	and define, 
	for $j\in\{1,\ldots,M\}$ and $\indk\in\{1,\ldots,K_n\}$,  
	the random variables 
	$Y_{j\indk}^n\from\Omega\to\bbR$ as 
	\begin{equation}\label{eq:Yjkn-definition} 
		Y_{j\indk}^n 
		:= 
		\frac{1}{\mu(A_\indk^n)} \int_{A_\indk^n} \eta_j(s) \,\rd\mu(s) 
		= 
		\langle g_\indk^n, \eta_j \rangle , 
		\qquad 
		g_\indk^n := \frac{1}{\mu(A_\indk^n)} \mathbf 1_{A_\indk^n} , 
	\end{equation}
	where $\langle\,\cdot\,, \,\cdot\,\rangle$ 
	denotes the duality pairing between 
	$(L^p(S))' \cong L^{p'}\!(S)$ and $L^p(S)$. 
	Note that 
	$g_\indk^n\in L^{p'}\!(S)$ holds 
	with $\norm{g_\indk^n}{L^{p'}\!(S)} = \mu(A_\indk^n)^{-\nicefrac{1}{p}}\!$, 
	for all $n\in\bbN$ and every $\indk\in\{1,\ldots,K_n\}$. 
	As $\eta_1, \ldots, \eta_M$ are assumed to be zero-mean and 
	independent, identically distributed as 
	$L^p(S)$-valued random variables in $L^q(\Omega;L^p(S))$, 
	we thus conclude that, 
	for every $j\in\{1,\ldots, M\}$, 
	all $n\in\bbN$ 
	and every $\indk\in\{1,\ldots,K_n\}$,   
	\begin{equation}\label{eq:Yjkn-centered} 
		Y_{j\indk}^n \in L^q(\Omega)
		\quad \text{with} \quad 
		\bbE\bigl[Y_{j\indk}^n \bigr]
		= 
		\langle g_\indk^n, \bbE[\eta_j] \rangle 
		= 0, 
	\end{equation}
	and, for all $n\in\bbN$ and every 
	$\indk\in\{1,\ldots,K_n\}$, 
	\begin{equation}\label{eq:Yjkn-iid} 
		\bigl( Y_{j\indk}^n \bigr)_{j=1}^M 
		= 
		\bigl( \langle g_\indk^n, \eta_j \rangle  \bigr)_{j=1}^M  
		\quad 
		\text{are independent and identically distributed}. 
	\end{equation} 
	Finally, setting 
	$\eta_j^n := 
	\sum_{\indk=1}^{K_n} 
	\mathbf 1_{A_\indk^n} 
	Y_{j\indk}^n$, 
	$n \in\bbN$, 
	the convergence   
	$\lim_{n\to\infty}  \eta_j^n  = \eta_j$ 
	in $L^p(S;L^q(\Omega))$, 
	for every $j\in\{1,\ldots,M\}$,  
	is immediate from the 
	construction using the 
	approximation-by-averaging property 
	of~$(S,\cS,\mu)$. 
	
	\emph{Case 2:} $|\mathscr{K}_n|$ 
	is countably infinite for at least one $n\in\bbN$. 
	Then,  
	by unconditional convergence of the series 
	in \eqref{eq:fn-definition}, 
	for every $n\in\bbN$ 
	and all $j\in\{1,\ldots,M\}$, 
	there exists a finite subset 
	$\mathscr{J}_n^j \subseteq \mathscr{K}_n$, i.e., 
	$|\mathscr{J}_n^j| < \infty$, 
	such that 
	\[
		\Biggl\| 
		\sum_{\indk\in\mathscr{K}_n \setminus \mathscr{J}_n^j} 		
		\mathbf 1_{B_\indk^n} \, 
		\overline{\eta}_{j\indk}^n 
		\Biggr\|_{L^p(S;L^q(\Omega))}
		< \frac{1}{Mn}, 
		\qquad 
		\overline{\eta}_{j\indk}^n 
		:= 
		\frac{1}{\mu(B_\indk^n)} \int_{B_\indk^n} \eta_j(s) \,\rd\mu(s). 
	\]
	We then consider, for every $n\in\bbN$, 
	\[ 
		\mathscr{J}_n 
		:= 
		\bigcup\limits_{j=1}^M \mathscr{J}_n^j , 
		\qquad 
		K_n := |\mathscr{J}_n| < \infty, 
		\qquad 
		\{A_1^n, \ldots, A_{K_n}^n\} 
		= 
		\{B_\indk^n : \indk\in\mathscr{J}_n\}, 
	\]
	where again the ordering 
	of the sets 
	$A_1^n, \ldots, A_{K_n}^n$
	is arbitrary but fixed. 
	Furthermore, we 
	define, for every $j\in\{1,\ldots,M\}$, 
	all $n\in\bbN$ and every $\indk\in\{1,\ldots,K_n\}$, 
	the random variables 
	$Y_{j\indk}^n$ as in \eqref{eq:Yjkn-definition}. 
	Then, 
	\eqref{eq:Yjkn-centered} 
	and \eqref{eq:Yjkn-iid} still hold 
	and, for every $j\in\{1,\ldots,M\}$, 
	the sequence 
	$\eta_j^n := \sum_{\indk=1}^{K_n} 
	\mathbf 1_{A_\indk^n} 
	Y_{j\indk}^n$, 
	$n \in\bbN$, 
	converges to~$\eta_j$ 
	in~$L^p(S;L^q(\Omega))$. 
	The latter follows, since 
	the sequence 
	\[
		\widetilde{\eta}_j^n 
		:=
		\sum_{\indk\in\mathscr{K}_n} 		
		\mathbf 1_{B_\indk^n} 
		\biggl( 
		\frac{1}{\mu(B_\indk^n)} \int_{B_\indk^n} \eta_j(s) \,\rd\mu(s) 
		\biggr)  
		= 
		\sum_{\indk\in\mathscr{K}_n} 		
		\mathbf 1_{B_\indk^n} \, 
		\overline{\eta}_{j\indk}^n , 
		\quad 
		n\in\bbN, 
	\] 
	has the limit $\eta_j$ in $L^p(S;L^q(\Omega))$ 
	by the approximation-by-averaging 
	property of $(S,\cS,\mu)$, 
	and 
	\begin{align*} 
		\| \widetilde{\eta}_j^n 
		&- 
		\eta_j^n \|_{L^p(S;L^q(\Omega))} 
		= 
		\Biggl\| 
		\sum_{\indk\in\mathscr{K}_n \setminus \mathscr{J}_n^j} 		
		\mathbf 1_{B_\indk^n} \, 
		\overline{\eta}_{j\indk}^n 
		- 
		\sum_{\indk\in\mathscr{J}_n \setminus \mathscr{J}_n^j} 		
		\mathbf 1_{B_\indk^n} \, 
		\overline{\eta}_{j\indk}^n 
		\Biggr\|_{L^p(S;L^q(\Omega))}  
		\\
		&\leq 
		\Biggl\| 
		\sum_{\indk\in\mathscr{K}_n \setminus \mathscr{J}_n^j} 		
		\mathbf 1_{B_\indk^n} \, 
		\overline{\eta}_{j\indk}^n 
		\Biggr\|_{L^p(S;L^q(\Omega))}  
		+ 
		\sum_{\substack{i=1 \\ i\neq j}}^M 
		\Biggl\| 
		\sum_{\indk\in\mathscr{J}_n^i \setminus \mathscr{J}_n^j} 		
		\mathbf 1_{B_\indk^n} \, 
		\overline{\eta}_{j\indk}^n 
		\Biggr\|_{L^p(S;L^q(\Omega))}  
		< \frac{1}{n}, 
	\end{align*} 
	by the choice of $\mathscr{J}_n^j$ 
	and $\mathscr{J}_n$, 
	completing the proof 
	in the second case. 
\end{proof} 

\begin{remark} 
The result of Lemma \ref{lem:approx-by-averag}, 
as well as the other results of this section, 
remain valid
\emph{without} assuming the approximation-by-averaging property 
of Definition~\ref{def:approx-averaging} 
if we restrict to a probability space $(\Omega,\mathcal{A},\mathbb{P})$ 
which has a tensor product structure,
that is, 
$(\Omega,\mathcal{A},\mathbb{P})=
(\wOmega^{M}\!, \wA^{M}\!, \wP^{M})$, 
$(\wOmega,\wA,\wP)$ being a complete probability space, 
and to random variables $\eta_1,\dots,\eta_M$ that are
identical copies of a given zero-mean 
$L^p(S)$-valued random variable $\eta$ 
defined on $\wOmega$. More precisely,
assume that $p,q\in [1,\infty)$ and $\eta$ is a random variable in 
$L^q(\wOmega;L^p(S))\cap L^p(S;L^q(\wOmega))$ 
with zero mean.
Then, $\eta$ can be approximated in $L^p(S;L^q(\wOmega))$ 
by a sequence of simple functions
$(\weta^n)_{n\in\mathbb{N}}$, 
satisfying 
$\weta^n(s,\,\cdot\,)=\sum_{\kappa=1}^{K_n} \mathbf{1}_{A^n_\kappa}(s) \widetilde{Y}^n_{\kappa}(\,\cdot\,)$, 
with $\widetilde{Y}^n_{\kappa} \in L^q(\wOmega)$ for every $\kappa\in \{1,\dots,K_n\}$ 
and $n\in\mathbb{N}$. 
Note that, for every $n\in\mathbb{N}$, 
$\widetilde{Y}^n_1, \ldots, \widetilde{Y}^n_{K_n}$ may not be centered. 
However, a sequence of simple approximations 
$(\eta^n)_{n\in\mathbb{N}}$ with centered random variables 
that still converges to $\eta$ in $L^p(S;L^q(\wOmega))$ 
can be obtained by setting 
$\eta^n(s,\,\cdot\,)=\sum_{\kappa=1}^{K_n} \mathbf{1}_{A^n_\kappa}(s) Y^n_\kappa(\,\cdot\,)$, 
where $Y^n_\kappa := \widetilde{Y}^n_\kappa - \widetilde{\bbE} [\widetilde{Y}^n_\kappa ]$.
Let now $\eta_j\from S\times \Omega\rightarrow \bbR$ be defined by 
$\eta_j(s,\widetilde{\omega}_1,\ldots,\widetilde{\omega}_M) := \eta(s,\widetilde{\omega}_j)$ 
for $j\in \{ 1,\dots,M \}$. 
Then, by construction, $\eta_1,\dots,\eta_M$ are independent and identically distributed 
on $(\Omega,\mathcal{A},\mathbb{P})=
(\wOmega^{M}\!, \wA^{M}\!, \wP^{M})$
and they belong to $L^q(\Omega;L^p(S))\cap L^p(S;L^q(\Omega))$. 
In addition, they can be simultaneously approximated 
over the same partition $\{ A^n_\kappa \}_{\kappa=1}^{K_n}$ 
by the sequences of simple functions 
$( \eta_j^n )_{n\in\bbN} \subset L^p(S;L^q(\Omega))$, 
$j\in\{1,\ldots,M\}$, defined by 
$\eta^n_j := \sum_{\kappa=1}^{K_n} \mathbf{1}_{A^n_\kappa} Y^n_{j\kappa}$, 
$n\in\mathbb{N}$,
where the random variables 
$Y^n_{j \kappa}(\widetilde{\omega}_1,\dots,\widetilde{\omega}_M) := Y^n_\kappa(\widetilde{\omega}_j)$ 
are centered, independent and identically distributed random variables 
in $L^q(\Omega)$. 
\end{remark}

\begin{proposition}\label{prop:minkowski-trick}  
	Let $(S,\cS,\mu)$ be a $\sigma$-finite measure space 
	with the 
	approximation-by-averaging property, 
	see Definition~\ref{def:approx-averaging}. 
	Let 
	$p\in [1,\infty)$, $q\in[p,\infty)$, 
	$\bar{q}:=\min\{q,2\}$, $M\in\bbN$, 
	and  
	$\eta_1,\ldots,\eta_M$ 
	be 
	independent and identically 
	distributed~$L^p(S)$-valued random variables 
	in $L^q(\Omega;L^p(S)) \cap L^p(S;L^q(\Omega))$ 
	with vanishing mean, 
	$\bbE[\eta_1]=0$. Then, 
	\begin{equation}\label{eq:prop:minkowski-trick}
		\biggl\|  
		\frac{1}{M} 
		\sum_{j=1}^M 
		\eta_j 
		\biggr\|_{L^q(\Omega;L^p(S))} 
		\leq 
		2 B_q
		M^{-\left( 1 - \frac{1}{\bar{q}} \right)} 
		\|  \eta_1 \|_{L^p(S;L^q(\Omega))} , 
	\end{equation} 
	where $B_q\in(0,\infty)$ is the Khintchine constant 
	from Definition~\ref{def:khintchine}. 
\end{proposition} 

\begin{remark} 
	Under the assumptions 
	of Proposition~\ref{prop:minkowski-trick}, 
	the analysis based on the 
	Rademacher type of $L^p(S)$ 
	(e.g., 
	\cite[Proposition 5.10]{CoxEtAl2021}) 
	yields, for $p\in[1,2]$ and $q\in[p,\infty)$, the bound   
	\[  
		\biggl\| 
		\frac{1}{M}
		\sum_{j=1}^M 
		\eta_j 
		\biggr\|_{L^q(\Omega;L^p(S))} 
		\leq 
		2 K_{q,p}  
		M^{-\left( 1 - \frac{1}{p} \right)}		
		\norm{ \eta_1 }{L^q(\Omega; L^p(S))}  . 
	\] 
	Proposition~\ref{prop:minkowski-trick} 
	shows that, for the case $q\geq p$, 
	the Monte Carlo convergence rate can be improved 
	from $1- p^{-1}$ 
	to $1- \bar{q}^{-1}$ 
	provided that  
	the stronger 
	integrability condition 
	$\eta_1,\ldots,\eta_M\in L^p(S;L^q(\Omega))$ 
	is satisfied. 
\end{remark} 

\begin{remark} 
	The statements of Lemma~\ref{lem:approx-by-averag}
	and Proposition~\ref{prop:minkowski-trick} 
	can be extended to the case of independent, 
	identically distributed  
	random variables $\eta_1,\ldots,\eta_M$ 
	with vanishing mean, 
	which take values in a \emph{Bochner space} $L^p(S;E)$, and 
	which satisfy the integrability conditions 
	$\eta_1,\ldots,\eta_M \in L^q(\Omega;L^p(S;E)) \cap L^p(S;L^q(\Omega;E))$. 
	More precisely, assuming that 
	$(S,\cS,\mu)$ is a $\sigma$-finite measure space 
	with the 
	approximation-by-averaging property,   
	$p\in [1,\infty)$, $q\in[p,\infty)$, 
	$\bar{q}:=\min\{q,2\}$,  
	and  
	$E$ is a Banach space of type $r\in[1,\bar{q}]$, 
	the analogue of the estimate \eqref{eq:prop:minkowski-trick} 
	reads as follows, 
	\[
		\biggl\|  
		\frac{1}{M} 
		\sum_{j=1}^M 
		\eta_j 
		\biggr\|_{L^q(\Omega;L^p(S;E))} 
		\leq 
		2 K_{q,r} \tau_r(E) 
		M^{-\left( 1 - \frac{1}{r} \right)} 
		\|  \eta_1 \|_{L^p(S;L^q(\Omega;E))} . 
	\]
\end{remark} 

\begin{proof}[Proof of Proposition~\ref{prop:minkowski-trick}]
	\emph{Step 1:}  
	We first prove the claim  
	in the case that, 
	for every $j\in\{1,\ldots,M\}$, 
	$\eta_j\in L^q(\Omega;L^p(S)) \cap 
	L^p(S;L^q(\Omega))$ is of the form 
	$\eta_j 
	= 
	\sum_{\indk=1}^K 
	\mathbf 1_{A_\indk} 
	Y_{j\indk}$, 
	where $K\in\bbN$, 
	the sets 
	$A_1,\ldots,A_K\in\cS$ are pairwise disjoint,  
	$\mu(A_\indk)\in(0,\infty)$  
	for every $\indk\in\{1,\ldots,K\}$, 
	and 
	\[
		\forall 
		\indk\in\{1,\ldots,K\}: 
		\quad 
		Y_{1\indk}, \ldots,  Y_{M\indk} \in L^q(\Omega)
		\   
		\text{are i.i.d.\ and have mean zero}. 
	\] 
	
	To this end,  
	let $(\Omegat,\cAt,\bbPt)$ be a second probability space 
	with expectation operator~$\bbEt$, and 
	$(\Omega\times\Omegat,\cA\otimes\cAt,\bbP\otimes\bbPt)$ 
	denote 
	the product probability space. 
	Let $(r_j)_{j=1}^M$ be a Rademacher family on $(\Omegat,\cAt,\bbPt)$
	and, for $j\in \left\{1,\dots,M\right\}$, 
	let $\boldsymbol{r}_j\from\Omega\times\Omegat\rightarrow \left\{-1,1\right\}$ 
	and 
	$\boldsymbol{\eta}_j\from\Omega\times\Omegat\rightarrow L^p(S)$ 
	denote the mappings that satisfy 
	\[
		\forall 
		(\omega,\omegat)\in\Omega\times\Omegat : 
		\quad 
		\boldsymbol{r}_j(\omega,\omegat)=r_j(\omegat), 
		\quad\;\;  
		\boldsymbol{\eta}_j(\omega,\omegat)=\eta_j(\omega)
		= 
		\sum_{\indk=1}^K 
		\mathbf 1_{A_\indk} Y_{j\indk}(\omega). 
	\]
	Note that on $(\Omega\times \Omegat,\cA\otimes\cAt,\bbP\otimes\bbPt)$ 
	the random variables $(\boldsymbol{r})_{j=1}^M$ are a Rademacher family,
	and $(\boldsymbol{\eta}_j)_{j=1}^M$ and $(\boldsymbol{r}_j)_{j=1}^M$ are independent. 
	Then, by symmetrization  
	and Minkowski's integral inequality 
	(see Lemma~\ref{lem:symmetrization}, 
	Theorem~\ref{thm:minkowski-pq} 
	and Remark~\ref{rem:fct-on-product-space}),  
	\begin{align*} 
		\biggl\|  
		\sum_{j=1}^M 
		\boldsymbol{\eta}_j 
		&\biggr\|_{L^q(\Omega\times\widetilde{\Omega};L^p(S))} 
		\leq 
		2\, 
		\biggl\|  
		\sum_{j=1}^M 
		\boldsymbol{r}_j 
		\boldsymbol{\eta}_j 
		\biggr\|_{L^q(\Omega\times\widetilde{\Omega};L^p(S))} 
		\leq 
		2\,
		\biggl\|  
		\sum_{j=1}^M 
		\boldsymbol{r}_j 
		\boldsymbol{\eta}_j 
		\biggr\|_{L^p(S; L^q(\Omega\times\widetilde{\Omega}))} 
		\\
		&= 
		2 
		\Biggl[ 
		\int_S 
		\Biggl(		
		\int_\Omega 
		\biggl\| 
		\sum_{j=1}^M r_j(\,\cdot\,)
		\eta_j(s, \omega)
		\biggr\|_{L^q(\widetilde{\Omega};\bbR)}^q 
		\rd \bbP(\omega) 
		\Biggr)^{\frac{p}{q}} 
		\rd\mu(s)
		\Biggr]^{\frac{1}{p}} 
		. 
	\end{align*} 
	Furthermore, 
	by the Khintchine inequalities 
	(see Definition~\ref{def:khintchine})
	and the fact that $(\bbR,|\,\cdot\,|)$ has 
	Rademacher type 
	$\bar{q}:= \min\{q,2\}$, 
	with $\tau_{\bar{q}}(\bbR)=1$, 
	we find that  
	\begin{align*} 
		\biggl\|  
		\sum_{j=1}^M 
		\boldsymbol{\eta}_j 
		&\biggr\|_{L^q(\Omega\times\widetilde{\Omega};L^p(S))} 
		\leq 
		2 B_q 
		\Biggl[ 
		\int_S 
		\Biggl(  
		\int_\Omega 
		\biggl\| 
		\sum_{j=1}^M r_j(\,\cdot\,)
		\eta_j(s, \omega)
		\biggr\|_{L^{\bar{q}}(\widetilde{\Omega};\bbR)}^q 
		\! \rd \bbP(\omega) 
		\Biggr)^{\frac{p}{q}} 
		\rd\mu(s)
		\Biggr]^{\frac{1}{p}} 
		\\
		&\leq 
		2 B_q 
		\Biggl[  
		\int_S 
		\Biggl( 
		\int_\Omega 
		\biggl( 
		\sum_{j=1}^M |  
		\eta_j(s, \omega) |^{\bar{q}} 
		\bigg)^{\nicefrac{q}{\bar{q}}}   
		\rd \bbP(\omega) 
		\Biggr)^{\frac{\bar{q}}{q} \cdot \frac{p}{\bar{q}}} 
		\rd\mu(s)
		\Biggr]^{\frac{1}{p}} 
		\\
		&\leq 
		2 B_q 
		\Biggl[  
		\int_S 
		\Biggl( 
		\sum_{j=1}^M \|   
		\eta_j(s, \,\cdot\,) \|_{L^q(\Omega;\bbR)}^{\bar{q}} 
		\Biggr)^{\frac{p}{\bar{q}}}    
		\rd \mu( s ) 
		\Biggr]^{\frac{1}{p}} , 
	\end{align*} 
	where we used the triangle 
	inequality on~$L^{\nicefrac{q}{\bar{q}}}(\Omega;\bbR)$ 
	in the last line. 
	Finally, by the specific 
	form of the random variables 
	$\eta_1,\ldots,\eta_M$ assumed in \emph{Step 1}, 
	we conclude~that 
	\begin{align*} 
		\biggl\|  
		\sum_{j=1}^M 
		\boldsymbol{\eta}_j 
		&\biggr\|_{L^q(\Omega\times\widetilde{\Omega};L^p(S))} 
		\leq 
		2 B_q 
		\Biggl[  
		\sum_{\indk=1}^K 
		\mu(A_\indk) 
		\biggl( 
		\sum_{j=1}^M \|   
		Y_{j \indk} \|_{L^q(\Omega;\bbR)}^{\bar{q}} 
		\biggr)^{\nicefrac{p}{\bar{q}}}     
		\Biggr]^{\frac{1}{p}} 
		\\
		&= 
		2 B_q 
		M^{\frac{1}{\bar{q}}} 
		\Biggl[  
		\sum_{\indk=1}^K 
		\mu(A_\indk)  
		\| Y_{1\indk} \|_{L^q(\Omega;\bbR)}^{p}  
		\Biggr]^{\frac{1}{p}} 
		= 
		2 B_q 
		M^{\frac{1}{\bar{q}}} 
		\|  \eta_1 \|_{L^p(S;L^q(\Omega))} . 
	\end{align*}  
	Consequently, 
	we obtain  
	\[
		\biggl\|  
		\frac{1}{M} 
		\sum_{j=1}^M 
		\eta_j 
		\biggr\|_{L^q(\Omega;L^p(S))} 
		= 
		\biggl\|  
		\frac{1}{M} 
		\sum_{j=1}^M 
		\boldsymbol{\eta}_j 
		\biggr\|_{L^q(\Omega\times\widetilde{\Omega};L^p(S))}
		\leq 
		2 B_q 
		M^{-\left( 1 - \frac{1}{\bar{q}} \right)} 
		\|  \eta_1 \|_{L^p(S;L^q(\Omega))}   . 
	\]

	\emph{Step 2:} 
	$\eta_1,\ldots,\eta_M$ 
	are arbitrary  
	independent and identically 
	distributed~$L^p(S)$-valued random variables 
	in $L^q(\Omega;L^p(S)) \cap L^p(S;L^q(\Omega))$ 
	with vanishing mean. 
	Then, 
	by Lemma~\ref{lem:approx-by-averag}
	there exist  
	sequences 
	$(\eta_1^n)_{n\in\bbN}, \ldots, (\eta_M^n)_{n\in\bbN}$ 
	of simple functions in $L^p(S;L^q(\Omega))$ 
	with the property that, 
	for every $n\in\bbN$, there are finitely many pairwise disjoint 
	sets $A_1^n, \ldots, A_{K_n\!}^n \in \cS$  
	with positive, finite $\mu$-measure 
	and zero-mean 
	random variables 
	$Y_{11}^n, \ldots, Y_{1K_n}^n, 
	\ldots, 
	Y_{M1}^n, \ldots, Y_{MK_n\!}^n \in L^q(\Omega)$ 
	such that \eqref{eq:lem:approx-by-averag-1}, 
	\eqref{eq:lem:approx-by-averag-2} hold. 
	Fix $\epsilon\in(0,\infty)$. 
	Then, for all $j\in\{1,\ldots,M\}$, 
	there exists an integer $n_j^\star = n_j^\star(\epsilon)\in\bbN$ such that 
	$\| \eta_j - \eta_j^{n_j^\star} \|_{L^p(S;L^q(\Omega))} < \epsilon$ 
	for all $n\geq n_j^\star$, 
	and letting $n^\star :=\max\{n_1^\star,\ldots,n_M^\star\}$ 
	we obtain using 
	the triangle inequalities on 
	$L^q(\Omega;L^p(S))$ and   
	$L^p(S;L^q(\Omega))$, 
	Minkowski's integral inequality 
	(Theorem~\ref{thm:minkowski-pq}),  
	and the result of \emph{Step 1}, applied 
	for $\eta_1^{n^\star},\ldots,\eta_M^{n^\star}$, 
	that 
	\begin{align*} 
		\biggl\|   
		\frac{1}{M} 
		\sum_{j=1}^M 
		\eta_j 
		\biggr\|_{L^q(\Omega;L^p(S))} 
		&\leq 
		\biggl\|
		\frac{1}{M}    
		\sum_{j=1}^M 
		\bigl( \eta_j - \eta_j^{n^\star} \bigr) 
		\biggr\|_{L^q(\Omega;L^p(S))} 
		+ 
		\biggl\|   
		\frac{1}{M} 
		\sum_{j=1}^M 
		\eta_j^{n^\star}
		\biggr\|_{L^q(\Omega;L^p(S))} 
		\\ 
		&\leq 
		\frac{1}{M}
		\sum_{j=1}^M 
		\bigl\| \eta_j - \eta_j^{n^\star} \bigr\|_{L^p(S;L^q(\Omega))} 
		+ 
		\biggl\|   
		\frac{1}{M} 
		\sum_{j=1}^M  
		\eta_j^{n^\star}
		\biggr\|_{L^q(\Omega;L^p(S))} 
		\\
		&<  
		\epsilon 
		+ 
		2 B_q 
		M^{ - \left( 1 - \frac{1}{\bar{q}} \right) } 
		\|  \eta_1^{n_\star} \|_{L^p(S;L^q(\Omega))}   
		\\
		&<  
		( 1 + 2 B_q ) 
		 \, \epsilon 
		+ 
		2 B_q 
		M^{ - \left( 1 - \frac{1}{\bar{q}} \right) } 
		\|  \eta_1 \|_{L^p(S;L^q(\Omega))}   . 
	\end{align*} 
	As $\epsilon\in(0,\infty)$ was arbitrary the assertion  
	follows from the limit $\epsilon\downarrow 0$. 
\end{proof} 

The following corollary 
combines    
Proposition~\ref{prop:minkowski-trick} 
and the analysis based 
on the Rademacher type 
of $L^p(S)$ in one result. 

\begin{corollary}\label{Cor:Gen_MC_Lpspaces}
	Let $(S,\cS,\mu)$ be a $\sigma$-finite measure space 
	with the approximation-by-averaging 
	property of Definition~\ref{def:approx-averaging}. 
	Let 
	$p,q\in [1,\infty)$, 
	$\bar{q}:=\min\{q,2\}$, ${M\in\bbN}$, 
	and 
	$X_1,\dots,X_M\in L^q(\Omega;L^p(S))\cap L^p(S;L^q(\Omega))$ 
	be independent and identically distributed $L^p(S)$-valued 
	random variables. Then, for every $U\in L^p(S)$,
	\begin{equation}\label{eq:Cor_Gen_MC}
		\biggl\|
		U-\frac{1}{M}\sum_{j=1}^M X_j 
		\biggr\|_{L^q(\Omega;L^p(S))} 
		\leq 
		\|U- \bbE[X_1] \|_{L^p(S)} 
		+ 
		M^{-\left(1-\frac{1}{\bar{q}} \right)} 
		C_{SL,q,p}(X_1),
	\end{equation}
	where 
	\[ 
	C_{SL,q,p}(X_1) 
	:=
	\begin{cases} 
		2K_{q,\bar{q}}
		B_p 
		\|X_1-\bbE[ X_1] \|_{L^q(\Omega;L^p(S))}   
		& \text{ if } q\in[1,p), 
		\\
		2 B_q  
		\|X_1-\bbE[X_1] \|_{L^p(S;L^q(\Omega))} 
		& \text{ if } q\in[p,\infty). 
	\end{cases} 
	\]
\end{corollary}

\begin{proof}
	The case $q\in [1,p)$ follows 
	as in 
	\cite[Proposition 5.10]{CoxEtAl2021} or 
	\cite[Corollary~3.15]{KKChS2024}, 
	since $L^p(S)$ has type~$r$, for any 
	$r\in [1,\min\{p,2\}]$, and thus, in particular, has type~$\bar{q}$, 
	with $\tau_{\bar{q}}(L^p(S)) 
	\leq \tau_{\min\{p,2\}}(L^p(S)) = B_p$, 
	see Example~\ref{ex:Ls-type}. 
	The case $q\in [p,\infty)$ is instead 
	a direct consequence of 
	Proposition~\ref{prop:minkowski-trick}.
\end{proof}

\begin{remark}\label{remark:L1}
	Corollary \ref{Cor:Gen_MC_Lpspaces} 
	shows convergence of the Monte Carlo method 
	to estimate the mean of an $L^1(S)$-valued 
	random variable $X$, 
	provided that $q\in (1,\infty)$ and 
	$X\in L^q(\Omega;L^1(S))\cap L^1(S;L^q(\Omega))$. 
	The analysis of Monte Carlo methods 
	based solely on Rademacher types
	does not guarantee convergence 
	for $L^1(S)$-valued random variables, 
	regardless of their integrability 
	with respect to the probability space, 
	since $L^1(S)$ has no Rademacher type greater than~$1$.
\end{remark}
\begin{remark}\label{remark:extra_integrability}
	Let $p\in [1,\infty)$, $q\in (1,\infty)$, 
	$\bar{q}:=\min\{q,2\}$, $M\in\mathbb{N}$, and 
	assume that the independent, identically distributed  
	random variables $X_1,\dots,X_M$ 
	satisfy a stronger integrability condition, 
	namely that 
	$X_1\in L^{\widetilde{q}}(\Omega;L^p(S))\cap L^p(S;L^{\widetilde{q}}(\Omega))$
	for some $\widetilde{q}\in[q,\infty)$, and 
	set  
	$\widehat{q}:=\min\{\widetilde{q},2\}$.
	Consider first the case $q\in (1,2)$. 
	Since the $L^q(\Omega;L^p(S))$-norm 
	is dominated by the $L^r(\Omega;L^p(S))$-norm 
	for any ${r\in [q,\widehat{q}] = [\bar{q},\widehat{q}]}$, 
	the estimate \eqref{eq:Cor_Gen_MC} 
	can be generalized to
	\begin{equation}\label{eq:extra_integrability}
		 \biggl\| 
		 U-\frac{1}{M}\sum_{j=1}^M X_j
		 \biggr\|_{L^q(\Omega;L^p(S))}
		 \leq 
		 \|U-\bbE[ X_1 ]\|_{L^p(S)} 
		 + 
		 M^{-\left(1-\frac{1}{r} \right)} 
		 C_{SL,r,p}(X_1),
	\end{equation}
	which holds true for any $r\in [\bar{q},\widehat{q}]$. 
	If instead $q\in [2,\infty)$, then 
	$\bar{q}=\widehat{q}=2$, 
	and \eqref{eq:Cor_Gen_MC} holds. 
	Estimates \eqref{eq:Cor_Gen_MC} 
	and \eqref{eq:extra_integrability} can be unified 
	in the general expression
	\[
		 \biggl\|
		 U - \frac{1}{M}\sum_{j=1}^M X_j
		 \biggr\|_{L^q(\Omega;L^p(S))}
		 \leq 
		 \|U-\bbE[X_1] \|_{L^p(S)} 
		 + 
		 M^{-\left(1-\frac{1}{r}\right)} 
		 C_{SL,\max\{q,r\},p}(X_1),
	\]
	which holds true for any $q\in (1,\infty)$, 
	$\widetilde{q}\in [q,\infty)$ and 
	$r\in [\bar{q},\widehat{q}]$. 
\end{remark}
In this work 
we do not necessarily assume that~$X$ 
has a finite second moment 
as is customary in classical Hilbertian analysis. 
Remark \ref{remark:extra_integrability} highlights that  
if the Monte Carlo error is measured in an $L^q$-norm, 
with $q\in(1,2)$,
and $X$ possesses higher integrability
characterized by the parameter $\widetilde{q}$, 
then the Monte Carlo convergence rate
 in the number of samples at the continuous level
is determined by the largest value of 
$r\in [\bar{q},\widehat{q}]$, that is, 
$\widehat{q}=\min\{\widetilde{q},2\}$. 
This observation is particularly relevant for 
multilevel methods, for which is unclear 
what could be the \emph{optimal}$\,$\footnote{In 
		this work, we use the adjective ``optimal” for quantities that are optimized based on error estimates. This should not be interpreted as guaranteeing that these quantities are the best possible in an absolute sense.} 
choice for~$r$: smaller values may be advantageous 
to improve the decay of the strong error, 
despite simultaneously increasing 
the sampling factor $M^{-\left(1-\frac{1}{r}\right)}$. 

The next theorem provides a general error estimate 
for multilevel Monte Carlo methods 
in the setting of 
Proposition~\ref{prop:minkowski-trick} and of 
Remark \ref{remark:extra_integrability}. 

\begin{theorem}\label{thm:MLMC_Lpspaces}
	Let $(S,\cS,\mu)$ be a $\sigma$-finite measure space 
	with the approximation-by-averaging 
	property of Definition~\ref{def:approx-averaging}.  
	Let
	$p\in [1,\infty)$, $q\in [p,\infty)$, 
	$\widetilde{q}\in [q,\infty)$,  
	$L\in\bbN$, and 
	set $\bar{q}:=\min\{q,2\}$,  
	$\widehat{q}:=\min\{\widetilde{q},2\}$. 
	Suppose further that, 
	for every $\ell\in \left\{1,\dots,L\right\}$, 
	$X_\ell \in 
	L^{\widetilde{q}}(\Omega;L^p(S)) 
	\cap L^p(S;L^{\widetilde{q}}(\Omega))$, 
	$M_\ell\in \mathbb{N}$, and 
	$\xi_{\ell,1},\dots,\xi_{\ell,M_{\ell}}$ 
	are independent copies of the 
	$L^p(S)$-valued random variables 
	\begin{equation}\label{eq:def:xi-ell-Minkowski} 
		\xi_{\ell} :=X_\ell -X_{\ell-1}, 
		\qquad 
		X_0:=0\in L^p(S). 
	\end{equation}
	Then, we have, 
	for every $U\in L^p(S)$ and  
	all $r\in [\bar{q},\widehat{q}]$, 
	\begin{equation}\label{eq:theorem_MLMC_Lpspaces}
		\begin{split}
		\biggl\| 
		U  -\sum_{\ell=1}^L  \frac{1}{M_\ell}
		\sum_{j=1}^{M_\ell} \xi_{\ell,j} 
		&\biggr\|_{L^q(\Omega;L^p(S))}
		\leq \|U-\bbE[X_L] \|_{L^p(S)}\\
		&+ 
		2 B_q 
		\sum_{\ell=1}^L M_{\ell}^{-\left(1-\frac{1}{r} \right)}
		\bigl\|  \xi_{\ell} - 
		\bbE[\xi_{\ell}] 
		\bigr\|_{L^p(S;L^{\max\{q,r\}}(\Omega))} . 
	\end{split}
	\end{equation} 
\end{theorem}

\begin{proof}
	Following the steps in the proof of 
	Proposition~\ref{prop:dim-dep-MLMC-1}, 
	we obtain for $U\in L^p(S)$,
	\begin{align*}
		\biggl\| 
		U - 
		\sum_{\ell=1}^L  \frac{1}{M_\ell} 
		\sum_{j=1}^{M_{\ell}} \xi_{\ell,j} 
		\biggr\|_{L^q(\Omega;L^p(S))}
		&\leq 
		\|U-\bbE[X_L] \|_{L^p(S)} 
		\\
		&\quad + 
		\sum_{\ell=1}^L 
		\biggl\| \frac{1}{M_\ell} \sum_{j=1}^{M_{\ell}} \bigl( \xi_{\ell,j} - \bbE[\xi_\ell] \bigr) 
		\biggr\|_{L^q(\Omega;L^p(S))}. 
	\end{align*}
	For every $\ell\in \{1,\dots,L\}$ 
	the random variables 
	$\eta_{\ell,j} := \xi_{\ell,j} - \bbE[\xi_\ell] 
	\in 
	L^{\widetilde{q}}(\Omega;L^p(S))  
	\cap L^p(S;L^{\widetilde{q}}(\Omega))$, 
	$j\in\{1,\ldots,M_\ell\}$, 
	have mean zero and are independent, 
	identically distributed. 
	Recalling that $q\in [p,\infty)$,
	the claim then follows from 
	Proposition~\ref{prop:minkowski-trick} 
	and 
	Remark~\ref{remark:extra_integrability}, 
	where we note that 
	$B_{\max\{q,r\}} = B_q$ for all 
	$r\in [\bar{q},\widehat{q}] \subseteq [1,2]$. 
\end{proof}

Since the right hand side of the error bound 
\eqref{eq:theorem_MLMC_Lpspaces}
holds for every $r\in [\bar{q},\widehat{q}]$, 
it is of interest to optimize the choice of~$r$ 
to compute an optimal budget allocation 
to reach a prescribed error tolerance. 
Clearly, this optimization is only meaningful 
provided that $q\in (1,2)$, 
since otherwise~$r$ can only be equal to~$2$
and the extra integrability does not 
yield improved convergence rates.
 
The next two theorems express 
the classical $``\alpha\beta\gamma$ theorem" 
\cite[Theorem~3.1]{giles2008multilevel} 
for the class of random variables considered in this section.
The first one assumes that $q\in [2,\infty)$, 
and shows that the multilevel Monte Carlo method 
achieves almost the same complexity
as in the Hilbertian case 
(up to an additional log-factor 
in the critical case $\beta = \tfrac{\gamma}{2}$), 
regardless of the Rademacher type of~$L^p(S)$.
The second theorem addresses 
the case $q\in (1,2)$ and provides
a general statement in ${r\in[q,\min\{\widetilde{q},2\}]}$ 
for the complexity of 
the multilevel Monte Carlo estimator 
under the assumption of additional integrability 
in $L^{\widetilde{q}}(\Omega;L^p(S)) 
\cap L^p(S;L^{\widetilde{q}}(\Omega))$.

\begin{theorem}\label{thm:MLMC_Lpspaces_complexity_qgreater2}
	Let $(S,\cS,\mu)$ be a $\sigma$-finite measure space 
	with the approximation-by-averaging 
	property of Definition~\ref{def:approx-averaging}, 
	$q\in [2,\infty)$, $p\in [1,q]$, 
	and ${X\in L^1(\Omega;L^p(S))}$. 
	For all $\ell\in\bbN$,
	let 
	$X_\ell \in L^q(\Omega;L^p(S))\cap L^p(S;L^q(\Omega))$  
	and define~$\xi_\ell$  
	as in \eqref{eq:def:xi-ell-Minkowski}. 
	For every $\ell\in\bbN$, 
	let $\cC_\ell$ denote the cost 
	(number of floating point operations)
	to generate one sample of 
	the random variable $\xi_\ell$ 
	in~\eqref{eq:def:xi-ell-Minkowski}, 
	and suppose that 
	there exist constants 
	${\alpha,\beta,\gamma,C_\alpha,C_\beta,C_\gamma\in(0,\infty)}$ 
	and 
	$A\in(1,\infty)$ such that  
	$N_\ell \eqsim A^\ell$ 
	for every $\ell\in\bbN$ and, moreover, 
	\begin{align} 
		\quad 
		\forall \ell\in\bbN : 
		&&
		\bigl\| \bbE[X] - \bbE[X_\ell] \bigr\|_{L^p(S)}   
		&\leq 
		C_\alpha N_\ell^{-\alpha} \!,
		\qquad
		\tag{$\alpha$} 
		\label{eq:ass:alpha_Minkowski_qgreater2} 
		\\
		\quad
		\forall \ell\in\bbN : 
		&&
		\| X_\ell - \bbE[X_\ell] - 
			(X_{\ell-1} - \bbE[X_{\ell-1}]) \|_{L^p(S;L^q(\Omega))} 
		&\leq 
		C_\beta N_\ell^{-\beta} \!, 
		\qquad 
		\tag{$\beta$} 
		\label{eq:ass:beta_Minkowski_qgreater2} 
		\\
		\quad
		\forall \ell\in\bbN : 
		&&
		\cC_\ell 
		&\leq 
		C_\gamma N_\ell^\gamma . 
		\qquad 
		\tag{$\gamma$}
		\label{eq:ass:gamma_Minkowski_qgreater2}
	\end{align} 
	For each $\ell\in\bbN$, 
	let 
	$( \xi_{\ell,j} )_{j\in\bbN} \subset 
	L^q(\Omega; L^p(S))\cap L^p(S; L^q(\Omega))$ 
	be a sequence of independent 
	copies of 
	the $L^p(S)$-valued random variable 
	$\xi_\ell$ in \eqref{eq:def:xi-ell-Minkowski}. 
	
	Then, for every $\epsilon\in(0,\nicefrac{1}{2}]$, 
	there exist integers 
	$L\in\bbN$ 
	and 
	$M_{1},\ldots,M_{L} \in \bbN$ such 
	that the $L^q$-accuracy of 
	the multilevel Monte Carlo 
	estimator for $\bbE[X]$ satisfies  
	\[ 
		\mathrm{err}^{\sf ML}_{q}(X)
		:=  
		\biggl\| 
		\bbE[X] 
		-
		\sum_{\ell = 1}^{L} 
		\frac{1}{M_{\ell}} 
		\sum_{j=1}^{M_{\ell}}  
		\xi_{\ell,j}
		\biggr\|_{L^q(\Omega; L^p(S))} 
		<\epsilon ,  
	\] 
 	and can be achieved at a computational 
	cost of order 
	\begin{equation}\label{eq:cC_LqLp_qgreater2}  
		\cC^{\sf ML}_{q}(X) 
		\lesssim
		\begin{cases} 
		\epsilon^{-\frac{\gamma}{\alpha}} 
		+ 
		\epsilon^{-2} 
		&\text{ if } 
		\beta>\frac{\gamma}{2}, 
		\\[2pt]
		\epsilon^{-\frac{\gamma}{\alpha}} 
		+ 
		\epsilon^{-2} 
		|\log_A \epsilon|^{3} 
		&\text{ if }  
		\beta = \frac{\gamma}{2},
		\\[2pt]
		\epsilon^{-\frac{\gamma}{\alpha}} 
		+ 
		\epsilon^{-2 - \frac{\gamma -2\beta}{\alpha}} 
		&\text{ if } 
		 \beta < \frac{\gamma}{2}.
		\end{cases}
	\end{equation} 
\end{theorem}

\begin{proof}
	The proof follows closely that of 
	Theorem~\ref{thm:dim-dep-alpha-beta-gamma-1}.
	Assumptions \eqref{eq:ass:alpha_Minkowski_qgreater2} 
	and \eqref{eq:ass:beta_Minkowski_qgreater2},
	together with Theorem \ref{thm:MLMC_Lpspaces}, 
	lead to the estimate 
	\begin{equation}\label{eq:MLMC_Lpspaces_accuracy_qgreater2}
		\mathrm{err}^{\sf ML}_{q}(X)
		\leq C_{\alpha} N_L^{-\alpha} 
		+ 
		2 B_q C_{\beta} 
		\sum_{\ell=1}^L M_{\ell}^{-\frac{1}{\bar{q}'}}N_{\ell}^{-\beta},
	\end{equation}
	where $\bar{q}:=\min\{q,2\}=2$.
	We next choose $L\in\bbN$ 
	as the smallest integer 
	such that 
	$N_L^{-\alpha}
	<
	(C_{\alpha}+2 B_q C_{\beta})^{-1}\epsilon$ 
	holds and, for every $\ell\in\{1,\dots,L\}$, we set 
	\[
		M_{\ell}
		:=
		\biggl\lceil N_L^{\alpha \bar{q}'} S_{L;\bar{q}}^{\bar{q}'} 
		N_{\ell}^{-\frac{(\beta+\gamma)\bar{q}'}{\bar{q}'+1}}
		\biggr\rceil, 
		\quad\;\text{where}\quad\;\;  
		S_{L;\bar{q}}
		:=
		\sum_{\ell=1}^L N_{\ell}^{\frac{\gamma-\beta\bar{q}'}{\bar{q}'+1}}.
	\]
	Inserting this choice for $L$ and $M_1,\dots,M_L$ 
	into \eqref{eq:MLMC_Lpspaces_accuracy_qgreater2} 
	verifies that
	$\mathrm{err}^{\sf ML}_{q}(X)<\epsilon$,
	while the cost of the corresponding estimator is
	by \eqref{eq:ass:gamma_Minkowski_qgreater2} bounded by 
	\begin{align}
		\cC^{\sf ML}_{q}(X) 
		\eqsim
		\sum_{\ell=1}^L \cC_{\ell} M_{\ell}
		&\leq 
		C_\gamma 
		\sum_{\ell=1}^L 
		N_\ell^\gamma 
		\left( 
		1 + 
		N_L^{\alpha \bar{q}'} 
		S_{L;\bar{q}}^{\bar{q}'} 
		N_{\ell}^{-\frac{(\beta+\gamma)\bar{q}'}{\bar{q}'+1}} 
		\right) 
		\\
		&=
		C_\gamma 
		\sum_{\ell=1}^L N_\ell^{\gamma} 
		+ 
		C_\gamma N_L^{\alpha\bar{q}'} 
		S_{L;\bar{q}}^{\bar{q}'+1} . 
	\end{align}
	Expressing this cost in terms of the accuracy $\epsilon$ as in the proof of Theorem \ref{thm:dim-dep-alpha-beta-gamma-1}, and recalling that $\bar{q}=2$,
	yields asymptotically \eqref{eq:cC_LqLp_qgreater2}. 
\end{proof}

\begin{theorem}\label{thm:MLMC_Lpspaces_complexity}
	Let $(S,\cS,\mu)$ be a $\sigma$-finite measure space 
	with the approximation-by-averaging 
	property of Definition~\ref{def:approx-averaging},  
	$q\in (1,2)$, $p\in [1,q]$, 
	${X\in L^1(\Omega;L^p(S))}$, 
	${\widetilde{q}\in [q,\infty)}$, and set 
	$\widehat{q}:=\min\{\widetilde{q},2\}$. 
	For every $\ell\in\bbN$,
	let 
	$X_\ell \in L^{\widetilde{q}}(\Omega;L^p(S))\cap L^p(S;L^{\widetilde{q}}(\Omega))$  
	and define~$\xi_\ell$  
	as in \eqref{eq:def:xi-ell-Minkowski}. 
	For every $\ell\in\bbN$, 
	let $\cC_\ell$ denote the cost 
	(number of floating point operations)
	to generate one sample of 
	the random variable $\xi_\ell$ 
	in~\eqref{eq:def:xi-ell-Minkowski}, 
	and suppose that 
	there exist 
	a function $\beta\from[q,\widehat{q}]\to(0,\infty)$ 
	and constants 
	${\alpha,\gamma,C_\alpha,C_\beta,C_\gamma\in(0,\infty)}$,  
	$A\in(1,\infty)$ such that  
	$N_\ell \eqsim A^\ell$ 
	for every $\ell\in\bbN$, and 
	\begin{align}  
		\forall \ell\in\bbN : 
		&&
		\bigl\| \bbE[X] - \bbE[X_\ell] \bigr\|_{L^p(S)}   
		&\leq 
		C_\alpha N_\ell^{-\alpha} \!,
		\qquad
		\tag{$\alpha$} 
		\label{eq:ass:alpha_Minkowski} 
		\\
		\forall r\in[q,\widehat{q}] \;\;
		\forall \ell\in\bbN : 
		&&  \!\! 
		\| X_\ell - \bbE[X_\ell] - 
			(X_{\ell-1} - \bbE[X_{\ell-1}]) \|_{L^p(S;L^r(\Omega))} 
		&\leq 
		C_\beta N_\ell^{-\beta(r)} \!, 
		\qquad 
		\tag{$\beta$} 
		\label{eq:ass:beta_Minkowski} 
		\\
		\forall \ell\in\bbN : 
		&&
		\cC_\ell 
		&\leq 
		C_\gamma N_\ell^\gamma . 
		\qquad 
		\tag{$\gamma$}
		\label{eq:ass:gamma_Minkowski}  
	\end{align} 
	For each $\ell\in\bbN$, 
	let $( \xi_{\ell,j} )_{j\in\bbN} 
	\subset 
	L^{\widetilde{q}}( \Omega; L^p(S))\cap L^p( S; L^{\widetilde{q}}(\Omega))$ 
	be a sequence of independent 
	copies of 
	the $L^p(S)$-valued random variable 
	$\xi_\ell$ in \eqref{eq:def:xi-ell-Minkowski}.    
	
	Then, for every $\epsilon\in(0,\nicefrac{1}{2}]$ 
	and all $r\in [q,\widehat{q}]$, 
	there exist integers 
	$L\in\bbN$ 
	and 
	$M_{1},\ldots,M_{L} \in \bbN$ such 
	that the $L^q$-accuracy of 
	the multilevel Monte Carlo 
	estimator for $\bbE[X]$ satisfies  
	\begin{equation}\label{eq:epsilon_LqLp} 
		\mathrm{err}^{\sf ML}_{q}(X)
		:=  
		\biggl\| 
		\bbE[X] 
		-
		\sum_{\ell = 1}^{L} 
		\frac{1}{M_{\ell}} 
		\sum_{j=1}^{M_{\ell}}  
		\xi_{\ell,j}
		\biggr\|_{L^q(\Omega; L^p(S))} 
		<\epsilon ,  
	\end{equation} 
 	and can be achieved at a computational 
	cost of order 
	\begin{equation}\label{eq:cC_LqLp}  
		\cC^{\sf ML}_{q;r}(X) 
		\lesssim
		\begin{cases} 
		\epsilon^{-\frac{\gamma}{\alpha}} 
		+ 
		\epsilon^{-r'} 
		&\text{ if } 
		\beta(r)>\frac{\gamma}{r'}, 
		\\[2pt]
		\epsilon^{-\frac{\gamma}{\alpha}} 
		+ 
		\epsilon^{-r'} 
		|\log_A \epsilon|^{r'+1} 
		&\text{ if }  
		\beta(r) = \frac{\gamma}{r'},
		\\[2pt]
		\epsilon^{-\frac{\gamma}{\alpha}} 
		+ 
		\epsilon^{-r' - \frac{\gamma -\beta(r)r'}{\alpha}} 
		&\text{ if } 
		 \beta(r) < \frac{\gamma}{r'}.
		\end{cases}
	\end{equation} 
\end{theorem}

\begin{proof}
	The application of Theorem \ref{thm:MLMC_Lpspaces}, 
	together with
	the assumptions \eqref{eq:ass:alpha_Minkowski} 
	and~\eqref{eq:ass:beta_Minkowski},   
	leads, for every $r\in [q,\widehat{q}]$, to
	\[
		\mathrm{err}^{\sf ML}_{q}(X) 
		\leq 
		C_\alpha N_L^{-\alpha} 
		+ 
		2 C_\beta 
		\sum_{\ell=1}^L M_{\ell}^{-\frac{1}{r'}} N_{\ell}^{-\beta(r)} \!, 
	\]
	where we further used that $B_q = 1$, since $q\in (1,2)$.  
	Repeating the arguments of the proof of 
	Theorem~\ref{thm:MLMC_Lpspaces_complexity_qgreater2} 
	with $\bar{q}$ replaced by $r$, for any $r\in [q,\widehat{q}]$, 
	and using~\eqref{eq:ass:gamma_Minkowski} 
	yields the assertions \eqref{eq:epsilon_LqLp} and \eqref{eq:cC_LqLp}.
\end{proof}

The optimal complexity is achieved by 
further optimizing with respect to~$r$. 
Note that this optimization is strongly problem-dependent 
as it depends on the behavior of the map 
$[q,\widehat{q}] \ni r\rightarrow \beta(r)$. 
The next corollary assumes a specific expression for $\beta(r)$ 
which will be encountered in the numerical experiment 
of Subsection~\ref{sec:num_FA}. 
 
\begin{corollary}\label{Cor:MLMC_Optimal_Lpspaces}
	Let the assumptions of 
	Theorem~\ref{thm:MLMC_Lpspaces_complexity} hold. 
	In addition, suppose that there exist 
	constants $b_0\in\bbR$ and $b_1\in (0,\infty)$ such that 
	$\beta(r)=b_0+\frac{b_1}{r}>0$ holds 
	for all $r\in [q,\widehat{q}]$. 
	Then,  
	the bound \eqref{eq:epsilon_LqLp}, i.e., 
	the $L^q$-accuracy $\epsilon$ of 
	the multilevel Monte Carlo 
	estimator for $\bbE[X]$, can be achieved 
	with the optimal complexity
	\begin{equation}\label{eq:Corollary_cC_LqLp} 
		\cC^{\sf ML}_{q}(X) 
		\lesssim
		\begin{cases} 
		\epsilon^{-\frac{\gamma}{\alpha}} 
		+ 
		\epsilon^{-\widehat{q}'} 
		&\text{ if } 
		t<\widehat{q}'\!, 
		\\[2pt]
		\epsilon^{-\frac{\gamma}{\alpha}} 
		+ 
		\epsilon^{-\widehat{q}'} 
		|\log_A \epsilon|^{\widehat{q}'+1} 
		&\text{ if }  
		t = \widehat{q}'\!,
		\\[2pt]
		\epsilon^{-\frac{\gamma}{\alpha}} 
		+ 
		\epsilon^{-\widehat{q}' - \frac{\gamma + b_1 - (b_0+b_1) \widehat{q}' \!}{\alpha}} 
		&\text{ if } 
		t \in(\widehat{q}'\!,q'] 
		\;\; \text{and} \;\; 
		\alpha \geq b_0+b_1, 
		\\[2pt]
		\epsilon^{-\frac{\gamma}{\alpha}} 
		+ 
		\epsilon^{-t} 
		|\log_A \epsilon|^{t+1} 
		&\text{ if } 
		t \in(\widehat{q}'\!,q'] 
		\;\; \text{and} \;\; 
		\alpha < b_0+b_1, 
		\\[2pt]
		\epsilon^{-\frac{\gamma}{\alpha}} 
		+ 
		\epsilon^{-\widehat{q}' - \frac{\gamma + b_1 - (b_0+b_1) \widehat{q}' \!}{\alpha}} 
		&\text{ if }  
		t > q'\! \;\;\text{and}\;\; 
		\alpha \geq b_0+b_1, 
		\\[2pt]
		\epsilon^{-\frac{\gamma}{\alpha}} 
		+ 
		\epsilon^{-q' - \frac{\gamma + b_1 - (b_0+b_1) q' \!}{\alpha}} 
		&\text{ if }  
		t > q'\! \;\;\text{and}\;\; 
		\alpha < b_0+b_1, 
		\end{cases}
	\end{equation}
	where $t:=\frac{\gamma+b_1}{b_0+b_1}$  
\end{corollary}

\begin{proof}
	Due to the assumption on $\beta(r)$, 
	\eqref{eq:cC_LqLp} can be reformulated as
	\begin{equation}\label{eq:cC_LqLp_t}  
		\cC^{\sf ML}_{q;r}(X) 
		\lesssim_{(\alpha,b_0,b_1,\gamma,q,A)}
		\begin{cases} 
		\epsilon^{-\frac{\gamma}{\alpha}} 
		+ 
		\epsilon^{-r'} 
		&\text{ if } 
		r'>t, 
		\\[2pt]
		\epsilon^{-\frac{\gamma}{\alpha}} 
		+ 
		\epsilon^{-r'} 
		|\log_A \epsilon|^{r'+1} 
		&\text{ if }  
		r' = t,
		\\[2pt]
		\epsilon^{-\frac{\gamma}{\alpha}} 
		+ 
		\epsilon^{-r' - \frac{\gamma +b_1 - (b_0+b_1)r'}{\alpha}} 
		&\text{ if } 
		 r'<t. 
		\end{cases}
	\end{equation}
	with $t=\frac{\gamma+b_1}{b_0+b_1}$, 
	where we note that 
	$b_0 + b_1 \geq \beta(q) > 0$ 
	by assumption. 
	The minimization of this cost bound 
	with respect to $r\in [q,\widehat{q}]\subset (1,2]$ 
	can be recasted into the minimization problem 
	considered in the proof of Theorem \ref{thm:dim-dep-alpha-beta-gamma-1}, 
	by the identification $b_1=a_0$ and $b_0+b_1=\beta+a_1$. 
	Hence, the optimal choice of 
	the parameter $r'\in [\widehat{q}',q']\subset [2,\infty)$, 
	and thus of $r\in [q,\widehat{q}]$, is
	\begin{align*}
		r' 
		= 
		\begin{cases} 
		\widehat{q}' 
		&\text{ if } 
		t \leq \widehat{q}'\!, 
		\\ 
		\widehat{q}' 
		&\text{ if } 
		t \in(\widehat{q}'\!,q'] 
		\;\;\text{and}\;\; 
		\alpha \geq b_0+b_1,  
		\\ 
		t
		&\text{ if } 
		t \in(\widehat{q}'\!,q'] 
		\;\;\text{and}\;\; 
		\alpha < b_0+b_1, 
		\\
		\widehat{q}' 
		&\text{ if } 
		t >q'\! \;\; \text{and}\;\; 
		\alpha \geq b_0+b_1, 
		\\ 
		q'  
		&\text{ if } 
		t >q'\! \;\; \text{and}\;\; 
		\alpha < b_0+b_1, 
		\end{cases} 
	\end{align*}
	which yields the optimal cost bound \eqref{eq:Corollary_cC_LqLp}.
\end{proof}

\begin{remark}\label{remark:gen_higher_moments}
	In this section, we have so far discussed 
	single- and multilevel Monte Carlo methods
	to approximate the mean of an 
	$L^p(S)$-valued random variable, 
	with $(S,\cS,\mu)$ 
	being a $\sigma$-finite measure space 
	satisfying the approximation-by-averaging 
	property of Definition~\ref{def:approx-averaging}. 
	As explained in Subsection~\ref{sec:prel_kmoments}, 
	for any $L^p(S)$-valued random variable 
	$X\in L^k(\Omega;L^p(S))$, 
	its $L^p$ $k$th moment 
	$\bbM_p^k[ X ]\in \otimes^k_p L^p(S)$ 
	can be identified with an 
	expected value in $L^p(S^k)$. 
	Moreover, the product measure space 
	$(S^k, \cS^k, \mu^k)$ inherits the approximation-by-averaging property 
	from $(S,\cS,\mu)$.  
	Hence, by suitably replacing $L^p(S)$ with $L^p(S^k)$, 
	Corollaries~\ref{Cor:Gen_MC_Lpspaces} 
	and~\ref{Cor:MLMC_Optimal_Lpspaces}, 
	and Theorems~\ref{thm:MLMC_Lpspaces},
	\ref{thm:MLMC_Lpspaces_complexity_qgreater2} 
	and~\ref{thm:MLMC_Lpspaces_complexity}
	provide error and complexity estimates 
	for single- and multilevel Monte Carlo methods 
	to approximate 
	$\bbM_p^k[ X ]$ in the norm 
	of~$L^q(\Omega;L^p(S^k))$ 
	provided that  
	$\otimes^k X$ is in $L^q(\Omega; L^p(S^k))$ 
	and in $L^p(S^k; L^q(\Omega))$. 
	Note that the former assumption is equivalent to requiring 
	that $X \in L^{kq}(\Omega;L^p(S))$, while the latter 
	is implied if $X \in L^p(S; L^{kq}(\Omega))$. 
\end{remark}

\section{Numerical experiments}\label{sec:numerics}
In this section we present two numerical experiments aimed at validating the theoretical results established in Sections \ref{sec:dimension_dependent_constants} and \ref{sec:Minkowski}. All experiments were performed in Matlab R2021b.
For reproducibility, the related codes and datasetes are downloadable from \cite{Codes}.

\subsection{Linear elliptic boundary value problem with random forcing}
\label{sec:num_rbvp}
Let $I:=(0,1)$, $\eta\in (1,2)$ and, for every $y\in I$, let 
$f_y \from I \rightarrow \bbR$ be a forcing term defined as $f_y(x):=C_\eta|x-y|^{-\eta}$, 
$C_\eta$ being a suitable constant to be specified later, and $g_y\in\mathbb{R}$. 
For a given $y\in I$, 
we consider the boundary value problem
\begin{equation}\label{eq:RBVP}
	\left\{ \begin{aligned} 
  		-u_y^{\dprime}(x) &= f_y(x),\quad \forall x\in I,\\
   		u_y(0) &=0,\\
   		u_y^\prime(1)&=g_y.
	\end{aligned} \right.
\end{equation}
Problem \eqref{eq:RBVP} requires a nonstandard variational formulation since,
for any $y \in I$ and $\eta\in (1,2)$, $f_y \notin L^1(I)$, 
and thus the forcing term has to be interpreted as an element of a suitable dual space. 
Since $I$ has dimension one, for any $p\in [1,\infty]$, elements in $W^{1,p}(I)$ coincide
(upon modification on a subset of $\overline{I}$ of zero Lebesgue measure)
with a unique function $\tilde{v}\in C^0(\overline{I})$. 
Hence, we may define the subspace
\[
	W^{1,p}_{\{0\}}(I):=\{ v \in W^{1,p}(I):\;\tilde{v}(0)=0 \},
\]
and we denote by $W^{-1,p}_{\{0\}}(I)$ the dual space of $W^{1,p'\!}_{\{0\}}(I)$.
Due to the Poincar\'e inequality, 
$W^{1,p}_{\{0\}}(I)$ can be equipped with the norm 
$\|v\|_{W^{1,p}_{\{0 \}}(I)}:=\|v^\prime\|_{L^p(I)}$. 

By multiplying formally the first equation in \eqref{eq:RBVP} 
with a test function $v$
and integrating by parts both sides,
we derive the following variational formulation for any $p\in (1,\infty)$:
\begin{equation}\label{eq:var_form}
	\text{Find }u_y\in W^{1,p}_{ \{0 \}}(I)\text{ such that}\quad 
	B(u_y,v)=L_y(v) 
	\quad \forall v\in W^{1,p'\!}_{ \{0 \}}(I),
\end{equation}
where
\begin{align*}
	B&\colon W^{1,p}_{\{0 \}}(I)\times W^{1,p'\!}_{ \{0 \}}(I)\rightarrow \bbR,&
	B(u,v)&:=\int_I u^\prime(x) v^\prime(x)\rd x,\\
	L_y&\colon W^{1,p'\!}_{ \{0 \}}(I)\rightarrow \bbR,
	& L_y(v)&:=-\int_I v^\prime(x)F_y(x)\rd x + v(1)(F_y(1)+g_{y}),
\end{align*}
and $F_y(x):=C_\eta \operatorname{sign}(x-y)|x-y|^{1-\eta}/(1-\eta)$ 
is a primitive of $f_{y}$.
For any $ p \in (1,\infty)$, 
the bilinear form $B$ is continuous and satisfies the inf-sup condition 
\[
	\inf_{0\neq w\in W^{1,p}_{ \{0 \}}(I)} \sup_{0\neq v\in W^{1,p'\!}_{ \{0 \}}(I)} 
	\frac{ B(w,v) }{ \|w\|_{W^{1,p}_{ \{0 \}}(I)} \|v\|_{W^{1,p'\!}_{ \{0 \}}(I)} } 
	\geq 1,
\]
together with its symmetric counterpart, 
see \cite[Section~4.2.1]{KKChS2024} 
and \cite[Theorem~3.1]{babuska1981posteriori} for a constructive proof.
In addition, provided that $p\in \bigl( 1,\frac{1}{\eta-1} \bigr)$, 
for any $y\in I$, we have that $F_y\in L^p(I)$ and
\[
	\|L_y\|_{W^{-1,p}_{ \{0 \}}(I)}= \|F_y(1)+g_y-F_y\|_{L^p(I)}<\infty.
\]  
Therefore, for any $ y\in I$, $\eta\in (1,2)$ and $p\in (1,\frac{1}{\eta-1})$,
\eqref{eq:var_form} admits a unique solution $u_y\in W^{1,p}_{ \{0 \}}(I)$,
satisfying the a-priori bound 
$\|u_y\|_{W^{1,p}_{ \{0 \}}(I)}\leq \|L_y\|_{W^{-1,p}_{ \{0 \}}(I)}$. 
The range $\eta\in \bigl[\frac{3}{2},2 \bigr)$ is of particular interest since then $p\in (1,2)$.

Next, let $(\Omega,\cA,\bbP)$ be a complete probability space 
and assume that $y\from\Omega\rightarrow I$ is a 
uniformly distributed random variable and 
$g\from\Omega\rightarrow \mathbb{R}$ a real-valued random variable.
The previous discussion permits to construct pathwise, 
for every $\eta\in (1,2)$ and $p\in (1,\frac{1}{\eta-1})$, 
the function-valued random variable
$u\from\Omega\rightarrow W^{1,p}_{ \{0 \}}$, $u(\w):=u_{y(\w)}$ for $\bbP$-a.e.~$\w$,
where $u_{y(\w)}$ solves \eqref{eq:RBVP} with right hand side $f_{y(\w)}$ 
and Neumann boundary condition $g_{y(\w)}=g(\w)$. 

Our first goal is to numerically estimate the convergence rates 
of the Monte Carlo method to approximate $\bbE[u]$ 
via samples of the map $\w\mapsto u(\w)$. 
To this end, for any $\eta\in (1,2)$, we introduce the manufactured solution
\begin{equation}\label{eq:uy-explicit} 
	u_y(x) = |x-y|^{2-\eta} - y^{2-\eta},
\end{equation}
which satisfies \eqref{eq:RBVP} with $C_\eta=-(2-\eta)(1-\eta)$
and $g_y=-\frac{C_\eta}{1-\eta}|1-y|^{1-\eta}$. 
The first moment of $u$ is then
\[
	\bbE[ u ] 
	= 
	\int_\Omega u(\omega)\rd \bbP(\w)
	=
	\int_0^1 \left(|x-y|^{2-\eta}-y^{2-\eta}\right)\rd y
	=
	\frac{(1-x)^{3-\eta}+x^{3-\eta}-1}{3-\eta}.
\]
Due to this particular choice of $u$, we have that $F_y(1)=-g_y$,
implying that $\|L_y\|_{W^{-1,p}_{ \{0 \}}(I)}= \|F_y\|_{L^p(I)}$, 
and a direct calculation shows that $u\in L^q(\Omega;W^{1,p}_{ \{0 \}}(I))$ 
for any $q\in [1,\infty]$ and $p\in \bigl( 1, \frac{1}{\eta-1} \bigr)$.

In our numerical experiments, given a $p\in (1,2]$, 
we set $\eta=1+\frac{1}{p}-10^{-2}$ to impose sharply 
that $u\in L^\infty(\Omega;W^{1,p}_{ \{0 \}}(I))$.
We draw $K=50$ sets, 
$\{ y(\w^k_i) \}_{i=1}^{M_j}$, $k\in\left\{1,\dots,K\right\}$, 
of $M_j$ i.i.d.\ realizations of $y$, with $M_j=10^j$, $j\in \{1,\dots,4 \}$. 
The error associated to the Monte Carlo approximation is estimated as
\[
		\biggl\| \bbE [u] -\frac{1}{M_j}\sum_{i=1}^{M_j} u_i \biggr\|_{L^2(\Omega;W^{1,p}_{ \{0 \}}(I))}
		\approx 
		\Biggl(\frac{1}{K}\sum_{k=1}^K 
		\biggl\| (\bbE[u])^\prime - \frac{1}{M_j}\sum_{i=1}^{M_j} u^\prime(\w^k_i) 
		\biggr\|_{L^p(I)}^2 \Biggr)^{\nicefrac{1}{2}} \!. 
\]
Here, the $L^p(I)$-norm is computed using a composite Gauss--Jacobi quadrature, 
since the sample average has $M_j$ singularities, 
one at each of the realizations 
$\{ y(\w^k_i) \}_{i=1}^{M_j}$, $k\in\left\{1,\dots,K\right\}$.
The error decay is fitted with respect to the sample sizes $M_j$ using least squares. 
Table \ref{tab:RBVP_CMC} compares the estimated and theoretical convergence rates, 
and confirms the sharpness of the Monte Carlo error analysis 
based on Rademacher types at the continuous level.
\begin{table}
    \centering
    \begin{tabular}{|c|c|c|c|c|c|c|}
        \hline
        $p$ & 1.1 & 1.3 & 1.5 & 1.7 & 1.9 & 2 \\
        \hline
        rate  & 0.10 (0.09) & 0.24 (0.23) & 0.34 (0.33) & 0.42 (0.41) & 0.48 (0.47) & 0.5 (0.5) \\
         \hline
    \end{tabular}
    \caption{
    		For different values of $p$, the numerically estimated 
    		Monte Carlo convergence rates and the theoretical ones $1-\frac{1}{p}$ (in brackets) 
    		in the norm of $L^2(\Omega; W^{1,p}_{\{0\}}(I))$ 
    		for approximating the first moment 
    		of $u(\w) = u_{y(\w)}$, see~\eqref{eq:uy-explicit}.
    		}
    \label{tab:RBVP_CMC}
\end{table}

We next consider partitions $\mathcal{T}_{h}$ of the interval $I$
constructed on $N_h+1$ equispaced points $\left\{x_j:=jh\right\}_{j=0}^{N_h}$,
$h=\frac{1}{N_h}$, and corresponding subintervals $I_h^j=[x_j,x_{j+1}]$, 
$j\in\left\{0,\dots,N_h-1\right\}$. 
For each partition $\mathcal{T}_{h}$, 
we introduce the space of continuous, piecewise affine functions,
\[
	S_0^1(I,\mathcal{T}_h) 
	:= 
	\bigl\{ v\in C^0(\overline{I}):\; 
	v(0)=0,\; v|_{I_h^j}\in \mathbb{P}_1,\;j\in \{0,\dots,N_h-1 \}
	\bigr\}
	\subset W^{1,r}_{ \{0 \}}(I),
\]
for any $r\in [1,\infty]$. 
Associated with this space, we introduce the nodal interpolant 
${\cI^1_h\from C^0(\overline{I})\rightarrow S_0^1(I,\mathcal{T}_h)}$ 
which is uniquely defined by the values of $v$ at the points $ \{x_j \}_{j=1}^{N_h}$.
For any $y\in I$, a discrete approximant $u^h_y$ of $u_y$ 
can be computed using a Galerkin approximation of \eqref{eq:var_form} 
with respect to $S_0^1(I,\mathcal{T}_{h})$: 
\begin{equation}\label{eq:var_form_FEM}
	\text{Find }u^h_y \in S_0^1(I,\mathcal{T}_{h})\text{ such that}\quad  
	B(u^h_y,v)=L_y(v) \quad \forall v\in S_0^1(I,\mathcal{T}_{h}).
\end{equation}
It can be shown (see \cite[Section 4.2.2]{KKChS2024}) 
that $B$ satisfies a \emph{discrete} inf-sup condition, 
and the discrete inf-sup constant is equal to 1.  
This implies that, 
for any $ y \in I$, $\eta\in (1,2)$ and $p\in\bigl( 1,\frac{1}{\eta-1} \bigr)$, 
\eqref{eq:var_form_FEM} admits a unique solution 
which satisfies the a-priori stability bound 
$\|u^h_y\|_{W^{1,p}_{ \{0 \}}(I)}\leq \|L_y\|_{W^{-1,p}_{\left\{0\right\}}(I)}$, 
and which enjoys the quasi-optimality property,
\[
	\|u_y-u^h_y\|_{W^{1,p}_{ \{0 \}}(I)} 
	\leq 
	2\inf_{v\in S_0^1(I,\mathcal{T}_h)}
	\|u_y-v\|_{W^{1,p}_{ \{0 \}}(I)}.
\]
Note, however, that the best approximation error does not decay with the optimal rate $h$
since $u_y\notin W^{2,p}(I)$. 
Next, we introduce the finite-dimensional random variable 
$u_h\from \Omega\rightarrow  S_0^1(I,\mathcal{T}_h)$, 
$u_h(\w):=u^h_{y(\w)}$ for $\bbP$-a.e.~$\w$,
which further satisfies $u_h\in L^q(\Omega;W^{1,p}_{ \{0 \}}(I))$,
for any  $q\in [1,\infty]$ and $p\in (1,\frac{1}{\eta-1})$.
For our experiments, 
samples of $u_h$ are computed by directly evaluating 
the continuous piecewise affine interpolant for samples of~$u$, 
which is justified by the relation 
$u^h_y=\cI^1_{h} u_y$ holding for all $y \in I$,
see \cite[Chapter~0.7]{brenner2007mathematical}. 

We now study the convergence of the single-level Monte Carlo method, 
in light of the new analysis established in Section \ref{sec:dimension_dependent_constants}
that accounts for dimension-dependent type constants. 
To this end, we first show that for any $p\in (1,2]$, 
the finite-dimensional subspace $S_0^1(I,\mathcal{T}_h)\subset W^{1,p}_{ \{0 \}}(I)$, 
of dimension $\operatorname{dim}(S_0^1(I,\mathcal{T}_h))=N_h$, 
satisfies Assumption~\ref{ass:constants_subspaces} 
for a suitable choice of $a_0, a_1 \in [0,\infty)$.
Let $v_h\in S_0^1(I,\mathcal{T}_h)$ and denote by $\vb_\delta\in\mathbb{R}^{N_h}$
the vector of components $(\vb_\delta)_i:=v_h(x_i)-v_h(x_{i-1})$, $i=1,\dots,N_h$. 
A direct calculation shows that
\[
	\|v_h\|_{W^{1,p}_{ \{0 \}}(I)}
	=
	\Biggl( 
	\sum_{j=0}^{N_h-1} \int_{I^j_h} |v_h'(t)|^p\rd t \Biggr)^{\nicefrac{1}{p}}
	= 
	h^{\frac{1}{p}-1} \|\vb_{\delta}\|_{\ell^p_{N_h}} \!.
\]
Then, for any $r\in [p,2]$, for any Rademacher family $(r_j)_{j\in\bbN}$ on a complete probability space 
$(\widetilde{\Omega},\widetilde{\cA},\widetilde{\bbP})$, 
for every $n\in\bbN$, and all functions $v_{h,1},\ldots,v_{h,n}\in S_0^1(I,\mathcal{T}_h)$
and associated vectors $\vb_{\delta,1},\ldots,\vb_{\delta,n}$,
\begin{align*} 
		\biggl\| 
		\sum_{j=1}^n r_j v_{h,j} &\biggr\|_{L^r(\widetilde{\Omega};W^{1,p}_{ \{0 \}}(I))}
		= 
		h^{\frac{1}{p}-1} 
		\biggl\|\sum_{j=1}^n r_j \vb_{\delta,j} \biggr\|_{L^r(\widetilde{\Omega};\ell^p_{N_h})}  
		\\
		&\leq 
		\tau_r(\ell^p_{N_h}) 
		h^{\frac{1}{p}-1} 
		\Biggl( \sum_{j=1}^n \|\vb_{\delta,j}\|^r_{\ell^p_{N_h}}  \Biggr)^{\nicefrac{1}{r}} 
		= 
		N_h^{\frac{1}{p}-\frac{1}{r}} \Biggl(\sum_{j=1}^n \|v_{h,j}\|^r_{W^{1,p}_{ \{0 \}}(I)} \Biggr)^{\nicefrac{1}{r}} \!,  
\end{align*}
see Example~\ref{ex:choice_Mell_SLMC}. 
This shows that $S_0^1(I,\mathcal{T}_h)$ satisfies 
Assumption \ref{ass:constants_subspaces} with $a_0=1$, $a_1=\frac{1}{p'}$ and $C_{\tau}=1$.
Next, we let $p\in (1,2]$, $\eta=1+\frac{1}{p}-10^{-2}$, 
and consider a finite sequence of mesh sizes $h_{\ell}=2^{-\ell}$, $\ell\in \mathbb{N}$, 
and fit using least squares a positive constant $C_{\alpha_p}$ and a rate 
$\alpha_p\in (0,\infty)$ such that 
\begin{align*}
		\|\bbE [u] - \bbE [ u_{h_\ell}] \|_{W^{1,p}_{ \{0 \}}(I)} 
		=
		\|\bbE [ u ] -\bbE[ \cI^1_{h_\ell} u ] \|_{W^{1,p}_{ \{0 \}}(I)}
		&= 
		\|\bbE[ u ] - \cI^1_{h_\ell} \bbE[ u ]\|_{W^{1,p}_{ \{0 \}}(I)}\\
		& \leq C_{\alpha_p} h_{\ell}^{\alpha_p}, 
\end{align*}
where the second equality follows from 
the linearity of the expectation operator and the continuity of $\bbE [ u ]$.
As for all $p\in (1,2]$ and every $\eta\in \bigl[1,1+\frac{1}{p} \bigr)$ 
we have that $\E[ u] \in W^{2,p}(I)$, 
we expect the piecewise affine interpolant to converge at a rate $\alpha_p\approx 1$, 
which is confirmed by numerical experiments.
For a sequence of decreasing tolerances $\{ \epsilon_j \}_{j=1}^6$, 
the fitted parameters $(C_{\alpha_p},\alpha_p)$ allow us to define 
a sequence of partitions $\bigl\{\mathcal{T}_{h_{\ell_j}} \bigr\}_{j=1}^{6}$, 
with $h_{\ell_j}=2^{-\ell_j}$, 
$\ell_j:=\Bigl\lceil \bigl| \tfrac{1}{\alpha_p}\log_2 \bigl(\frac{\epsilon_j}{C_{\alpha_p}} \bigr) \bigr|\Bigr\rceil$, 
which guarantee that the bias associated to $S_0^1(I,\mathcal{T}_{h_{\ell_j}})$ is less than $\epsilon_j$, 
for $j\in\left\{1,\dots,6\right\}$. 
We then consider three different scalings of the sample sizes $ \{M_j \}_{j=1}^6$ 
as functions of the prescribed error tolerances $ \{\epsilon_j \}_{j=1}^6$. 
We consider (i) $M_j= \bigl\lceil\epsilon_j^{-2} \bigr\rceil$, 
which is the asymptotic optimal scaling 
derived in a classical Monte Carlo error analysis 
in Hilbert spaces, 
(ii)~${M_j=\bigl\lceil\epsilon_j^{-\pprime} \bigr\rceil}$, 
which is the asymptotic optimal scaling of Monte Carlo methods 
for random variables with values in a Banach space of type~$p$, 
and (iii) $M_j = \bigl\lceil \epsilon_j^{-2-(\pprime-2)/(\pprime \alpha_p)} \bigr\rceil$, 
which is the optimal scaling derived in Section~\ref{sec:dimension_dependent_constants} 
taking into account 
the dimension-dependent Rademacher constants. 
The latter scaling can be derived by 
substituting~$N_{h_{\ell_j}}\propto 2^{\ell_j}\propto \epsilon_j^{-\frac{1}{\alpha_p}}$ 
into the first case of~\eqref{eq:explicit_choice_Mell_SLMC}, 
since $\alpha_p\geq a_1=\frac{1}{p'}$, 
and noting that $\bar{q}=2$ for this example.
In each setting, we draw $K=30$ sets of $M_6$ realizations of~$y$, 
$\{ y( w^k_i)  \}_{i=1}^{M_6}$, $k\in\left\{1,\dots,K\right\}$, 
and approximate the Monte Carlo error 
on each partition $\mathcal{T}_{h_{\ell_j}}$ using
\begin{equation}\label{eq:error_DMC}  
		\mathrm{err}^{\sf SL}_{2}(u) 
		\approx 
		\Biggl( \frac{1}{K}\sum_{k=1}^K 
		\biggl\| (\bbE[ u ] )^\prime
		- \frac{1}{M_j}\sum_{i=1}^{M_j} u_{h_{\ell_j}}^\prime \!(\w^k_i) \biggr\|^2_{L^p(I)} 
		\Biggr)^{\nicefrac{1}{2}} \!. 
\end{equation}
The $L^p(I)$-norms in \eqref{eq:error_DMC} are computed 
using a composite Gauss--Legendre quadrature formula 
since $(\bbE [u] )^\prime$ is smooth, 
while the sample average is piecewise constant 
over the partition $\mathcal{T}_{h_{\ell_j}}$.

Figure~\ref{Fig:Decay_SLMC_RBVP_first_moment} 
shows the decay of the single-level Monte Carlo error
for a decreasing sequence of tolerances and for $p\in \{1.5,1.6,1.8 \}$. 
The numerical results do confirm the necessity 
for the more sophisticated Monte Carlo analyses, 
based either on the Rademacher type of the infinite-dimensional image space, 
or by additionally considering the dependence of the type constants 
on the dimensions of the approximation spaces. 
The classical Hilbertian analysis falls short in practice, 
suggesting smaller sample sizes which do not allow to reach asymptotically the prescribed tolerances.

\begin{figure}
\includegraphics[scale=0.28]{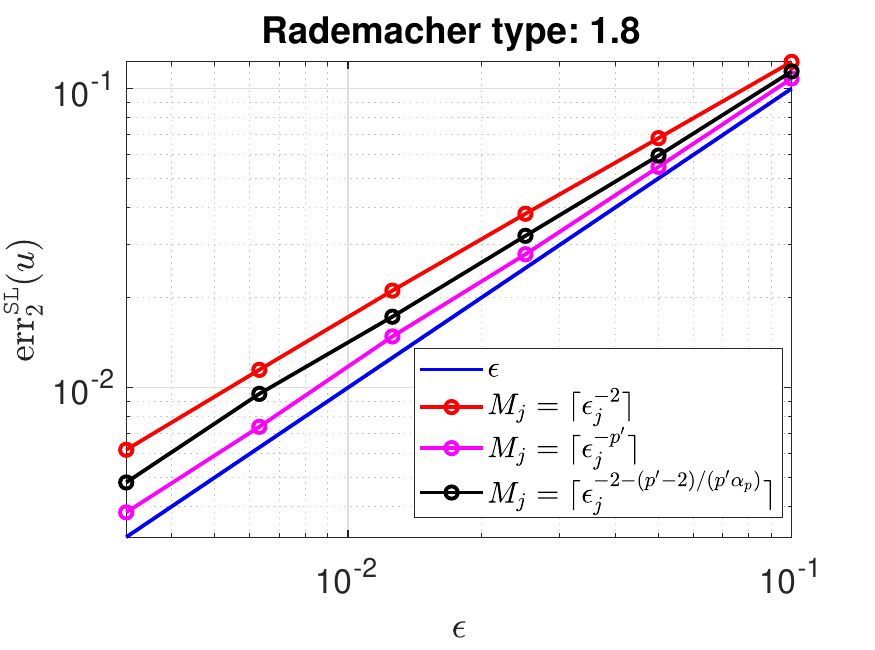}
\includegraphics[scale=0.28]{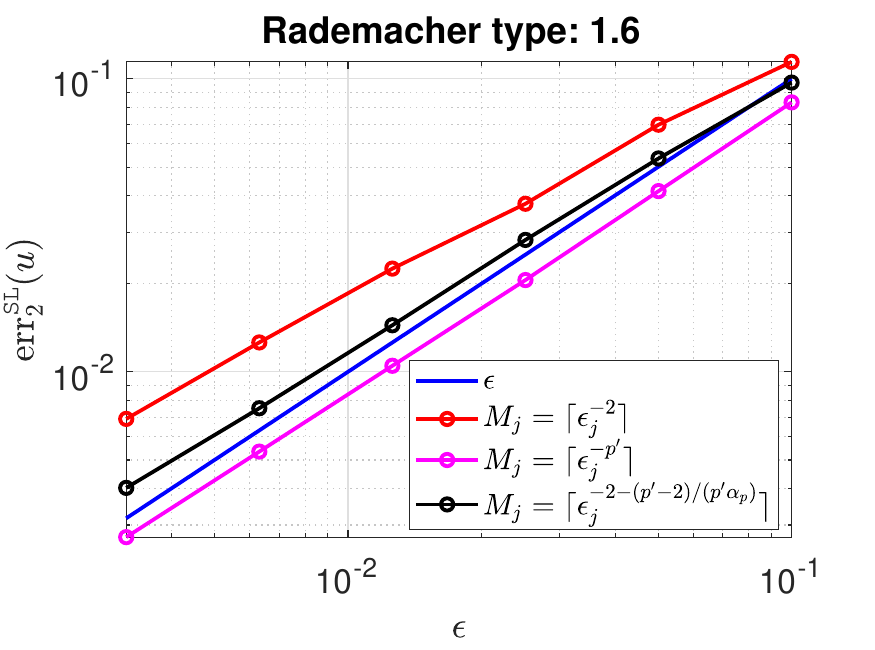}
\includegraphics[scale=0.28]{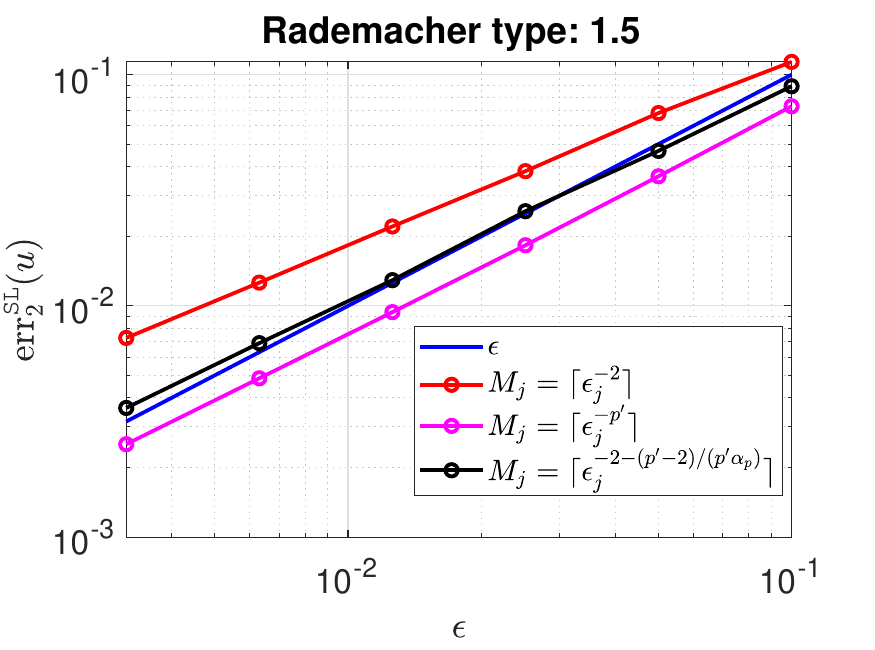}
\caption{
		Decay of the single-level Monte Carlo error 
		measured in the norm of $L^2(\Omega;L^p(I))$ 
		for the first moment of the solution to \eqref{eq:var_form}, 
		for different Rademacher types and choices of sample sizes.
		}\label{Fig:Decay_SLMC_RBVP_first_moment}
\end{figure}

We finally analyze the approximation of the second moment of $u$. 
Let $\mathcal{T}^2_h$ be a uniform grid of $[0,1]^2$ obtained by 
tensorizing the one-dimensional partition $\mathcal{T}_h$. 
We introduce the approximation space
\[
	S_0^{1,1}(I\times I,\mathcal{T}^2_{h}):=
	S_0^1(I,\mathcal{T}_h)\otimes S_0^1(I,\mathcal{T}_h)
	\subset \otimes^2_\varepsilon W^{1,r}_{\left\{0\right\}}, 
\]
for any $r\in [1,\infty]$, and the nodal interpolant 
$\cI^{2}_{h_\ell}\from C^0(\overline{I\times I})\rightarrow S_0^{1,1}(I\times I,\mathcal{T}^2_{h_\ell})$, 
satisfying $\cI^{2}_{h_\ell}=\cI^1_{h_\ell}\otimes \cI^1_{h_\ell}$.
Since a closed form expression for $\bbE[ u \otimes u]$ is not available, 
we used in our experiments a numerically computed ``reference second moment'',
$\bbE\bigl[ \cI^{2}_{h_{\ell_\refe}}\!(u \otimes u) \bigr]$.
To this end, we used a sufficiently small mesh size 
$h_{\ell_\refe}=2^{-\ell_\refe}$, $\ell_{\refe}\in \mathbb{N}$, 
which is stored as matrix of dimension $(N_{h_{\ell_\refe}}\!+1)\times (N_{h_{\ell_\refe}}\!+1)$. 
Its $(i,j)$th element corresponds to
\begin{equation}\label{eq:reference_second_moment}
	(\bbE [u\otimes u])(x_i,x_j)=\int_0^1 u_y(x_i)u_y(x_j)\rd y,
\end{equation}
where $x_i$ and $x_j$ are nodes of $\mathcal{T}_{h_{\ell_\refe}}\!$ 
for any $i,j\in \left\{0,\dots, N_{\ell_\refe}\right\}$.
Since the mapping $y\rightarrow u_y(x_i)u_y(x_j)$ 
is not differentiable at $y=x_i$ and $y=x_j$, 
the integral in \eqref{eq:reference_second_moment} is split 
over the subintervals $[0,\min\{x_i,x_j\}]$, $[\min\{x_i,x_j\},\max\{x_i,x_j\}]$ and $[\max\{x_i,x_j\},1]$. 
Then, each of the three integrals is approximated using an adaptive composite $hp$ quadrature formula.
Similarly to the first moment, given $p\in (1,2]$, 
we set $\eta=1+\frac{1}{p}-10^{-2}$, and estimate through 
least squares a pair $(C_{\alpha_p},\alpha_p)\in (0,\infty)^2$ 
such that
\begin{equation}\label{eq:bias_second_moment_RBVP}
	\bigl\|
	\bbE\bigl[ \cI^{2}_{h_{\ell_\refe}} \! ( u \otimes u) \bigr] 
	-
	\bbE\bigl[ \cI^{2}_{h_{\ell}} (u \otimes u) \bigr] 
	\bigr\|_{\otimes_{\varepsilon}^2 W^{1,p}_{ \{0 \}}(I)}
	\leq C_{\alpha_p} h_{\ell}^{\alpha_p}, 
	\quad \ell<\ell_\refe.
\end{equation} 
Since $(x,x')\mapsto (\bbE[ u\otimes u ])(x,x')$ has a kink along the diagonal $x'=x$, 
we observe rates $\alpha_p<1$, which  deteriorate as $p$ decreases. 
Note further that \eqref{eq:bias_second_moment_RBVP} involves the 
\emph{numerical evaluation of the injective tensor norm} 
$\| \,\cdot\, \|_{\otimes_{\varepsilon}^2 W^{1,p}_{ \{0 \}}(I)}$. 
The numerical algorithm we used for this purpose 
is detailed in the Appendix.

Figure \ref{Fig:Decay_SLMC_RBVP_second_moment} shows the decay of 
\[
	\mathrm{err}^{2,{\sf SL}}_{2,\varepsilon}(u)
	= 
	\biggl\| 
	\bbE \bigl[ \cI^{2}_{h_{\ell_\refe}}\! (u\otimes u) \bigr] 
	- 
	\frac{1}{M_j} 
	\sum_{i=1}^{M_j} \cI^1_{h_{\ell_j}} \! u_i\otimes 
	\cI^1_{h_{\ell_j}} \! u_i 
	\biggr\|_{L^2(\Omega;\otimes^2_\varepsilon W^{1,p}_{ \{0 \}}(I))},
\]
for a decreasing sequence of tolerances 
$ \{\epsilon_j \}_{j=1}^6$, for $p\in\left\{1.5,1.6,1.8\right\}$, 
and for the three different choices of $M_j$ 
detailed for the first moment.

We emphasize once more that increasing 
the sample size proportionally to $\epsilon^{-2}$, 
as would be suggested by a Monte Carlo error analysis 
in Hilbert spaces
is insufficient to achieve the target tolerances. 
This highlights that an analysis, which takes into account the 
Rademacher type of the image space of the random variable, 
is not merely of theoretical interest, 
but also of practical relevance. 
 
\begin{figure}
\includegraphics[scale=0.28]{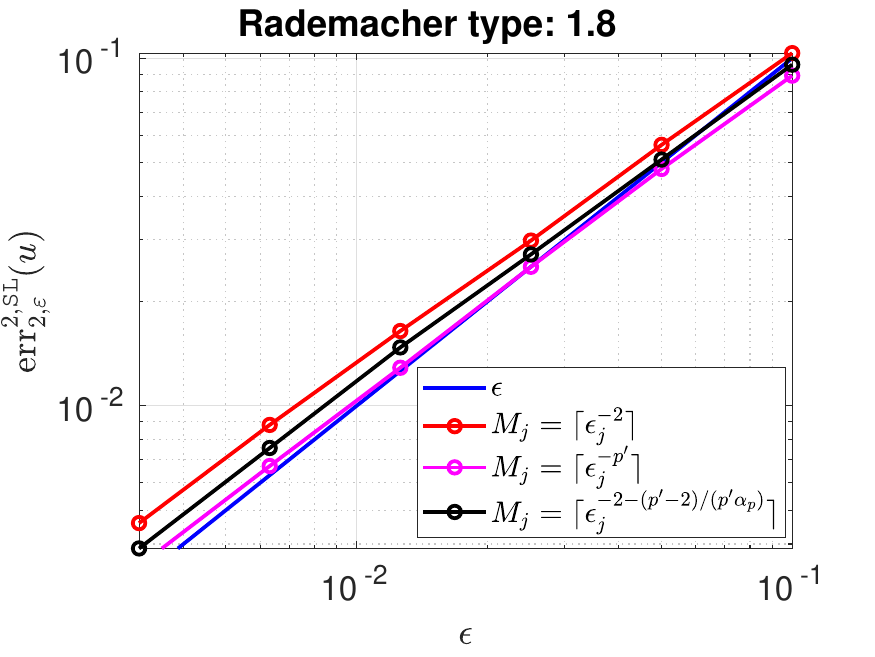}
\includegraphics[scale=0.28]{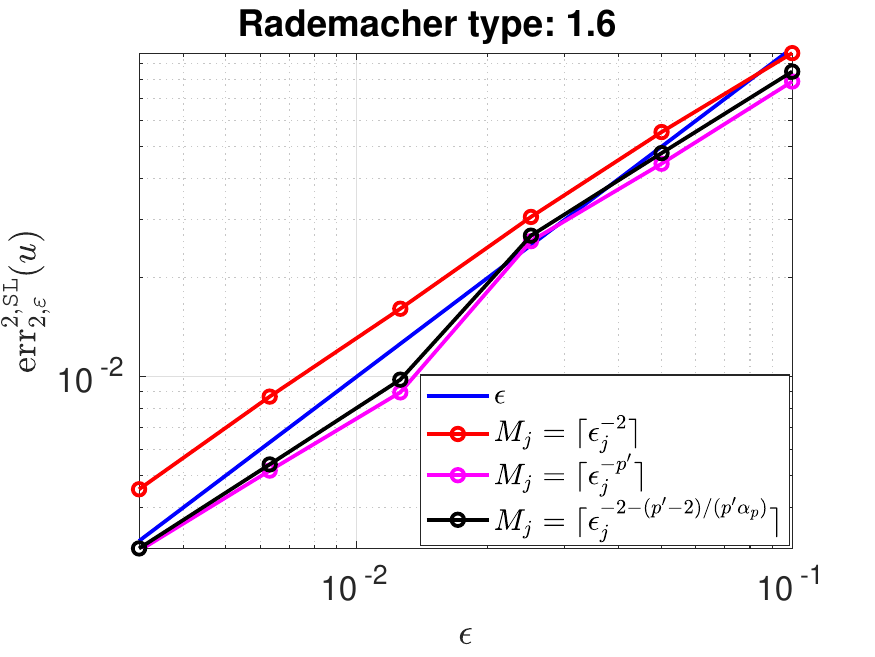}
\includegraphics[scale=0.28]{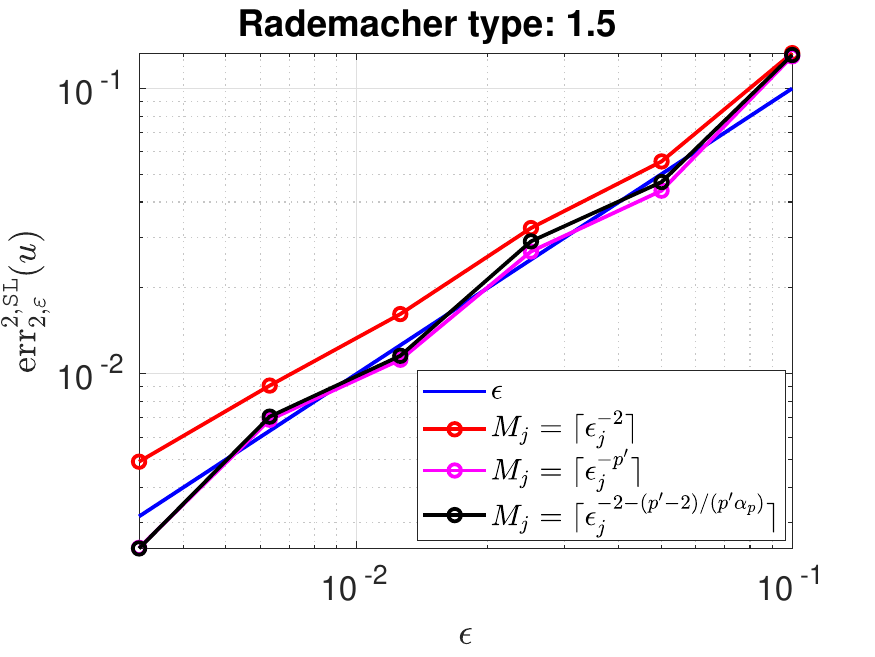}
\caption{
		Decay of the single-level Monte Carlo error 
		in the norm of 
		$L^2(\Omega; \otimes_\varepsilon^2 W^{1,p}_{ \{0\}}(I))$ 
		for the second moment of the solution to \eqref{eq:var_form}, 
		for different Rademacher types and choices of sample sizes.
		}\label{Fig:Decay_SLMC_RBVP_second_moment}
\end{figure}

\subsection{Function Approximation}\label{sec:num_FA}
The goal of this subsection is to verify numerically Corollaries \ref{Cor:Gen_MC_Lpspaces}
and \ref{Cor:MLMC_Optimal_Lpspaces}, 
concerning Monte Carlo convergence rates for $L^p$-valued 
random variables satisfying additional integrability properties. 

To this end, we let $(\Omega,\mathcal{A},\bbP)$ be a complete probability space, 
$I:=(0,1)$ and $y\from\Omega\rightarrow I$ 
be a uniformly distributed random variable over $I$. 
We consider the function-valued random variable 
\begin{equation}\label{eq:functiona_approximation}
	u\from\Omega\rightarrow C^0(\overline{I}), 
	\quad 
	u(\w):= (x+y(\w))^{-\eta},\quad\text{with } x\in I \text{ and }\eta\in (0,2).
\end{equation}
Despite being continuous on $\overline{I}$, $\bbP$-a.s., 
$u$ exhibits a singularity at $(0,0)$ 
if interpreted as the map $I\times I \ni (x,y)\mapsto |x+y|^{-\eta}$. 
The strength of this singularity determines the integrability of $u$. 
Direct calculations show that 
$u\in L^q(\Omega;L^p(I))\cap L^p(I;L^q(\Omega))$ 
provided that $\eta< \frac{1}{p}+\frac{1}{q}$, and that
\[ 
	\bbE[u] 
	= 
	\int_0^1 (x+y)^{-\eta}\rd y = 
	\frac{1}{1-\eta} 
	\bigl( (x+1)^{1-\eta} - x^{1-\eta} \bigr).
\]
The range $\eta\in (1,2)$ is of particular interest 
since then $u\notin L^2(\Omega;L^2(I))$ and $u\notin L^2(I;L^2(\Omega))$.

We first analyze the Monte Carlo convergence rates to approximate 
$\bbE[ u ]$ via samples of the map $\w\mapsto u(\w)$. 
Given a couple $(p,q)\in [1,2]\times (1,2]$, 
we impose sharply that 
$u\in L^q(\Omega;L^p(I))\cap L^p(I;L^q(\Omega))$ 
by setting $\eta=\frac{1}{p}+\frac{1}{q}-10^{-2}$. 
The rates are estimated as in Subsection~\ref{sec:num_rbvp}. 
Note, however, that due to the low integrability 
of the map $\w\mapsto \|u(\w)\|_{L^p(I)}$ 
with respect to the probability measure $\bbP$, 
the outer Monte Carlo average converges slowly 
and, hence, we set $K=10^4$. 
Table~\ref{tab:FA_CMC} verifies Corollary~\ref{Cor:Gen_MC_Lpspaces}, 
since the convergence rates are determined 
exclusively by the integrability parameter $q$, 
\textit{regardless} of the Rademacher type $p$.
Of particular interest is the case $p=1$, 
see Remark~\ref{remark:L1}.

\begin{table}
    \centering
    \begin{tabular}{|c|c|c|c|c|c|c|}
        \hline
        \diagbox{$p$}{$q$} & 1.1 & 1.3 & 1.5 & 1.7  & 2 \\
        \hline
         1 &  0.1 (0.09) & 0.23 (0.23) &  0.33 (0.33)  & 0.39 (0.41) &   0.52  (0.5)\\
		\hline         
         1.1  & 0.16 (0.09) & 0.28 (0.23) &  0.38 (0.33) &  0.44 (0.41) &  0.48 (0.5)\\
         \hline 
         1.3  & 0.09 (0.09) & 0.23 (0.23) & 0.33 (0.33) &  0.40 (0.41) &  0.46 (0.5)\\
         \hline
         1.5  &  0.04 (0.09) & 0.16 (0.23) &  0.28 (0.33) &  0.37 (0.41) & 0.44 (0.5)\\
         \hline
         1.7  &  0.12 (0.09) & 0.30 (0.23) & 0.35 (0.33)  &  0.42 (0.41) & 0.47 (0.5)\\
         \hline
         2    &  0.13 (0.09) & 0.23 (0.23) & 0.36 (0.33) & 0.43 (0.41) & 0.51 (0.5)\\
         \hline
    \end{tabular}
    \caption{ 
 	   For different values of $p$ and $q$, 
 	   the numerically estimated 
 	   Monte Carlo convergence rates 
 	   and the theoretical ones $1-\frac{1}{q}$ (in brackets) 
 	   in the norm of $L^q(\Omega;L^p(I))$
 	   for approximating the first moment of $u$ 
 	   in~\eqref{eq:functiona_approximation}.
 	   }\label{tab:FA_CMC}
\end{table}

In addition, we perform an additional numerical experiment 
where we set $\eta=1.1$. 
For every couple $(p,q)$ such that $\eta<\frac{1}{p}+\frac{1}{q}$, 
we are in the setting of Remark~\ref{remark:extra_integrability}, 
since $u\in L^{r}(\Omega;L^p(I))\cap L^p(I;L^{r}(\Omega))$ 
holds for any $r\in [1,\widehat{q})$, 
where we set  
${\widehat{q} :=
\sup\bigl\{s\in [q,2]: \frac{1}{s}>\eta-\frac{1}{p} \bigr\}}$.
Hence, we expect that the Monte Carlo convergence rates   
are determined by $\widehat{q}$, 
and this is numerically confirmed by Table~\ref{tab:FA_CMC2}.

\begin{table}
    \centering
    \begin{tabular}{|c|c|c|c|c|}
        \hline
        \diagbox{$p$}{$q$} & 1.1 & 1.3 & 1.5  \\
        \hline
         1 &  0.49 (0.50) & 0.50 (0.50) & 0.50 (0.50)   \\
		\hline         
         1.1  & 0.49 (0.50) & 0.49 (0.50) &  0.49 (0.50)\\
         \hline 
         1.3  & 0.49 (0.50) & 0.49 (0.50) & 0.49 (0.50) \\
         \hline
         1.5  &  0.47 (0.50) & 0.47 (0.50) & 0.48 (0.50) \\
         \hline
         1.7  &  0.44 (0.49) & 0.44 (0.49) & 0.45 (0.49)  \\
         \hline
         2    &  0.39 (0.40) & 0.39 (0.40) &  0.40 (0.40) \\
         \hline
    \end{tabular}
    \caption{
    For $\eta=1.1$ and different values of $p$ and $q$, 
    the numerically estimated 
    Monte Carlo convergence rates   
    and theoretical ones $1-\frac{1}{\widehat{q}}$ (in brackets) 
    in the norm of $L^q(\Omega;L^p(I))$ 
    for the approximation of the first moment 
    of $u$ in \eqref{eq:functiona_approximation}.}
    \label{tab:FA_CMC2} 
\end{table}

Next, let 
$S^0(I,\mathcal{T}_h)
:=
\bigl\{v\in L^2(I): v|_{I^j_h}\in \bbP_0,\; j\in \{0,\dots,N_h-1 \} \bigr\}$ 
be the space of piecewise constant functions over $\mathcal{T}_h$. 
For $\bbP$-a.e.~$\w$, we now approximate $u(\w)$ 
using the projection 
$\cI^0_h\from L^1(I)\rightarrow S^0(I,\mathcal{T}_h)$ defined by 
\[
	\bigl( \cI^0_h u(\w) \bigr)\big|_{I^j_h}
	:=
	\frac{1}{h}\int_{I^j_h} u(\w)\rd x
	=\frac{(x_{j+1}+y(\w))^{1-\eta}-(x_{j}+y(\w))^{1-\eta}}{(1-\eta)h}.
\]
Similarly to Subsection~\ref{sec:num_rbvp}, 
we first fit $\alpha_p$ and $C_{\alpha_p}$ such that the bias satisfies
\[
	\bigl\| \bbE[ u ] - \bbE\bigl[ \cI_h^0 u \bigr] \bigr\|_{L^p(I)} 
	= 
	\| \bbE[ u ] - \cI_h^0 \bbE[ u ]\|_{L^p(I)} 
	\leq C_{\alpha_p} h^{\alpha_p},
\]
where the equality holds due to the Fubini theorem. 
Note that, since $\bbE[u] \in W^{s,p}(I)$ 
for every $s\in \bigl[ 0, 1-\eta+\frac{1}{p}\bigr)$  
and $\cI^0_h$ satisfies
\begin{equation}\label{eq:conditions_interpolation_thm}
	\bigl\| \cI^0_h \bigr\|_{L^p(I)\rightarrow L^p(I)}\leq 1,
	\qquad 
	\bigl\| I-\cI^0_h \bigr\|_{W^{1,p}(I)\rightarrow L^p(I)}\leq h ,
\end{equation}
we expect a rate $\alpha_p=\min\bigl\{ 1,1-\eta+\frac{1}{p} \bigr\}$ 
due to the interpolation theorem, see e.g.\ \cite[Theorem 14.1.5]{brenner2007mathematical}, 
which is numerically confirmed. 

In Figure~\ref{Fig:MC_FA}, we report two numerical experiments 
for the single-level Monte Carlo method, 
that differ due to the imposed integrability of $u$. 
The fitted parameters $(\alpha_p,C_{\alpha_p})$ determine 
the partitions such that the associated bias are smaller 
than the tolerances $\{\epsilon_j \}_{j=1}^6$. 
We remark that scaling the sample sizes asymptotically according 
to the error bound of Corollary \ref{Cor:Gen_MC_Lpspaces} permits to meet the prescribed tolerances.

\begin{figure}
\includegraphics[scale=0.3]{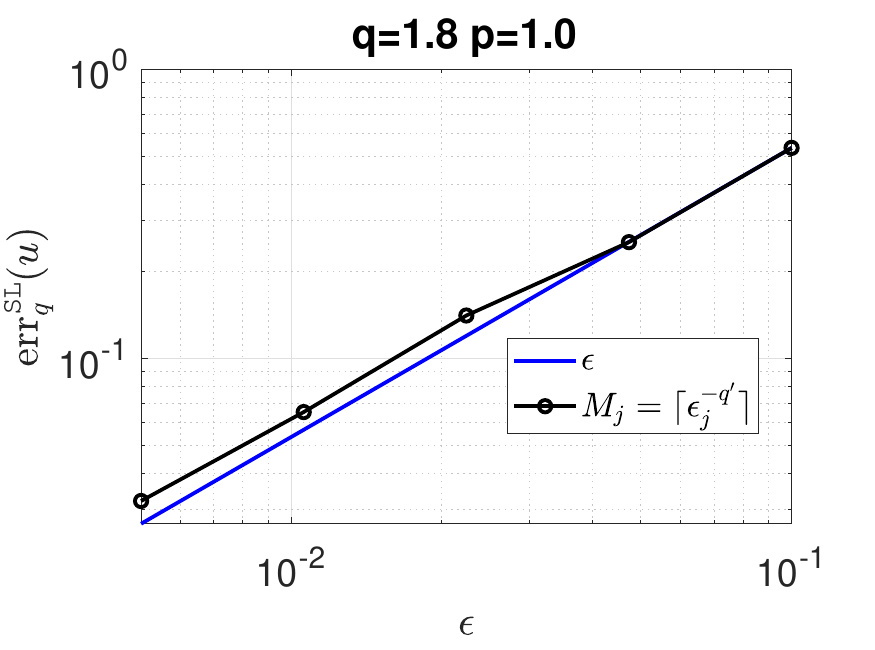}
\includegraphics[scale=0.3]{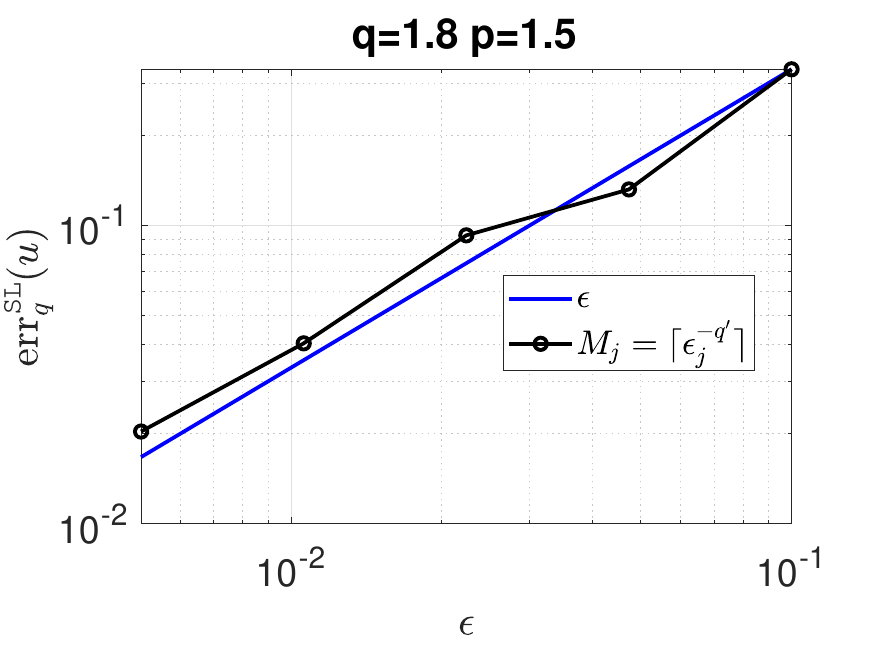}
\includegraphics[scale=0.3]{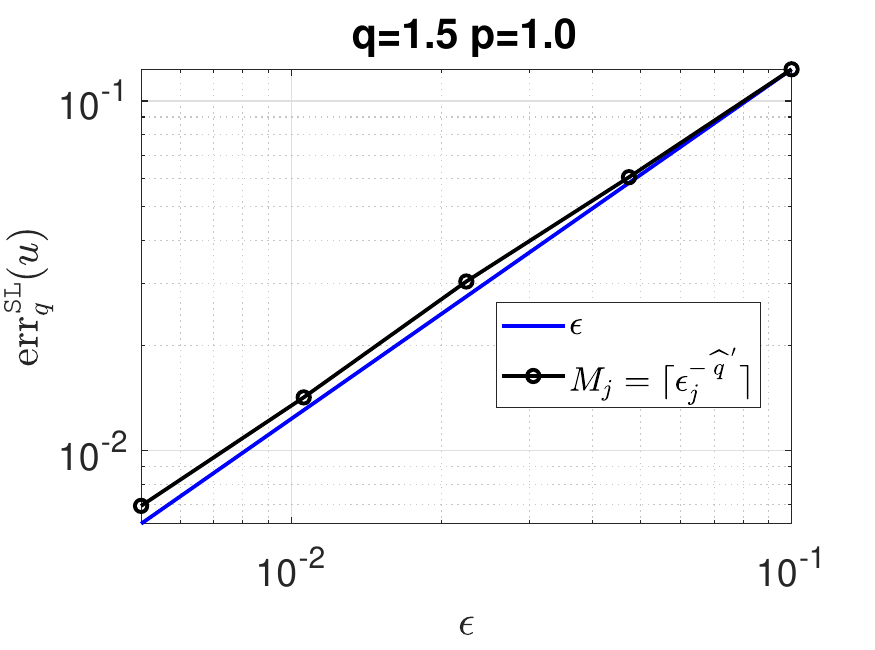}
\includegraphics[scale=0.3]{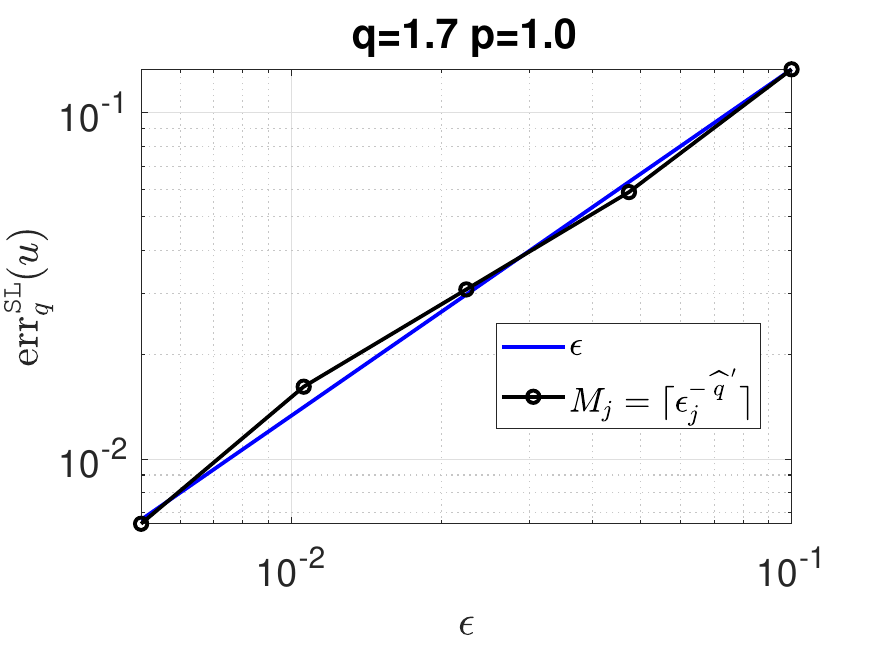}
\caption{ 
	Decay of the single-level Monte Carlo error 
	measured in the norm of 
	$L^q(\Omega;L^p(I))$ 
	for the first moment 
	of $u$ in \eqref{eq:functiona_approximation}.   
	The first row refers to $\eta=\frac{1}{q}+\frac{1}{p}-10^{-2}$, 
	while the second one to $\eta=1.1$.
		}\label{Fig:MC_FA}
\end{figure}

Next, we consider the multilevel Monte Carlo method aiming at investigating sharpness of
the complexity result of Corollary \ref{Cor:MLMC_Optimal_Lpspaces}.
The setting of the numerical experiments is as follows. 
We set $\eta=1.1$ and, for every tolerance $\epsilon_j$, we let 
$L_j=\Bigl\lceil \bigl|\frac{1}{\alpha_p}\log_2\bigl(\frac{\epsilon_j}{C_{\alpha_p}}\bigr)\bigr|\Bigr\rceil$, 
and introduce a nested sequence of uniform partitions $\mathcal{T}_{h_{\ell}}$, 
with $\ell\in \cI_j:=\left\{\ell_{\min},\dots,L_j\right\}$, $\ell_{\min}:=4$. 
Each partition $\mathcal{T}_{h_{\ell}}$ has 
${N_{\ell}:=\frac{1}{h_{\ell}}+1=2^{\ell}+1}$ points.
For a given couple $(p,q)$ such that $\eta<\frac{1}{p}+\frac{1}{q}$, 
we now discuss the values of the parameters 
$\alpha,\beta(r), \gamma$ of Theorem~\ref{thm:MLMC_Lpspaces}. 
As previously argued, 
we have that $\alpha_p=\min\bigl\{ 1,1-\eta+\frac{1}{p} \bigr\}$.  
Furthermore, as the computational cost of $\cI^0_h$ is linear 
with respect to $N_\ell$, we find that $\gamma=1$. 
Concerning the strong rate, it holds that $\beta_p(r)=\frac{1}{p}+\frac{1}{r}-\eta$,
hence $b_0=\frac{1}{p}-\eta$, $b_1=1$ and $t=t_p:=\frac{2}{\frac{1}{p}-\eta+1}$. 

Figure~\ref{Fig:MLMC} shows the error decay of the multilevel Monte Carlo method 
(measured in norm of $L^q(\Omega;L^p(I))$), 
the growth of the theoretical computational cost 
(measured as $\text{Cost}(\epsilon_j):=\sum_{\ell\in \cI_j} M_\ell\times N_{\ell}$) 
and the growth of the computational times (measured in seconds)
as the prescribed tolerances decrease, and for two couples $(p,q)$. 
In the first row, we set $q=1.5$, $p=1$ which leads to $\widehat{q}=2$ and $t_p\approx 2.2$. 
Since $\alpha_p=1-\eta+\frac{1}{p}<1$, it holds that $\alpha_p=b_0+b_1$ and we are in the third case of \eqref{eq:Corollary_cC_LqLp}, 
which states that the optimal $r$ is $r=\widehat{q}=2$. 
The computational cost of the multilevel Monte Carlo estimator should then grow as 
$\cC^{\sf ML}_{q}(X)
\sim 
\epsilon^{-\widehat{q}'-\frac{\gamma-\beta_p(\widehat{q})\widehat{q}'}{\alpha_p}}\! 
=
\epsilon^{-2.2}$ 
which is numerically verified. 
For comparison, we also include the error decay and the growth of computational cost and times, 
obtained by scaling the number of samples according to the non optimal choice with $r'=q'=3$. 
This would have been the natural choice prior to our refined analysis.
We observe a remarkable reduction in CPU time 
thanks to Corollary~\ref{Cor:MLMC_Optimal_Lpspaces}.
The second row reports the results obtained by setting $q=1.5$, $p=1.2$
which lead to $t_p\approx 2.72$. 
We are still in third case of \eqref{eq:Corollary_cC_LqLp}, 
so the optimal $r$ is $\widehat{q}=2$ and 
the optimal multilevel Monte Carlo estimator 
has a computational cost of order $\cC^{\sf ML}_{q}(X)\sim \epsilon^{-2.73}\!$, 
as shown in Figure~\ref{Fig:MLMC}.

\begin{figure}
\includegraphics[scale=0.275]{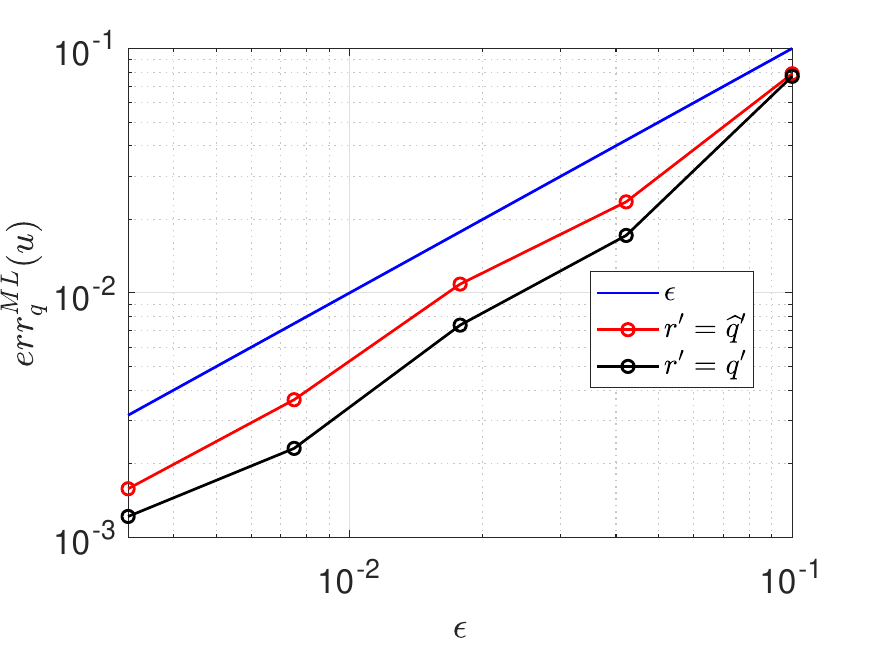}
\includegraphics[scale=0.275]{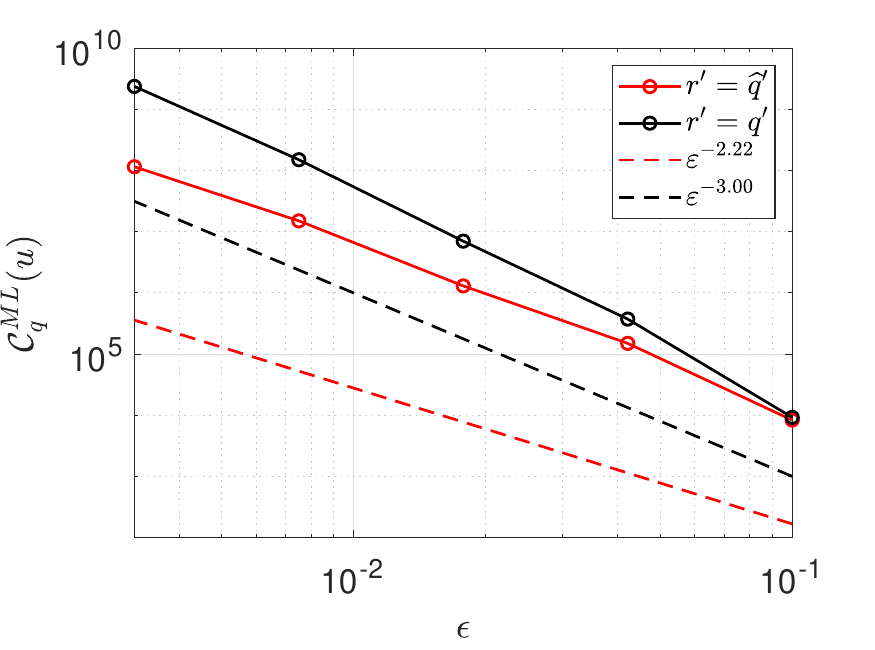}
\includegraphics[scale=0.275]{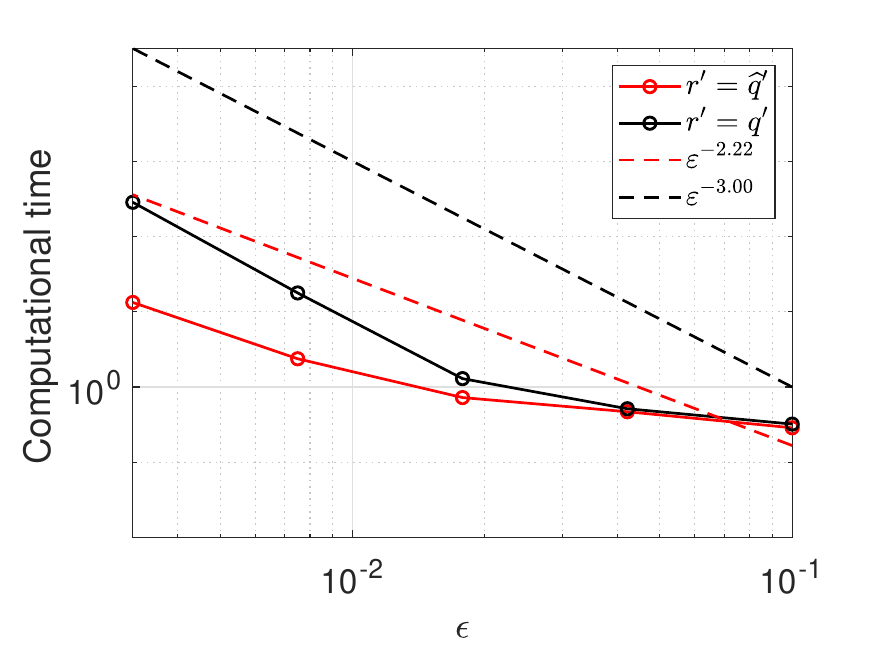}
\includegraphics[scale=0.275]{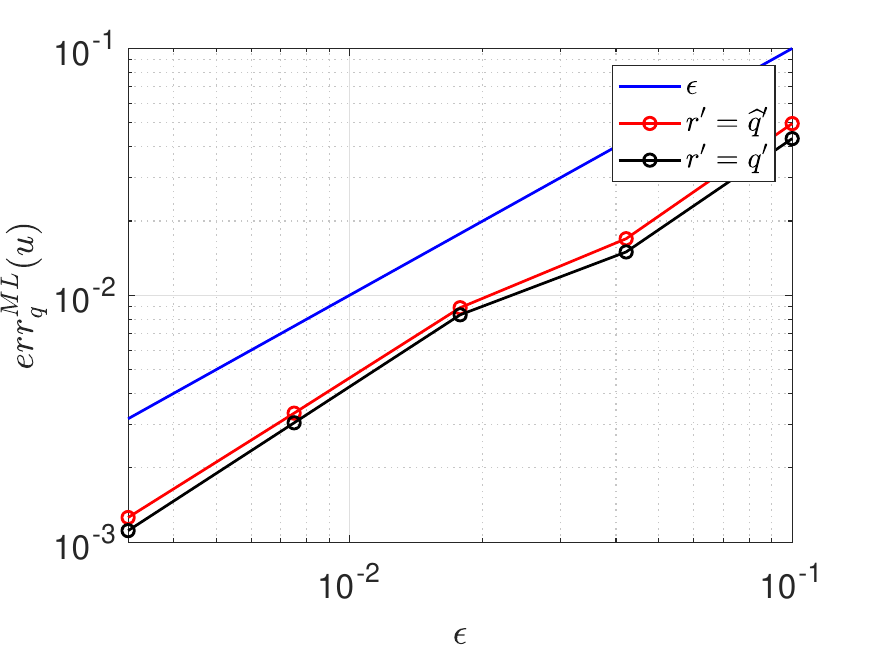}
\includegraphics[scale=0.275]{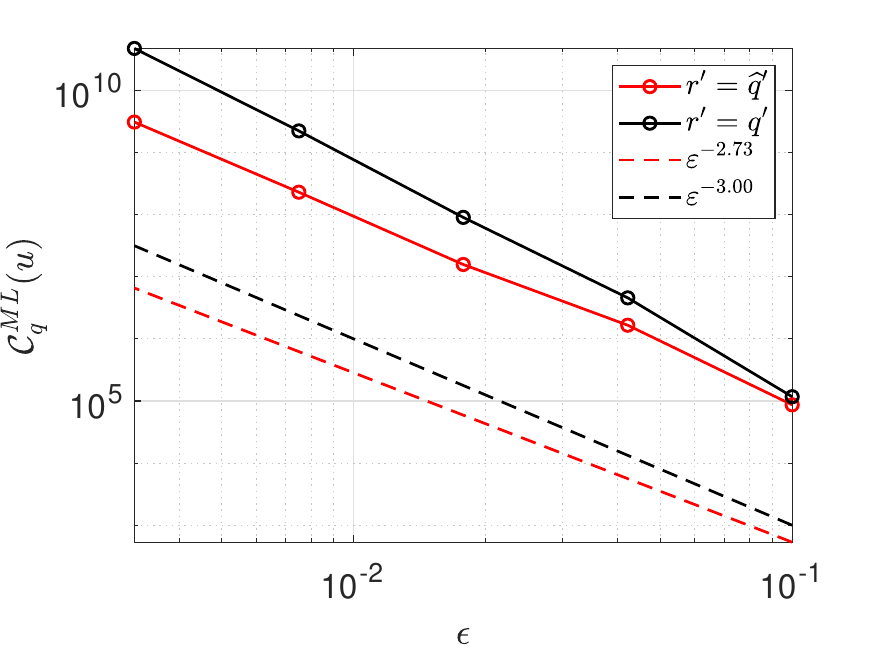}
\includegraphics[scale=0.275]{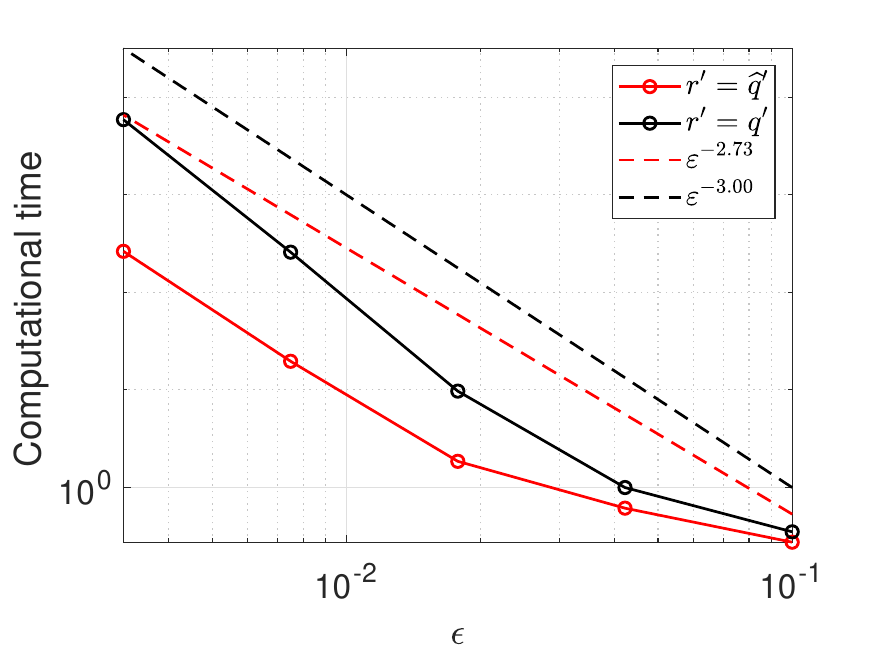}
\caption{Error decay and growth of the computational cost and computational time 
	for the multilevel Monte Carlo method to approximate the first moment 
	of $u$ in \eqref{eq:functiona_approximation} with $\eta=1.1$.}\label{Fig:MLMC}
\end{figure}

Finally, we consider the second moment $\bbM^2_p[u]$, 
defined in \eqref{eq:Lpkmoment}, as
\[
	\bbM^2_p [ u ]=\int_{\Omega} \otimes^2 u(\w)\rd\bbP(\w),
\]
whenever this integral exists (in Bochner sense) 
as an element of $\otimes_p^2 L^p(I)$. 
As discussed in Subsection~\ref{sec:prel_kmoments}, 
a sufficient condition is $u\in L^2(\Omega;L^p(I))$. 
The relations 
\[
	\phi\bigl(\bbM^2_p[ u ]\bigr)
	=
	\phi\bigl(\bbE[ u\otimes u ]\bigr)
	=
	\bbE\bigl[ \phi(u\otimes u)\bigr],
\]
where the expectation operator and $\phi$ 
(see Subsection \ref{sec:prel_kmoments}) can be exchanged 
due to the linearity and continuity of $\phi$ 
from $\otimes^2_p L^p(I)$ to $L^p(I\times I)$, 
allow us to identify $\bbM^2_p [ u ]$ with the $L^p(I\times I)$ function, 
which we still denote by $\bbM^2_p [ u ]$, 
\[
	\bbM_p^2[ u ](x,x') 
	= 
	\int_{\Omega} \bigl( (x+y(\w))(x'+y(\w)) \bigr)^{-\eta}\rd\bbP(\w).
\]
For the sake of numerical computations, 
from now we set $\eta=1$ which permits to derive, 
by direct calculations, the closed formula
\[
	\bbM_p^2 [u ](x,x')=
		\begin{cases} 
			\frac{1}{x'-x}\log\left(\frac{(x+1)x'}{(x'+1)x}\right) & \text{ if } x'\neq x, 
			\\
			\frac{1}{x}-\frac{1}{x+1} 
			& \text{ if } x'=x. 
		\end{cases}
\]
A careful analysis reveals that $\bbM_p^2 [u ]$ 
has a logarithmic singularity 
as $(x,y)$ approaches the axes $x=0$ or $y=0$, 
and 
a stronger singularity of order $\frac{1}{x}$ at the origin. 

Our first goal is to approximate $\bbM_p^2 [u ]$ 
via samples of the map 
\[
	\omega\rightarrow \frac{1}{(x+y(\omega))(x'+y(\omega))}.
\]
As emphasized in Remark~\ref{remark:gen_higher_moments}, 
the theory developed in Section~\ref{sec:Minkowski} 
characterizes the convergence of the Monte Carlo method 
both at the continuous level, and in its single- 
and multilevel variants to approximate $k$th moments. 
The convergence rates are determined by the integrability 
of $\otimes^2 u$ measured in norm $L^q(\Omega;L^p(I\times I))$ 
for $q\in (1,\infty)$ and $p\in [1,\infty)$. 
A direct calculation
shows that if $u\in L^{2q}(\Omega;L^p(I))$
then $\otimes^2 u\in L^q(\Omega;L^p(I\times I))$. 
Hence, for our model problem and for $\eta=1$, $\otimes^2 u \in L^q(\Omega;L^p(I\times I))$ 
provided that $\frac{1}{2q}+\frac{1}{p}>1$.
Table \ref{tab:FA2_CMC2} reports the estimated rates for the Monte Carlo error,
\begin{align*}
		\mathrm{err}^{2,\sf SL}_{q}(u)
		&=
		\biggl\|\bbM_p^2 [u ]-\frac{1}{M_j}\sum_{j=1}^{M_j} 
		 \otimes^2 u_j \biggr\|_{L^q(\Omega;L^p(I\times I))}
		 \\
		&\approx 
		\Biggl( \frac{1}{K}\sum_{k=1}^K \biggl\|\bbM_p^2 [u ]-\frac{1}{M_j}\sum_{j=1}^{M_j} 	
		 \otimes ^2 u(\w^k_j) \biggr\|^q_{L^p(I\times I)} \Biggr)^{\nicefrac{1}{q}} \!, 
\end{align*}
where $K=10^3$ and the sample sizes $\left\{M_j\right\}_{j=1}^6$ 
are equispaced on a logarithmic scale
between $10^2$ and $10^4$. 
To better capture the singularity at the origin, 
the $L^p(I\times I)$-norm is numerically evaluated with a two-dimensional, 
composite Gauss--Legendre quadrature formula on the rectangles 
$\widetilde{I}_h^j\times \widetilde{I}_h^i$, $i,j\in\left\{1,\dots,N_h-1\right\}$, 
where
$\widetilde{I}_{h}^j:=[\widetilde{x}_j,\widetilde{x}_{j+1}]$, 
with $\Bigl\{\widetilde{x}_k:= \bigl(\frac{k}{N_h}\bigr)^2\Bigr\}_{k=0}^{N_h}$. 
We denote by $\cTt_h$ the one-dimensional partitions 
induced by the points 
$\{\widetilde{x}_k \}_{k=0}^{N_h}$.

Table~\ref{tab:FA2_CMC2} confirms that the Monte Carlo convergence rates are determined by 
${ \widehat{q}:=\sup\bigl\{r\in[q,2]: \frac{1}{2r}+\frac{1}{p}>1\bigr\}}$ 
and not by the Rademacher type $p$, 
verifying Corollary~\ref{Cor:Gen_MC_Lpspaces} 
and Remark~\ref{remark:extra_integrability}, 
and in agreement with Table~\ref{tab:FA_CMC2} for the first moment.

\begin{table}
    \centering
    \begin{tabular}{|c|c|c|c|c|c|}
        \hline
        \diagbox{$p$}{$q$} & 1.1 & 1.3 & 1.5 & 1.7 & 2 \\
        \hline
         1 &  0.47 (0.50) & 0.47 (0.50) & 0.48 (0.50) & 0.46 (0.50) & 0.48 (0.50)\\
		\hline         
         1.3  & 0.40 (0.50) & 0.40 (0.50) &  0.42 (0.50)& 0.44 (0.50) & 0.42 (0.50)\\
         \hline 
         1.5  & 0.28 (0.33) & 0.42 (0.33) & - & -  &  -\\
         \hline
         1.7  &  0.21 (0.18) & - & - &  -  &  -\\
         \hline
    \end{tabular}
    \caption{
	    	For $\eta=1$ and different values of $p$ and $q$, 
	    	the numerically estimated 
	    	Monte Carlo convergence rates   
	    	and theoretical ones $1-\frac{1}{\widehat{q}}$ (in brackets) 
	    	in the norm of $L^q(\Omega;L^p(I\times I))$ 
	    	for the approximation of the second moment 
	    	of $u$ in \eqref{eq:functiona_approximation}.  
	    	The symbol ``$-$'' means that no rate was estimated 
	    	since $(p,q)$ violates the condition $\frac{1}{2q}+\frac{1}{p}>1$.
    		}\label{tab:FA2_CMC2}
\end{table}

Next, let $\cIt_h^0\from C^0(I)\rightarrow S^0(I,\cTt_h)$ 
be the piecewise constant interpolator such that, for any $v\in C^0(I)$,
\[
	\bigl( \cIt_h^0 v \bigr)\big|_{\cIt^j_h}
	:=
	v\left(\tfrac{\widetilde{x}_j+\widetilde{x}_{j+1}}{2}\right),
\]

and $\cIt_h^{2,0}\from C^0(I\times  I)\rightarrow S^{0,0}(I\times I,\cTt_h\times \cTt_h)$ 
be the tensorized interpolant such that, for any $v\in C^0(I\times I)$, 
\[
	\bigl( \cIt_h^{2,0} (v) \bigr)\big|_{\cIt^j_h\times \cIt^i_h} 
	:= 
	\bigl( (\cIt_h^{0} \otimes \cIt_h^{0}) v \bigr)\big|_{\cIt^j_h\times \cIt^i_h}
	=
	v\left(\tfrac{\widetilde{x}_j+\widetilde{x}_{j+1}}{2}, \tfrac{\widetilde{x}_i+\widetilde{x}_{i+1}}{2}\right).
\] 

Figure~\ref{Fig:MC_FA2} shows the decay of the single-level Monte Carlo error
to estimate $\bbM_p^2 [u ]$ via samples 
of the map $\omega\rightarrow  \cIt_h^{2,0}\bigl( \otimes^2 u(\w) \bigr)$. 
Once more, scaling the sample sizes according to the sharper error bound of 
Corollary~\ref{Cor:Gen_MC_Lpspaces} permits to have a linear decay 
of the error with respect to the decreasing tolerances.

\begin{figure}
\includegraphics[scale=0.28]{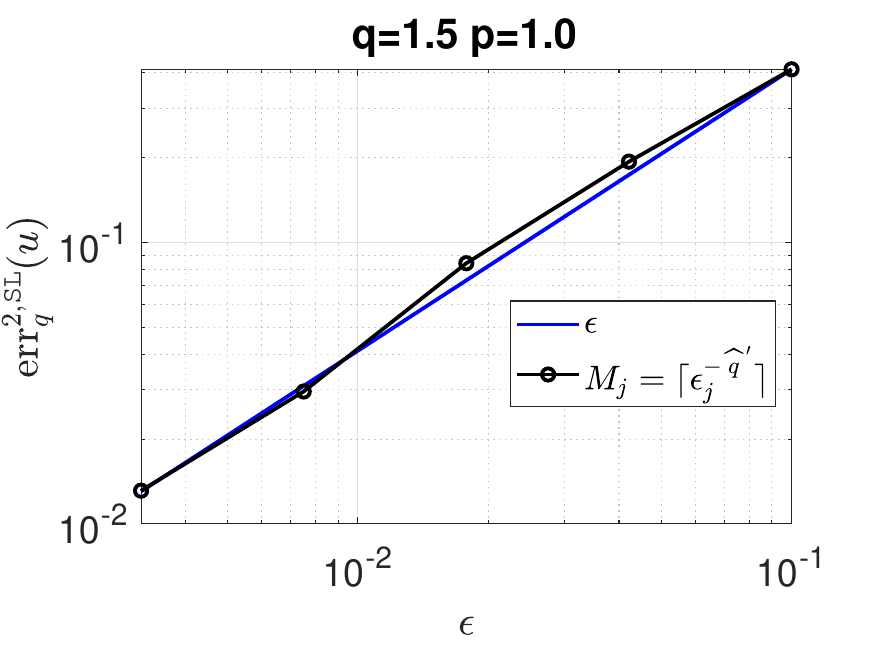}
\includegraphics[scale=0.28]{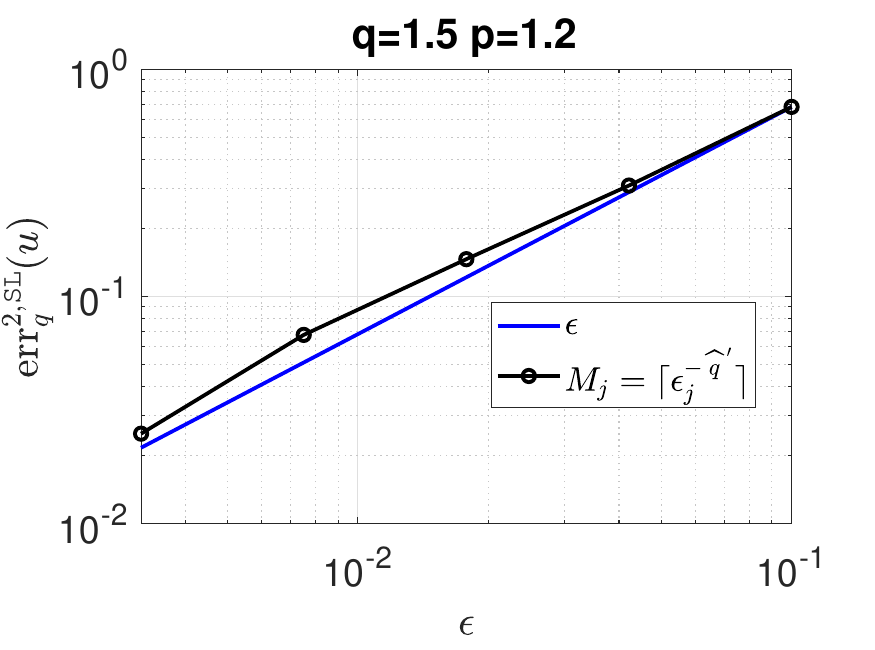}
\includegraphics[scale=0.28]{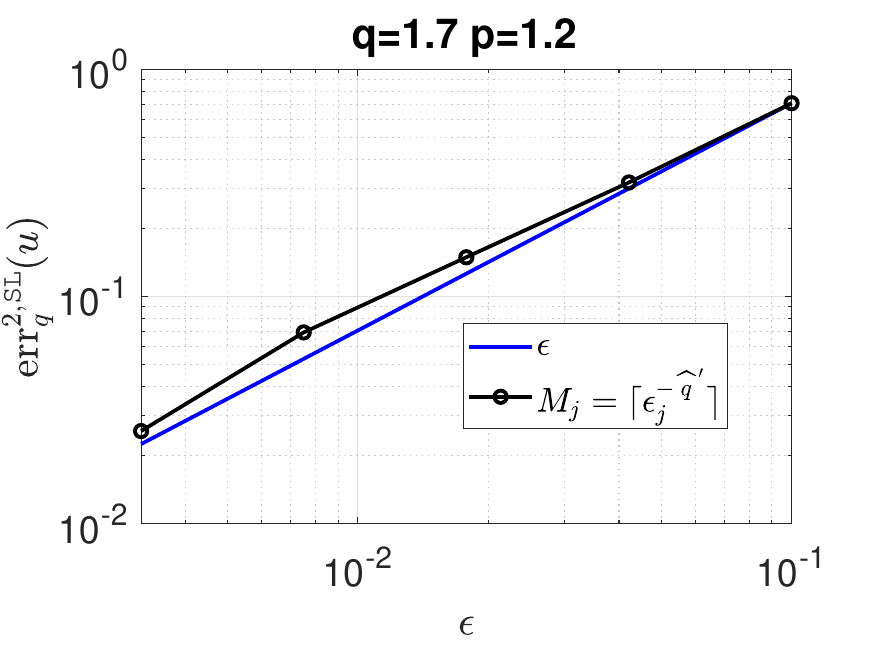}
\caption{
	Decay of the single-level Monte Carlo error 
	measured in the norm of 
	$L^q(\Omega;L^p(I\times I))$ 
	for the second moment 
	of $u$ in~\eqref{eq:functiona_approximation}.   
	The sample sizes are scaled according to $\widehat{q}$.}\label{Fig:MC_FA2}
\end{figure} 

To conclude, we present results for the multilevel Monte Carlo method. 
The setting of the numerical experiments is 
the same as described for the first moment. 
The parameters $\alpha,C_{\alpha},b_0,b_1$, 
and $\beta(r),C_{\beta}(r)$ for $r\in [\bar{q},\widehat{q}]$, 
are numerically estimated via least squares, 
and $\gamma=2$ since the cost per sample is dominated 
by the assembly of the tensor product. 
Figure~\ref{Fig:MLMC_FA2} shows the error decay 
as well as the growth of the theoretical computational cost 
and of the computational time for a sequence of decreasing tolerances, 
and for $q=1.5$ and $p=1$. 
We have $t=2.56$ and are in the third case of 
Corollary~\ref{Cor:MLMC_Optimal_Lpspaces}. 
Inserting the numerical estimated values 
of the parameters into \eqref{eq:Corollary_cC_LqLp}, 
we derive the asymptotic optimal computational cost of the multilevel estimator
$\mathcal{C}^{\sf ML}_q(X)\sim\epsilon^{-2.59}$, 
which improves over the non-optimized cost 
$\mathcal{C}^{\sf ML}_{q;q}(X)\sim\epsilon^{-3}$, 
obtained by using the non-optimal choice, but standard prior to our analysis, that is $r=q'=3$.

\begin{figure}
\includegraphics[scale=0.275]{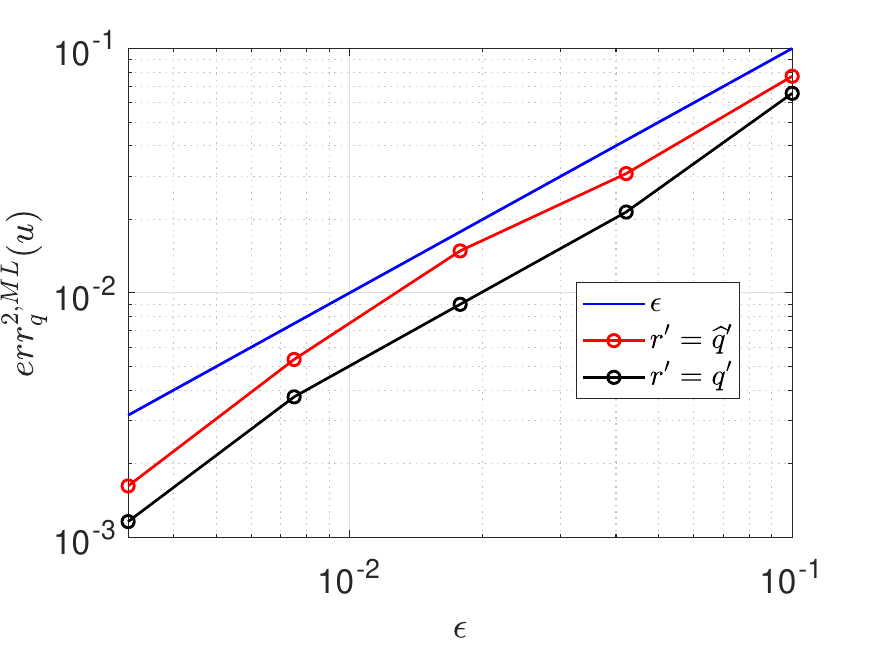}
\includegraphics[scale=0.275]{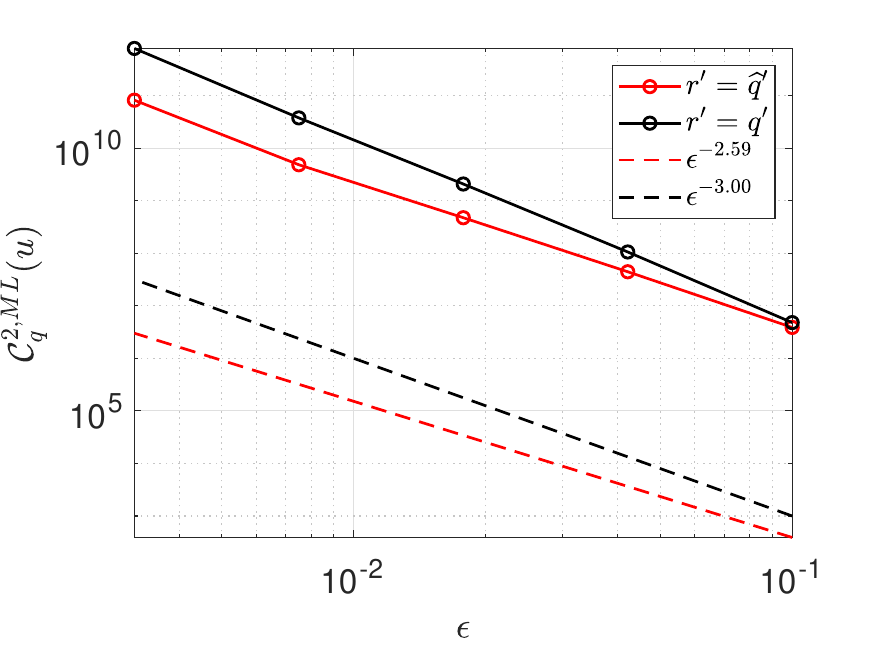}
\includegraphics[scale=0.275]{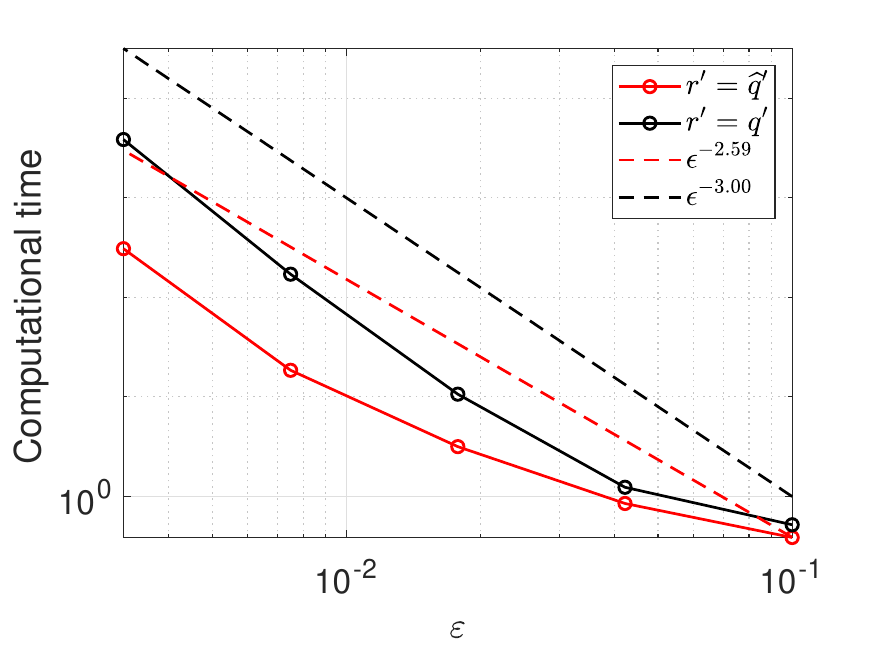}
\caption{
	Error decay and growth of the computational cost and computational time 
	for the multilevel Monte Carlo method to approximate the second moment 
	of $u$ in \eqref{eq:functiona_approximation} with $\eta=1$.} 
\label{Fig:MLMC_FA2}
\end{figure} 

\section{Conclusion}
\label{sec:Concl}
We investigated the convergence of single- and multilevel Monte Carlo methods 
for random variables with finite $k$th moments 
in Banach spaces $E$ of Rademacher type $p\in [1,2]$. 
We showed, theoretically and numerically, 
that Monte Carlo sample sizes to reach a prescribed target tolerance 
$\epsilon \in (0,1]$ must in general account for the Rademacher type of $E$.
In favorable cases, we have also shown that 
the theoretical convergence rate $1-\frac{1}{\min\left\{p,q\right\}}$ 
can be overly pessimistic for actual numerical computations 
due to the finite dimension of the approximation subspaces.
The present findings have implications for multilevel strategies 
in numerical statistical estimation 
for solutions of nonlinear PDEs 
with random data which are well-posed 
in non-Hilbertian Banach spaces
of Rademacher type $p\in [1,2)$ as, e.g., 
degenerate convection-diffusion equations \cite{koley2017multilevel} 
or compressible fluid flows \cite{lions1996mathematical}.

Future efforts are aimed at extending the present framework to anisotropic $k$-fold correlations 
of the form $\bbE[ X_1\otimes\cdots\otimes X_k]$, 
where $X_1,\dots,X_k$ may belong to different Banach spaces $E_1,\dots,E_k$ 
of different Rademacher types $p_1,\dots,p_k\in [1,2]$, and to the analysis of sparse estimators.

\appendix 

\section{Computation of the injective norm}\label{appendix:Injective_norm}
We here detail an algorithm to numerically approximate 
the injective tensor norm $\|\cdot\|_{\otimes_\varepsilon^2 W^{1,p}_{\left\{0\right\}}(I)}$, for $p\in (1,\infty)$. 
Recall that for a $U\in \tensorW$ with representation $U=\sum_{j=1}^M u_{j,1}\otimes u_{j,2}$, 
the injective tensor norm is defined as
\[
	\|U\|_{\tensorW}:=\sup\Biggl\{\biggl|\sum_{j=1}^M f_1(u_{j,1})f_2(u_{j,2})\biggr| : f_1,f_2 \in B_{\Wdual}\Biggr\},
\]
$B_{\Wdual}$ being the unit ball in $\Wdual$. 
Due to \cite[Proposition 8.14]{brezis2011}, 
for any $f\in \Wdual$ there exists an element $l\in L^{p'\!}(I)$ such that
\[
	f(u)=\int_I u'(x)l(x)\;\rd x=:l(u),\;\; \forall u\in W^{1,p}_{\left\{0\right\}}(I),
	\quad\text{and}\quad\|f\|_{\Wdual}=\|l\|_{L^{p'\!}(I)}.
\]
Hence, the injective tensor norm can be equivalently reformulated as
\begin{equation}\label{eq:equivalent_injective_norm}
	\begin{aligned}
		&\|U\|_{\tensorW}=\sup\Biggl\{\biggl|\sum_{j=1}^M l_1(u_{j,1})l_2(u_{j,2})\biggr| : 
		l_1,l_2 \in B_{{L^{p'\!}(I)}}\Biggr\}
		\\
		&=\sup\Biggl\{\biggl|\sum_{j=1}^M l_1(u_{j,1})l_2(u_{j,2})\biggr| : 
		l_1,l_2 \in {L^{p'\!}(I)} \text{ s.t.}\ \|l_1\|_{L^{p'\!}(I)}\leq 1,\;\|l_2\|_{L^{p'\!}(I)}\leq 1\Biggr\}.
	\end{aligned}
\end{equation}
Our procedure consists in a finite dimensional approximation of $l_1$ and $l_2$ 
in order to relate the computation of the supremum in \eqref{eq:equivalent_injective_norm} 
to a finite dimensional optimization problem. 
The latter is then solved using an alternating procedure.

Let $\mathcal{T}_h$ be a partition of $I$ (see Section \ref{sec:num_rbvp}), 
and denote the space of continuous, piecewise polynomials of order $k\geq 0$ by
\[
	S^k(I,\mathcal{T}_h)
	:=
	\bigl\{v\in C^0(\overline{I}):\;
	 v|_{I_h^j}\in \bbP_k,\;j\in \left\{0,\dots,N_h-1\right\}\bigr\}. 
\]
Although it is possible to consider any polynomial order, 
in our numerical experiments we set $k=1$. 
Furthermore, let $\Nht$ denote the dimension of $S^1(I,\mathcal{T}_h)$ 
and $\mathcal{B}:=\left\{\varphi_i\right\}_{i=1}^{\Nht}$ 
be the nodal basis for $S^1(I,\mathcal{T}_h)$. 
For any $l\in S^1(I,\mathcal{T}_h)$, ${\lb=(l_i)_{i=1}^{\Nht}\in \mathbb{R}^{\Nht}}$ 
is the vector expressing the coefficients of $l$ 
with respect to the basis $\mathcal{B}$. 

Next, we approximate the supremum in \eqref{eq:equivalent_injective_norm} by restricting $l_1$ and $l_2$ to $S^1(I,\mathcal{T}_h)$, leading to the maximization problem,  
\begin{equation}\label{eq:maximation}
\begin{aligned}
&\max_{l_1,l_2\in S^1(I,\mathcal{T}_h)} \biggl|\sum_{j=1}^M l_1(u_{j,1})l_2(u_{j,2})\biggr|\\
&\text{s.t.}\quad \|l_1\|_{L^{p'\!}(I)}\leq 1,\quad \|l_2\|_{L^{p'\!}(I)}\leq 1.
\end{aligned}
\end{equation}   
Problem \eqref{eq:maximation} is amenable to numerical solution, 
using numerical routines to approximate the norm on $L^{p'\!}(I)$,  
and observing that for any $l\in S^1(I,\mathcal{T}_h)$ and $u\in W^{1,p}_{\{0 \}}(I)$, 
$l(u)=\lb^\top \cb(u)$ where $\cb(u)_i= \int_I u^\prime \varphi_i$, 
$i\in\bigl\{1,\dots,\Nht \bigr\}$. 
The computational cost could however be prohibitive.
A favorable setting arises when $U=\sum_{j=1}^M u^h_{j,1}\otimes u^h_{j,2}$, 
with $u^h_{j,i}\in S^1_0(I,\mathcal{T}_h)\subset S^1(I,\mathcal{T}_h)$, 
$i=1,2$, $j\in \{1,\dots,M \}$, as it is the case in Subsection~\ref{sec:num_rbvp}. 
Indeed, for any $u\in S^1(I,\mathcal{T}_h)$ and $l\in S^1(I,\mathcal{T}_h)$, $l(u)=\lb^\top H \ub$, 
where $H\in \mathbb{R}^{\Nht\times \Nht}$ with $(H)_{i,j}:=\int_{I} \varphi_i\varphi^\prime_j$. 
Denoting with $U_M$ the matrix representation of $U$ into the tensorized basis $\mathcal{B}\otimes\mathcal{B}$, 
we then have
\[
	\sum_{j=1}^M l_1(u_{j,1})l_2(u_{j,2})=\lb_1^\top H U_M H^\top \lb_2.
\]
Next, for any $l\in S^1(I,\mathcal{T}_h)$ we approximate $\|l\|_{L^{p'\!}(I)}$ using a trapezoidal rule,
\[
	\|l\|^{p'}_{L^{p'}(I)}
	\approx \sum_{i=0}^{N_h-1} \frac{|l(x_i)|^{p'}+|l(x_{i+1})|^{p'\!}}{2}\, h
	=\sum_{i=1}^{\Nht} |l_i|^{p'} d_i^{p'} 
	= \|D \lb\|_{\ell^{p'}_{\Nht}}^{p'} 
	=:\| \lb\|^{p'}_{\ell^{p'}_{\Nht\!,D}} \!,
\] 
where $D=\operatorname{diag}(d_1,\dots,d_{\Nht})$, 
with $d_1=d_{\Nht}=\bigl(\tfrac{h}{2}\bigr)^{1/p'}$ 
and $d_i=h^{1/p'}\!$, for any $i\in\{ 2,\dots,\Nht-1\}$. 
Any quadrature formula can be used in what follows, 
provided that it leads to a positive diagonal matrix $D$.

We are thus led to consider the optimization problem
\begin{equation}\label{eq:maximation_2}
	\begin{aligned}
		&\max_{\lb_1,\lb_2\in \mathbb{R}^{\Nht}} |\lb_1^\top H U_M H^\top \lb_2|\\
		&\text{s.t.}\quad\| \lb_1\|_{\ell^{p'}_{\Nht,D}}\!\leq 1,
		\quad \| \lb_2\|_{\ell^{p'}_{\Nht,D}}\!\leq 1,
\end{aligned}
\end{equation}   
To solve \eqref{eq:maximation_2} we consider an iterative alternating procedure. 
Starting from two initial guesses $\lb_1^0,\lb_2^0$ satisfying the constraints, 
at each iteration $k\geq 1$ we solve separately,
\begin{equation}\label{eq:maximation_3}
	\begin{array}{cccc}
		&\max\limits_{\lb_1\in \mathbb{R}^{\Nht}} |\lb_1^\top \hb_2^{k-1}|\qquad 
		& \qquad 
		&\max\limits_{\lb_2\in \mathbb{R}^{\Nht}} |\lb_2^\top \hb_1^{k}|
		\\
		&\text{s.t.}\quad \| \lb_1\|_{\ell^{p'}_{\Nht,D}}\! \leq 1, &   
		&\text{s.t.}\quad \| \lb_2\|_{\ell^{p'}_{\Nht,D}}\! \leq 1,
	\end{array}
\end{equation}   
with $\hb_i^{k}:=H U_M H^\top \lb_i^k$, $i=1,2$, 
and $(\lb_1^k,\lb_2^k)$ denote the approximated maximizers computed at iteration $k$.
Each subproblem in \eqref{eq:maximation_3} consists in the maximization of the absolute value of a linear function over the unit ball of a weighted $\ell^{p'}_{\Nht,D}$ space. A close formula is available. 
Indeed, for a general problem 
\begin{gather*} 
		\max\limits_{\lb\in \mathbb{R}^{\Nht}} |\lb^\top \cb|\\
	    \text{s.t.}\quad \| \lb\|_{\ell^{p'}_{\Nht,D}}=1,  
\end{gather*}
we remark that  
\[
	|\lb^\top \cb|
	=
	\biggl|\sum_{i=1}^{\Nht} d_il_i\tfrac{1}{d_i}c_i\biggr|
	\leq 
	\|\lb\|_{\ell^{p'}_{\Nht,D}}\|\cb\|_{\ell^{p}_{\Nht,D^{-1}}}
	=\|\cb\|_{\ell^{p}_{\Nht,D^{-1}}},
\]
hence the objective functional is bounded from above by $\|\cb\|_{\ell^{p}_{\Nht,D^{-1}}}$. 
By direct construction, the vector $\lb^\star$ of components
\begin{equation}\label{eq:optimal_sol}
	l^\star_i
	=
	\frac{\operatorname{sign}(c_i)|c_i|^{p-1}}{d_i^p}
	\frac{1}{\left(\sum_{i=1}^{\Nht} d_i^{p'(1-p)} |c_i|^{p'(p-1)}\right)^{1/p'}},
\end{equation}
$i\in\{1,\dots,\Nht\}$, satisfies $\|\lb^\star\|_{\ell^{p'}_{\Nht,D}}=1$ 
and $(\lb^\star) ^\top \cb=\|\cb\|_{\ell^{p}_{\Nht,D^{-1}}}$, 
and thus is the sought maximizer. 
The overall procedure to approximate the injective norm is summarized in Algorithm 1.

\begin{algorithm}
\setlength{\columnwidth}{\linewidth}
\caption{Approximation of $\|U\|_{\otimes_\varepsilon^2 W^{1,p}_{\left\{0\right\}}(I)}$}
\begin{algorithmic}[1]\label{alg:Injective}
\State \textbf{Input}: $U_M$, $H$, $\lb_1^0$, $\lb_2^0$, max\_it, Tol.
\State $\lb_j^0=\lb_j^0/\|\lb^0_j\|_{\ell_{p',D}}$, $j=1,2$. (Normalize)
\State Set $k=1$, $\text{Err}=1$, $f^0_{\text{opt}}=(\lb_1^0)^\top H U_M H^\top \lb_2^0$.
\While{$k\leq \text{max\_it}$ and $\text{Err}>Tol$}
\State Compute $\lb_1^k$ solving the first subproblem in \eqref{eq:maximation_3} using \eqref{eq:optimal_sol}.
\State Compute $\lb_2^k$ solving the second subproblem in \eqref{eq:maximation_3} using \eqref{eq:optimal_sol}.
\State Compute $f^k_{\text{opt}}=|(\lb^k_1)^\top H U_M H^\top \lb^k_2|$.
\State Set $\text{Err}=|f^k_{\text{opt}}-f^{k-1}_{\text{opt}}|$, $k=k+1$.
\EndWhile
\State \textbf{Output}: $f^{k-1}_{\text{opt}}$.
\end{algorithmic}
\end{algorithm}

\section*{Funding}

KK acknowledges support of the research project 
\emph{Efficient spatiotemporal statistical
modelling with stochastic PDEs} 
(with project number VI.Veni.212.021) by the talent
program \emph{Veni} 
which is financed by the Dutch Research Council (NWO). 
TV has been partially supported by the INdAM-GNCS project {\em GNCS 2026 - CUP\_E53C25002010001}.
\section*{Acknowledgments} 

KK acknowledges helpful comments on 
Minkowski's integral inequality 
by Mark Veraar. 
TV is member of the INdAM-GNCS group.

\end{document}